%% file: main.tex
\RequirePackage{fix-cm}

\documentclass[smallextended,numbook,envcountsame]{svjour3}        

\smartqed  

\usepackage{graphicx}
\usepackage{amsmath,amssymb}
\usepackage{doi}
\usepackage{xurl}
\usepackage{hyperref}
\usepackage[capitalise]{cleveref}
\usepackage{mathtools}
\usepackage{dsfont}
\usepackage[ruled,vlined]{algorithm2e}
\usepackage{tikz}
\usepackage{standalone}
\usetikzlibrary{positioning,calc}
\usetikzlibrary{decorations.pathmorphing, arrows.meta}
\usetikzlibrary{shapes.geometric}
\tikzset{harp/.style={arrows = {[harpoon]<->[harpoon]}}}
\definecolor{p0Blue}{RGB}{0, 90, 169} 
\definecolor{p1Red}{RGB}{230, 0, 26}
\usepackage{pgfplots}
\pgfplotsset{compat=1.18}
\usepgfplotslibrary{fillbetween}

\newcommand{\N}{\mathds{N}}

\newcommand{\Q}{\mathds{Q}}
\newcommand{\R}{\mathds{R}}

\newcommand{\tl}{\triangleleft}
\newcommand{\tr}{\triangleright}
\newcommand{\ML}{\left\lfloor\frac{\left\lfloor\frac{m_i}{6\ell_i}\right\rfloor}{2}\right\rfloor}
\newcommand{\thmtext}[1]{#1}

\newcommand{\NP}{\ensuremath{\textup{\textsf{NP}}}\xspace}
\newcommand{\spmatrix}[1]{\bigl(\begin{smallmatrix}#1\end{smallmatrix}\bigr)}
\newcommand{\pspace}{\ensuremath{\textup{\textsf{PSPACE}}}\xspace}

\DeclareMathOperator{\LP}{LP}
\DeclareMathOperator{\val}{Val}
\DeclareMathOperator{\rc}{rc}
\DeclareMathOperator{\rew}{rew}

\DeclareMathOperator*{\argmax}{arg\,max}

\renewcommand{\vec}[1]{\ensuremath{\boldsymbol{#1}}}
\renewcommand{\epsilon}{\varepsilon}
\spnewtheorem{observation}[theorem]{Observation}{\bf}{\it}
\newcommand{\LIGadget}[9][]{ 
    \node[circle, draw=p0Blue!50, fill=p0Blue!20, thick, minimum size=2mm] (gamma) at ($(#2)!0.25!(#3)$) {#8};
    \node[rectangle, draw=p1Red!50, fill=p1Red!20, thick, minimum size=2mm] (alpha) at ($(#2)!#7!(#3)$) {};
    \node[circle, draw=p0Blue!50, fill=p0Blue!20, thick, minimum size=2mm] (beta) at ($(#2)!0.75!(#3)$) {#9};
    
    \draw (#2) edge [bend right=10,-{Latex[length=2mm]}] (gamma);
    \draw (gamma) edge [bend right=10,-{Latex[length=2mm]}] (#2);
    \draw (alpha) edge [p1Red,bend right=20,-{Latex[length=2mm]}] (#2);
\ifx\relax#1\relax 
    \draw (beta) edge [-{Latex[length=2mm]}] node [midway,auto] {#4} (#3);
    \draw (alpha) edge [p1Red,-{Latex[length=2mm]}] node [p1Red,midway,auto] {#5} (beta);
    \draw (gamma) edge [-{Latex[length=2mm]},auto] node {#6} (alpha);
\else
    \draw (beta) edge [-{Latex[length=2mm]}] node [midway,#1] {#4} (#3);
    \draw (alpha) edge [p1Red,-{Latex[length=2mm]}] node [p1Red,midway,#1] {#5} (beta);
    \draw (gamma) edge [-{Latex[length=2mm]},#1]  node {#6} (alpha);
\fi    
}

\usepackage{orcidlink}
\usepackage{fontawesome5}
\usepackage{xpatch}
\usepackage{hyperref}

\newcommand{\emaillink}[1]{%
  \href{mailto:#1}{\faEnvelope}%
}

\makeatletter
\xpatchcmd{\maketitle}
  {\footnotetext[0]{\kern-\bibindent
  \ignorespaces\@institute}\vspace{5dd}}
  {}
  {}{}
\makeatother
\begin{document}

\title{Lower bounds for ranking-based pivot rules\footnote{A preliminary version~\cite{disser_et_al:LIPIcs.STACS.2026.31} appeared in the proceedings of STACS 2026.}
}

\titlerunning{Lower bounds for ranking-based pivot rules}     

\author{%
Yann Disser
{\normalfont\small -- TU Darmstadt, Germany}\,
\orcidlink{0000-0002-2085-0454}\,
\emaillink{disser@mathematik.tu-darmstadt.de}
\\[7pt]
Georg Loho
{\normalfont\small -- FU Berlin, Germany \& University of Twente, Netherlands}\,
\orcidlink{0000-0001-6500-385X}\,
\emaillink{g.loho@utwente.nl}
\\[7pt]
Matthew Maat
{\normalfont\small -- University of Twente, Netherlands}\,
\orcidlink{0009-0004-7361-8538}\,
\emaillink{m.t.maat@utwente.nl}
\\[7pt]
Nils Mosis
{\normalfont\small -- TU Darmstadt, Germany}\,
\orcidlink{0000-0002-0692-0647}\,
\emaillink{mosis@mathematik.tu-darmstadt.de}
}

\authorrunning{Y. Disser, G. Loho, M. Maat, and N. Mosis}

\institute{}

\maketitle

\begin{abstract} 
The existence of a polynomial pivot rule for the simplex method for linear programming, policy iteration for Markov decision processes, and strategy improvement for parity games each are prominent open problems in their respective fields.
While numerous natural candidates for efficient rules have been eliminated, all existing lower bound constructions are tailored to individual or small sets of pivot rules.
We introduce a unified framework for formalizing classes of rules according to the information about the input that they rely on.
Within this framework, we show lower bounds for \emph{ranking-based} classes of rules that base their decisions on orderings of the improving pivot steps induced by the underlying data.
Our first result is a superpolynomial lower bound for strategy improvement, obtained via a family of sink parity games, which applies to memory-based generalizations of Bland's rule that only access the input by comparing the ranks of improving edges in some global order.
Our second result is a subexponential lower bound for policy iteration, obtained via a family of Markov decision processes, which applies to memoryless rules that only access the input by comparing improving actions according to their ranks in a global order, their reduced costs, and the associated improvements in objective value.
Both results carry over to the simplex method for linear programming.

\end{abstract}

\section{Introduction}
\label{sec:introduction}
The simplex method for linear programming is among the most well-studied algorithms in existence. 
To this day, the fundamental question of whether the method admits a (strongly) polynomial pivot rule remains one of the most important unsolved problems in mathematical optimization.
A positive answer would address Smale's 9th problem~\cite{smale1998mathematical} and resolve the polynomial Hirsch conjecture~\cite{dantzig1963linear,Santos:2012}.
The question of the complexity of the simplex method also surfaced in other settings that are now known to reduce to linear programming: the complexity of strategy improvement for parity games~\cite{voge2000discrete} and of policy iteration for Markov decision processes~\cite{puterman1994markov}.

Despite the importance of the question and despite numerous negative results for individual pivot rules~\cite{avis_notes_1978,avis_exponential_2017,disser_exponential_2023,friedmann2011randomfacet,friedmann_subexponential_2011,friedmann_errata_2014,GoldfarbSit:1979,jeroslow_simplex_1973,KleeMinty,murty1980computational}, a general lower bound \emph{for all} rules is currently very much out of reach. 
Both, a polynomial bound on the diameter of polytopes and a systematic understanding of pivot rules, remain elusive. 
Hence, almost all results to date have been tailored to single rules with little hope of generalization.
Exceptions include a recent bound for shadow vertex rules and steepest edge rules~\cite{black2025exponential}, as well as a bound for combinations of three classical rules~\cite{disser_unified_2023}. 
While simultaneously handling a continuum of pivot rules, these latter results are still rather interpolations between pivot rules than conceptual classes that capture a meaningful subset of all rules.

We propose a classification of pivot rules according to the information they use, and provide general lower bounds for a natural family of rules.
We hope that our treatment can be the basis for a systematic study of general classes of pivot rules in the future.

\textbf{Our results.}
We introduce a formal framework for defining classes of pivot rules in terms of the information on which they base their decisions.
We focus on \emph{ranking-based} classes of rules that rely only on orderings of the improving neighbors in each step according to some underlying data (e.g., objective value, global (Bland) index, rate of improvement, etc.), without otherwise using the underlying data.
This is a natural way of capturing ``combinatorial'' pivot rules and encompasses all classical deterministic rules.

Our first result is a superpolynomial lower bound on the running time of pivot rules that only use the relative ordering of improving edges/variables according to their global indices.
This generalizes Bland's rule, which always selects the improvement of lowest global index, to rules that use the local ranks of improving edges/variables in more elaborate ways.
In particular, our result extends beyond memoryless rules.
For example, it applies to an adaptation of Bland's rule that alternates between selecting the first and the second improving edge/variable.

\begin{theorem}[Informal version]\label{thm:blandmemoryInformal}
For every deterministic pivot rule that uses~$o(N/\log (N))$ memory states and bases decisions solely on the ordering of improving edges/variables by global index, the strategy improvement and the simplex algorithm take~$N^{\omega(1)}$ iterations in the worst-case, where $N$ is the input size.
\end{theorem}

The granted memory is in some sense not far from optimal, since, for every given game, there exists a pivot rule with~$N$ memory states that takes linear time (Observation~\ref{obs:n_memory}).
We prove this theorem by constructing a family of sink parity games and applying a direct connection between the strategy improvement algorithm for parity games and the simplex algorithm for linear programming. 
This continues a fruitful line of work originating in lower bounds for policy iteration for Markov decision processes and parity games that were transferred to LPs in different ways~\cite{avis_exponential_2017,disser_exponential_2023,disser_unified_2023,Fearnley:2010,Friedmann:2009,friedmann2011randomfacet,friedmann_subexponential_2011,maat2025strategyimprovementsimplexalgorithm,MelekopoglouCondon:1994}. 
A key idea in these constructions is to model a binary counter by a game and use gadgets to force the algorithm to increase the counter bit by bit. 

For our second result, we devise an adversarial construction for the policy iteration algorithm for Markov decision processes, and again use a well-known connection to the simplex algorithm for linear programming~\cite{Friedmann2011ExponentialPrograms}.
This allows us to prove a subexponential (i.e. of the form $2^{\Theta(n^c)}$ for some $0<c<1$) 
lower bound on pivot rules that simultaneously use rankings by global index (such as Bland's rule), by reduced cost (such as Dantzig's original rule), and by objective improvement (such as the largest-increase rule). 
As before, our result applies to rules using combinations of these orderings in elaborate ways.

\begin{theorem}[Informal version] \label{thm:rankingsSimplexInformal}
For every deterministic and memoryless pivot rule that bases decisions solely on orderings of improving actions/\allowbreak variables by global index, reduced cost, and objective improvement, the policy iteration and the simplex algorithm take~$\Omega(2^{\sqrt{N}})$ iterations in the worst-case, where $N$ is the input size.
\end{theorem}

Conceptually, this shows that the inefficiency of Bland's rule, Dantzig's rule, and the largest-increase rule arises not merely from their greedy nature but already from the fact that the information they rely on is too limited.

\textbf{Related work.} 
On the positive side, regarding the performance of the simplex method, variants of the shadow vertex pivot rule provide polynomial simplex algorithms in average-case~\cite{borgwardt_shadow-vertex_1987} and smoothed~\cite{dadush2018friendly,huiberts2023upper,spielman_smoothed_2001} analysis.
Further, linear programs can be solved in weakly-polynomial time via ellipsoid~\cite{khachiyan1980polynomial} and interior-point~\cite{karmarkar1984new} methods.
An adaptation of the latter is strongly polynomial if the constraint matrix satisfies some sparsity condition~\cite{allamigeon2025interior,dadush2024strongly}, while barrier methods cannot be strongly polynomial~\cite{allamigeon2022no}.

On the negative side, the simplex algorithm is \NP-mighty \cite{disser_simplex_2018}, and it is~\pspace-complete to predict the behavior of certain pivot rules~\cite{adler2014simplex,fearnley2015complexity}.
Further, it is \NP-hard to approximate the length of a shortest monotone path to the optimum~\cite{cardinal2025inapproximability,de2022pivot}.
Recently, it was proven that there is no polynomial pivot rule for the active-set method~\cite{disser2025unconditional}, which is a natural generalization of the simplex method to non-linear objectives, even if a convex quadratic is to be maximized~\cite{bach2025unconditionallowerboundactiveset}.
The proof of the latter uses \emph{deformed products}, which were introduced in~\cite{amenta1999deformed} to unify existing lower bound constructions for classical simplex pivot rules.

Interestingly, every pivot rule can be formulated as a \emph{normalized-weight} pivot rule if the associated normalization may depend on the input data~\cite{black2023polyhedral}. 
Further, every polyhedron allows for a similar polyhedron with small diameter~\cite{KaibelKukharenko:2024} such that a strongly polynomial simplex algorithm for polyhedra of small diameter would already resolve Smale's problem.
In the context of tropical linear programmimng~\cite{benchimol_tropical_2014}, it was shown that every strongly polynomial, \emph{combinatorial} simplex pivot rule would provide a strongly polynomial algorithm for mean payoff games~\cite{allamigeon2014combinatorial}.
Unique sink orientations are another combinatorial abstraction of linear programs~\cite{gartner2006linear,schurr2004finding,szabo2001unique}.

\section{Preliminaries}
\label{sec:preliminaries}

\textbf{The simplex method} solves linear programs (LPs) of the form
\begin{equation}\tag{$LP_{A,\vec b,\vec c}$}\label{eq:LP}
\begin{array}{rl}
   \max & \vec c^\top \vec x, \\
   \text{s.t.} & A\vec x = \vec b, \\
   & \vec x \geq \vec 0,
\end{array}
\end{equation}
where the matrix~$A\in \mathds{Q}^{m\times n}$ has full row-rank,~$\vec b\in \mathds{Q}^{m}$, and~$\vec c\in\mathds{Q}^n$.
We call a set~$B=\{b_1,\ldots,b_m\}\subseteq[n]\coloneqq\{1,2,\ldots, n\}$ a \emph{basis} if it consists of indices of~$m$ linearly independent columns of~$A$. 
We denote the (regular) matrix consisting of these columns by~$A_{\cdot B}$.
Given a basis~$B$, there is a unique~$\vec x\in\R^n$ such that~$A\vec x=\vec b$ and~$x_i=0$ for all~$i\in[n]\setminus B$.
If this~$\vec x$ is feasible for~\eqref{eq:LP}, i.e.,~$x_j\geq0$ for all~$j\in B$, we call it a \emph{basic feasible solution} (BFS) with basis~$B$, and~$B$ is a \emph{feasible basis}.
In this case,~$\vec x$ is a vertex of the feasible region~$\{\vec x\colon\, A\vec x=\vec b,\,\vec x\geq\vec 0\}$.

In this paper, we assume (without loss of generality) that the feasible region is non-empty, that the linear program is bounded, i.e., the optimum is attained at a vertex, and that the simplex method gets an initial feasible basis~$B$ as input.
Then, the algorithm selects an \emph{entering} index~$i\in[n]\setminus B$ and a \emph{leaving} index~$j\in B$ such that
\begin{itemize}
    \item~$\rc(i)\coloneqq c_i-\vec{c_B}^\top A_{\cdot B}^{-1}\vec{A_{\cdot i}} >0$,
    \item~$B'\coloneqq B\cup\{i\}\setminus\{j\}$ is a feasible basis,
\end{itemize}
and repeats with~$B\gets B'$, until reaching the optimum.
We call~$\rc(i)$ the \emph{reduced costs} of index~$i$ (or variable~$x_i$) and say that~$i$ (or~$x_i$) is an \emph{improving index} (or \emph{improving variable}) for~$B$ if~$\rc(i)>0$.
We denote the set of all improving indices for~$B$ by~$\Gamma^+(B)\coloneqq\{i\in[n]\setminus B\colon\, \rc(i)>0\}$. 
Since the choices of~$i$ and~$j$ are in general not unique, \emph{pivot rules} are used to determine the behavior of the algorithm, see Section~\ref{sec:pivotrules}.

In this paper, we mainly consider \emph{non-degenerate} linear programs, where we have a one-to-one correspondence between bases and vertices of the feasible region, such that, in each iteration of the simplex method, the choice of the leaving index is uniquely determined by the entering index.

\textbf{Sink parity games} are two-player games that are played on the vertices of a directed graph. 
A sink parity game (SPG) is given by a tuple~$(V_0,V_1,E,\pi)$, specifying a directed graph~$G=(V_0\cup V_1\cup \{\top\},E)$ and a \emph{priority} function~\mbox{$\pi:V_0\cup V_1\cup\{\top\}\to \mathds{Z}\cup\{-\infty\}$}.
We will often also refer to the sink parity game itself by~$G$. 
We generally assume that each vertex of~$G$ has at least one outgoing edge.
Further, the vertex~$\top$, called the \emph{sink}, has precisely one outgoing edge, which is a self-loop, and is the only vertex with priority~$\pi(\top)=-\infty$. 

The game is played between player 0 and player 1, where player~$i$ owns the vertices of~$V_i$. At the start of the game, a token is placed on an initial vertex~$v_0$. 
At step~$k$, the owner of the current node~$v_k$ moves the token along an outgoing edge to a successor node~$v_{k+1}$.
The game continues infinitely, so the token travels along an infinite path~$(v_0,v_1,v_2,\ldots)$.
The objective of player~0 is to maximize the outcome, given by $\sum_{i=0}^{\infty}(-t)^{\pi(v_i)}$, where~$t\geq \vert V_0\cup V_1\vert$ is a constant.
With this objective, the strategy improvement algorithm we will define later is equivalent to the version described in \cite{voge2000discrete}, see also \cite{fijalkow_games_2023} for background on parity games. 
Therefore, player~$0$ wants to collect large even priorties and avoid large odd priorities. 
Player~$1$ aims to minimize the outcome. 

It can be shown that the optimal strategy for both players is positional, meaning they always make the same choice in the same node. 
We thus consider a player~$i$ strategy as a function~$\sigma:V_i\to V_G$, where~$V_G\coloneqq V_0\cup V_1\cup\{\top\}$. So, if player~$i$ commits to a strategy~$\sigma_i$, the token can only move along the edges of~$E_{\sigma_i}:=\{(v,\sigma_i(v)):v\in V_i\}\cup \{(v,w):v\in V_{1-i}\}$.
We assume that every sink parity game admits
\begin{itemize}
    \item a player 0 strategy~$\sigma_0$ s.t.\ the highest priority on any cycle in~$E_{\sigma_0}$ is even,
    \item a player 1 strategy~$\sigma_1$ s.t.\ the highest priority on any cycle in~$E_{\sigma_1}$ is odd.
\end{itemize}
We call the strategies from these assumptions \emph{admissible} strategies. These two assumptions guarantee that, with optimal play, the token travels along a simple path towards~$\top$, which guarantees that the outcome is finite for every given constant~$t$. It also turns out that there always exists a strategy that is optimal for all starting vertices simultaneously~\cite{voge2000discrete}.
Hence, \emph{solving} the game usually means finding such an optimal strategy.

Let~$\sigma$ be an admissible player 0 strategy. 
Then, the optimal counterstrategy~$\bar{\sigma}$ for player 1 can easily be found by a shortest path computation.
If the players play according to~$\sigma$ and~$\bar{\sigma}$, the token will follow a simple path~$(v_0, v_1, \ldots, v_k,\top)$ that ends in the sink.
This defines a \emph{valuation} function~$\val_{\sigma}:V_G\to \mathds{N}^{\mathds{N}}$ by~$\val_{\sigma}(v_0)\coloneqq\{\pi(v_0),\pi(v_1),\ldots,\pi(v_k)\}$, for every~$v_0\in V_G$. 
We can compare these multisets by~$\tl$, where $S_1\tl S_2$ if and only if~$\sum_{s\in S_1}(-|V_0\cup V_1|)^{s}<\sum_{s\in S_2}(-|V_0\cup V_1|)^{s}$; in this summation, elements occurring multiple times are used multiple times.
We call an edge~$a=(v,w)\in E_0$ with~$\val_{\sigma}(w)\tr \val_{\sigma}(\sigma(v))$ an improving move or \emph{improving switch} for~$\sigma$.
If we \emph{apply} an improving switch~$a\in E_0$ to a strategy~$\sigma$, we obtain a new strategy~$\sigma^a$ by~$\sigma^a(v)= w$ and~$\sigma^a(v')=\sigma(v')$ for all~$v'\in V_0\setminus \{v\}$. 
The improving switch~$a$ increases the value of~$v$ without decreasing the value of any other state.
A strategy~$\sigma$ is optimal for a sink parity game if and only if there is no improving switch for~$\sigma$.

\textbf{Markov decision processes} provide a framework for modeling the deci\-sion-making of an agent operating in an uncertain environment. 
The environment is described by a finite set~$S$ of \emph{states} and a finite set~$A$ of \emph{actions} representing the feasible decisions.
Each action~$a\in A$ is associated with a \emph{reward}~$\rew(a)\in\R$ and a probability distribution~$P_a$ over~$S$.
In each state~$s \in S$, the agent may choose among a non-empty subset~$A_s\subseteq A$ of \emph{available} actions.
When an action~$a\in A_s$ is chosen, the agent receives the reward~$\rew(a)$, and the process \emph{transitions} to a state~$s'\in S$, which is drawn according to~$P_a$.
For each~$s'\in S$, we denote the probability that action~$a$ leads to~$s'$ with~$P_{a,s'}$ and call it a \emph{transition probability}.

A state~$s$ is \emph{reachable} from another state~$s'$ if there exists a sequence of actions such that the process transitions from~$s'$ to~$s$ with positive probability.
We call a state~$\top\in S$ \emph{sink} if it is reachable from all states and~$A_{\top}=\{a\}\subseteq A$ with~$\rew(a)=0$ and~$P_{a,\top}=1$.
Thus, once the agent reaches the sink, they cannot collect any further reward.

A \emph{policy} is a function~$\sigma\colon S\to A$ assigning to each state~$s \in S$ an action~$\sigma(s) \in A_s$. 
Given a policy~$\sigma$, the \emph{value} of each~$s\in S$, denoted by~$\val_\sigma(s)$, is the expected total reward obtained when the process starts in~$s$ and the agent follows~$\sigma$ in every state. 
In other words, the \emph{value} function~$\val_\sigma\colon S\to\R$ is determined by the following system of Bellman~\cite{bellman1966dynamic} equations
\begin{equation}\label{eq:Bellman}
\val_\sigma(s)=\rew(\sigma(s))+\sum_{s'\in S} P_{\sigma(s),s'}\val_\sigma(s'), \quad \forall s\in S,
\end{equation}
together with~$\val_\sigma(\top)=0$ if there is a sink~$\top$.
A policy~$\sigma$ is \emph{optimal} (with respect to the \emph{expected total reward criterion}) if~$\val_\sigma(s)\geq \val_{\sigma'}(s)$ for all~$s\in S$ and all policies~$\sigma'$.
See~\cite{puterman1994markov} for a discussion on other optimality criteria that consider discounted or average rewards.

A policy~$\sigma$ is \emph{weak unichain} if there is a sink~$\top$, and the agent will finally reach~$\top$ with probability one when following~$\sigma$ in every state, irrespective of the starting vertex. We call the process weak unichain if it admits an optimal policy that is weak unichain.
We refer to~\cite{Friedmann2011ExponentialPrograms} for more details on the weak unichain condition. 

An action~$a\in A_s$ is an \emph{improving switch} for policy~$\sigma$ if
\[
\rc_\sigma(a)\coloneqq \rew(a)+\sum_{s'\in S} P_{a,s'}\val_\sigma(s')-\val_\sigma(s)>0,
\]
where~$\rc_\sigma(a)$ is called \emph{reduced costs} of~$a$ (with respect to~$\sigma$).
If we \emph{apply} an improving switch~$a\in A_s$ to a policy~$\sigma$, we obtain a new policy~$\sigma^a$ by~$\sigma^a(s)= a$ and~$\sigma^a(s')=\sigma(s')$ for all~$s'\in S\setminus \{s\}$. 
The improving switch~$a$ increases the value of~$s$ without decreasing the value of any other state.
A policy~$\sigma$ is optimal for a weak unichain Markov decision process if and only if there is no improving switch for~$\sigma$.

\textbf{Strategy improvement} or \textbf{policy iteration} is an algorithmic idea that can be used to find optimal strategies and values for many types of games on graphs, see, e.g.,~\cite{fijalkow_games_2023}. 
To avoid confusion, we use the name \textsc{Strategy\-Improvement} when it is applied to sink parity games, and the name \textsc{Policy\-Iteration} when applied to weak unichain Markov decision processes.
The idea is simple: given an initial policy/strategy, iteratively apply improving switches until reaching an optimal solution, see Algorithm~\ref{alg:PI}. 

\begin{algorithm}
    \DontPrintSemicolon
    \caption{\textsc{PolicyIteration} or \textsc{StrategyImprovement}}
    \label{alg:PI}
    \vspace{0pt}
    \textbf{input: }$\sigma$ weak unichain policy (MDP) or admissible strategy (SPG)
    \vspace{2pt}
    \hrule  
    \vspace{0pt}
    \While{$\sigma$ admits an improving switch}
    {
       ~$a\gets \text{an improving switch for } \sigma$\;
        \vspace{0pt}
       ~$\sigma\gets \sigma^a$\;
    }
    \textbf{return}~$\sigma$
\end{algorithm}

In this paper, we only consider a version of this algorithm that applies a single switch in each iteration.
For a more general discussion, see~\cite{puterman1994markov,voge2000discrete}.
The algorithm will only visit admissible strategies/weak unichain policies, and is guaranteed to return the optimal strategy/policy after a finite number of iterations.
Whenever there are multiple improving switches, the choice in the algorithm is determined by an \emph{improvement rule}, analogously to pivot rules for the simplex method. 

It is well-known that, for every weak unichain Markov decision process~$\mathcal{M}$, there is a linear program~$\LP_{\mathcal{M}}$ such that the application of the simplex method to~$\LP_{\mathcal{M}}$ is equivalent to the application of \textsc{PolicyIteration} to~$\mathcal{M}$, in the following sense.

\begin{theorem}[\thmtext{\cite[Sec. 4.5]{Friedmann2011ExponentialPrograms}}]\label{thm:MDPreduction}
    Let~$\mathcal{M}$ be a weak unichain Markov decision process. 
    Then, there exists a non-degenerate, bounded linear program~$\LP_{\mathcal{M}}$ of the form (\ref{eq:LP}) such that
    \begin{itemize}
        \item there is a bijection between the actions of~$\mathcal{M}$ and the variables of~$\LP_{\mathcal{M}}$,
        \item there is a bijection between weak unichain policies for~$\mathcal{M}$ and basic feasible solutions of~$\LP_{\mathcal{M}}$, where variables in the basis correspond to actions used in the policy,
        \item given a weak unichain policy~$\sigma$ for~$\mathcal{M}$ and the corresponding basic feasible solution~$\vec x$ of~$\LP_{\mathcal{M}}$, the reduced costs of every action equal the reduced costs of its associated variable, and~$\vec c^\top \vec x=\sum_{s\in S}\val_{\sigma}(s)$.
    \end{itemize}
\end{theorem}

Recently,~\cite{maat2025strategyimprovementsimplexalgorithm} provided an analogous reduction for sink parity games, based on~\cite{schewe_parity_2009}.
\begin{theorem}[\thmtext{\cite[Thm. 3.4, Sec. 3.1]{maat2025strategyimprovementsimplexalgorithm}}]\label{thm:PGreduction}
    Let~$G$ be a sink parity game in which every vertex has a unique priority. 
    Then, there exists a non-degenerate, bounded linear program~$\LP_G$ of the form~\eqref{eq:LP} such that
    \begin{itemize}
        \item there is a bijection between the edges in~$(V_0\times V_G)\cup \{(\top,\top)\}$ and the variables of~$\LP_G$,
        \item there is a bijection between the admissible strategies for~$G$ and the basic feasible solutions of~$\LP_G$, where the variables in the basis correspond to the edges used in the strategy,
        \item given a basic feasible solution~$\vec x$ of~$\LP_G$ and the associated admissible strategy~$\sigma$ for~$G$, a variable is improving at~$\vec x$ if and only if the corresponding edge is improving for~$\sigma$.
    \end{itemize}
\end{theorem}

Conveniently, it is very simple to transform a sink parity game into an equivalent sink parity game with unique priorities by repeating the following: whenever there are two nodes with priority~$p$, we can keep one of them at priority~$p$ and increase the priority of all other nodes with priority at least~$p$ by 2. This has no effect on the run of strategy improvement, unless there are two strategies with the same value, in which case it may introduce new improving moves. We refer to this well-known transformation as the \emph{standard transformation}.

\section{Ranking-based pivot rules}
\label{sec:pivotrules}
Pivot rules specify, in each iteration of the simplex method, which index enters and which index leaves the basis. Since the inception of the simplex method by Dantzig in 1947~\cite{Dantzig82}, many pivot rules have been proposed.
Formally, a pivot rule is a function that gets the data of the linear program as well as the current basis as input, and returns a pair of an entering and a leaving index. Given an LP of the form~\eqref{eq:LP} with~$m$ equality constraints and~$n$ variables, its data can be written as elements of~$\mathcal{L}_{m,n}\coloneqq\{(A,\vec b,\vec c)\colon A\in \Q^{m\times n},\, \vec b\in\Q^m, \vec c\in\Q^n\}$. The basis is an element of $\mathcal{P}([n])$; here, we denote the powerset of a finite set~$S$ by~$\mathcal{P}(S) \coloneqq 2^S$.
Some pivot rules allow for ties between several options of switches, that is, instead of determining a unique pair of an entering and a leaving index, a pivot rule may suggest a set of admissible index pairs.
Generalizing further, a pivot rule may return a probability distribution on the set of sets of feasible index pairs, in which case we call the rule \emph{randomized}~-- otherwise, it is \emph{deterministic}. To this end, we denote the \emph{probability simplex} of a set~$S$ by~$\Delta(S)=\{p\in[0,1]^S\colon \sum_{s\in S}p(s)=1\}$.
Depending on whether the pivot rule has access to a finite memory of previous steps or not, we call it \emph{memory-based} or \emph{memoryless}, respectively. 
The number of available memory states may depend on the input size, so we denote it by~$H_{m,n}$.
The following definition gives a unified framework for pivot rules.

\begin{definition}\label[definition]{def:pivotrule}
A \emph{(simplex) pivot rule} that uses~$\left(H_{m,n}\right)_{n\in\N,\, m\in[n]}\in \N^{\N \times \N}$ memory states is a family~$\Pi=\left(\Pi_{m,n}\right)_{n\in\N,\, m\in[n]}$ of functions
\begin{align*}    
\Pi_{m,n}\colon\mathcal{L}_{m,n}\times\mathcal{P}([n])\times [H_{m,n}] & \to \Delta\!\left(\mathcal{P}([n]^2)\right)\times [H_{m,n}], \\
((A,\vec b,\vec c), B,h)&\mapsto (p,h'),
\end{align*}

such that, if~$B$ is a feasible non-optimal basis of a linear program~\eqref{eq:LP}, and, additionally, we have~$(i,j)\in S$ for some~$S\in\mathcal{P}([n]^2)$ with~$p(S)>0$, then~$i$ is an improving index for~$B$, and~$B\cup\{i\}\setminus\{j\}$ is a feasible basis.
\end{definition}

Note: we use $\mathcal{P}([n])$ for ease of notation. In reality, we are only concerned with the cases where the set $B$ is a feasible basis.
We interpret \cref{def:pivotrule} as follows. Suppose we apply the simplex method with pivot rule~$\Pi$, the algorithm is currently at basis~$B$, and the current memory state of~$\Pi$ is~$h$.
Then~$\Pi$ updates its memory state to~$h'$, we draw a set~$S\in\mathcal{P}([n]^2)$ according to~$p$, and we choose one pair~$(i,j)\in S$ of entering index~$i\in\Gamma^+(B)$ and leaving index~$j\in B$ yielding the new basis $B\cup\{i\}\backslash\{j\}$. 
Note that, if~$B$ is not a feasible basis for~\eqref{eq:LP}, then we do not pose any restriction on the output of~$\Pi$ (since this will never occur in practice).

\begin{example}
    The previous definition covers all classical pivot rules. For example, Zadeh's rule~\cite{Zadeh1980WhatAlgorithm} selects the improving index that was chosen least-often before.
    Thus, if it does not cycle, the rule uses at most~$H_{m,n}={\binom{n}{m}}^n$ memory states, allowing it to count the number of times each index was chosen in the past.
    It returns~$p\in \Delta\left(\mathcal{P}([n]^2)\right)$ with~$p(S)=1$ for
    \[
    S=\{(i,j)\in[n]^2\colon\, i\in\Gamma^+(B) \text{ chosen least-often}, B\cup\{i\}\setminus\{j\} \text{ feasible basis}\}.
    \]
\end{example}

Given a pivot rule~$\Pi=\left(\Pi_{m,n}\right)_{n\in\N,\, m\in[n]}$, we may write each~$\Pi_{m,n}$ as the composition of an \emph{information function}
\[
I^\Pi_{m,n}:\mathcal{L}_{m,n}\times\mathcal{P}([n])\to \Omega^\Pi_{m,n},
\]
where~$\Omega^\Pi_{m,n}$ can be any set, and a \emph{decision function}
\[
D^\Pi_{m,n}:\Omega^\Pi_{m,n}\times [H_{m,n}]\to\Delta\left(\mathcal{P}([n]^2)\right)\times [H_{m,n}]
\]
such that
\begin{equation}
    \Pi_{m,n}((A,\vec b,\vec c),B,h)=D^\Pi_{m,n}(I^\Pi_{m,n}((A,\vec b,\vec c),B),h) \label{eq:PivotRuleSplit}
\end{equation}
for all inputs.
There are many ways to decompose a given pivot rule~$\Pi$ into an information and a decision function, for example, one may always let~$I^\Pi_{m,n}$ be the identity function,~$\Omega^\Pi_{m,n}=\mathcal{L}_{m,n}\times\mathcal{P}([n])$, and~$D^\Pi_{m,n}=\Pi_{m,n}$.

The purpose of this decomposition is to give a natural notion of classes of pivot rules, based on the information they may use. 
We focus on classes of pivot rules that use combinatorial information, where~$\Omega^{\Pi}_{m,n}$ is a finite set.
We can, e.g., create such a class by fixing the set~$\Omega_{m,n}$, the information function~$I_{m,n}$, and the memory~$H_{m,n}$, while allowing any decision function~$D_{m,n}$.
Of course, the fewer restrictions we pose, the broader the resulting class of pivot rules.

\begin{example}
    Consider the class~$\mathcal{C}$ of \emph{combinatorial} pivot rules as defined in~\cite{allamigeon2014combinatorial}.
    That is, the rules in~$\mathcal{C}$ choose, at every step, the entering index based solely on the signs of all minors of the matrix~$M=\spmatrix{A \;\;\ \vec b \\
    \vec c^\top\ 0}$. 
    Thus~$\mathcal{C}$ is exactly the class of pivot rules of the form~\eqref{eq:PivotRuleSplit}, where~$\Omega_{m,n}$ denotes the finite set of all possible sign patterns of minors of~$M$, and~$I_{m,n}$ is the function that computes all these signs.
    
\end{example}

In this paper, we restrict ourselves to \emph{deterministic} pivot rules, which return deterministic distributions~$p\in \Delta\left(\mathcal{P}([n]^2)\right)$, that is,~$p(S)=1$ for a unique~$S\in \mathcal{P}([n]^2)$.
Further, we only focus on \emph{non-degenerate} linear programs, where the leaving index~$j$ is always uniquely determined by the entering index~$i$, for our lower bounds.
Therefore, we will use the simplified notation
\begin{align*}
    \Pi_{m,n}\colon\mathcal{L}_{m,n}\times\mathcal{P}([n])\times [H_{m,n}]&\to\mathcal{P}([n])\times [H_{m,n}], \\ ((A,\vec b,\vec c), B,h)&\mapsto (S,h'),
\end{align*}
in the following; here, if~$B$ is a feasible, non-optimal basis, then~$i\in S$ implies~$i\in\Gamma^+(B)$.
Traditionally, in the worst-case analysis of pivot rules, the \emph{tie-breaking} -- that is, choosing an element of~$S$ -- is up to the adversary, whose goal is to establish lower bounds, see e.g. \cite{disser_exponential_2023} for a discussion on this. Therefore, when proving lower bounds, we may assume that the pivot rule never outputs a tie, as we could fix a tie-breaking otherwise.
Then, we have~$\vert S\vert=1$ and simplify the notation even further by letting the codomain of~$\Pi_{m,n}$ be~$[n]\times [H_{m,n}]$.
If there is only a single memory state, i.e.~$H_{m,n}=1$ for all indices~$m$ and~$n$, we call the pivot rule \emph{memoryless}, in which case we will also leave out the memory component from the notation.

Most pivot rules from the literature select the \emph{best} improving variable(s) according to some underlying measure derived from the input data. Such a measure naturally induces a total preorder\footnote{A binary relation over a set~$S$ is a \emph{preorder} if it is reflexive and transitive.
A preorder over~$S$ is \emph{total} on~$S'\subseteq S$ if any two elements of~$S'$ are comparable. Thus, a total preorder is a total order without the requirement of antisymmetry, allowing for ties between distinct elements.} on the set of improving variables. In our framework, this happens when the information function~$I$ has a special structure, in which case we refer to~$I$ as a \emph{neighbor ranking}.
Given a finite set~$S$, we denote the set of preorders over~$S$ by~$\mathcal{T}(S)$. 
    
\begin{definition}\label[definition]{def:neighbor_ranking}
    A \emph{neighbor ranking} is a family~$(R_{m,n})_{n\in \N,m\in [n]}$ of functions 
    \begin{align*}
    R_{m,n}\colon\ \mathcal{L}_{m,n}\times\mathcal{P}([n])\ &\to\ \mathcal{T}([n]), \\
    ((A,\vec b,\vec c), B)\ &\mapsto\ \preceq,    
    \end{align*}
    such that, if~$B$ is a feasible basis for~\eqref{eq:LP}, then~$\preceq$ is total on~$\Gamma^+(B)$, and every non-improving index~$i\in[n]\setminus\Gamma^+(B)$ is incomparable (wrt.\ $\preceq$) to any~$j\in[n]\setminus\{i\}$.
\end{definition}

Every neighbor ranking allows for a \emph{greedy} pivot rule.

\begin{definition}\label[definition]{def:greedy}
Given a neighbor ranking~$R=(R_{m,n})_{n\in \N,m\in [n]}$, a deterministic, memory\-less pivot rule~$\Pi=(\Pi_{m,n})_{n\in \N,m\in [n]}$ is \emph{greedy (for~$R$)} if each~$\Pi_{m,n}$ is of the form~\eqref{eq:PivotRuleSplit} with
\begin{itemize}
    \item~$I^\Pi_{m,n}=R_{m,n}$,
    \item~$D^\Pi_{m,n}$ returns the set~$\{i\in \Gamma^+(B)\colon\, i\succeq i',\,\forall i'\in \Gamma^+(B)\}$ of maximal elements with respect to~$\preceq\coloneqq R_{m,n}((A,\vec b,\vec c),B)$.
\end{itemize}
\end{definition}

\begin{observation}\label[observation]{obs:greedy}
    Given a feasible basis~$B$ for a non-degenerate~\eqref{eq:LP}, let~$B[i]$ denote the unique feasible basis obtained by adding index~$i$ to~$B$ and removing another index, and let~$\vec x$ and~$\vec{x[i]}$ be the basic feasible solutions with bases~$B$ and~$B[i]$, respectively. The following classical pivot rules are greedy:
    \begin{itemize}
      \item \textup{\textsc{Bland}}~\cite{bland1977new} orders by index, selecting~$\min(\Gamma^{+}(B))$.
      \item \textup{\textsc{Dantzig}}~\cite{dantzig1963linear} orders by reduced cost, selecting
      \[
      \argmax_{i\in\Gamma^+(B)}\quad c_i-\vec {c_B}^\top A_{\cdot B}^{-1}\vec{A_{\cdot i}}.
      \]
      \item \textup{\textsc{LargestIncrease}}~\cite{jeroslow_simplex_1973} orders by objective value, selecting
      \[
      \argmax_{i\in\Gamma^+(B)}\quad \vec{c_{B[i]}}^\top \vec{x[i]}.
      \]
      \item \textup{\textsc{ShadowVertex}}\cite{murty1980computational} orders by the ratio of shadow objective decrease to objective increase, selecting
      \[
      \argmax_{i\in\Gamma^+(B)}\quad \frac{\vec d^T(\vec{x[i]}-\vec x)}{\vec c^T(\vec{x[i]}-\vec x)},
      \] 
      where~$\vec d\in\R^n$ is the shadow objective\footnote{ The shadow vertex rule is greedy under the assumption that $\vec d$ is fixed beforehand, instead of depending on the initial basis (which would require some memory). This assumption has no effect on worst-case analysis.}.
      \item \textup{\textsc{SteepestEdge}}~\cite{GoldfarbSit:1979} orders by the ratio of objective increase to length of the connecting edge, selecting
      \[
      \argmax_{i\in\Gamma^+(B)}\quad \frac{\vec c^T(\vec{x[i]}-\vec x)}{\vert\vert \vec{x[i]}-\vec x\vert\vert}.
      \] 
    \end{itemize}
\end{observation}

Now we are ready to introduce the class of \emph{ranking-based} pivot rules, which generalizes these classical rules by relaxing their greedy behavior and, additionally, allowing the pivot rules to depend on a set~$\mathcal{R} = \{ R^1, \dots, R^k \}$ of neighbor rankings~$R^i$.
That is, a ranking-based rule should base its decision solely on the relative positions of the improving variables in the preorders~$R^i_{m,n}((A,\vec b,\vec c),B)$.
However, by definition, a pivot rule outputs improving indices, so it must have access to these indices as well. With this information, it could, say, treat variable~$x_5$ differently from~$x_8$, or even specify its decision for every basis. To prevent the rule from abusing this knowledge, we add an additional symmetry condition, forcing the rule to only base its decision on the relative positions in the rankings. This is done by ensuring that whenever the rule encounters two equivalent situations with equivalent rankings, it is forced to make equivalent choices.

\begin{definition}\label[definition]{def:rankingbased}
    Let~$\mathcal{R} = \{ R^1, \dots, R^k \}$ be a fixed set of neighbor rankings. 
    A deterministic pivot rule~$\Pi=\left(\Pi_{m,n}\right)_{n\in\N,\, m\in[n]}$ is called \emph{$\mathcal{R}$-ranking-based} if
    \begin{itemize}
        \item each~$\Pi_{m,n}$ is of the form~\eqref{eq:PivotRuleSplit} with~$\Omega^\Pi_{m,n}=\mathcal{T}([n])^k$, and for all inputs we have 
        \begin{eqnarray*}
        I^\Pi_{m,n}((A,\vec b,\vec c),B)&=&(R^1_{m,n}((A,\vec b,\vec c),B),\ldots, R^k_{m,n}((A,\vec b,\vec c),B))\\
        &\eqqcolon& (\preceq^{R_1}_B,\ldots,\preceq^{R_k}_B),
        \end{eqnarray*}
    \item\label{enum:cond2} for every pair of linear programs $L\in\mathcal{L}_{m,n}$ and $L'\in\mathcal{L}_{m',n'}$, for all feasible non-optimal bases~$B$ of~$L$ and~$B'$ of~$L'$, and for every memory state~$h\leq \min\{H_{m,n},H_{m',n'}\}$, if the outputs are~$\Pi_{m,n}(L,B,h)=(S,h_1)$ and~$\Pi_{m,n}(L',B',h)=(S',h_1')$, and if there exists a ranking-preserving bijection~$\varphi\,\colon\,\Gamma^+(B)\to \Gamma^+(B')$, i.e.,
    \[
    i \preceq^{R_j}_{B} i' \iff \varphi(i)\preceq^{R_j}_{B'}\varphi(i'), \quad \forall\,i,i'\in\Gamma^+(B),~ \forall\, j\in\{1,2,\ldots,k\},
    \]
    then~$\varphi(S)=S'$; additionally, if $n=n'$, then also~$h_1=h_1'$.
    \end{itemize}
\end{definition}
Every deterministic memoryless pivot rule can be formulated as a ranking-based rule: choose the ranking such that the first improving variables in the ranking are the ones preferred by the pivot rule. Then, choose as the decision function the function that always picks the first variables. This gives an equivalent ranking-based rule.

Recall that the information function, which specifies a tuple of neighbor rankings in the class of ranking-based rules, is always memoryless, while the decision function may be memory-based.
As an example, let \textsc{Bland} denote the neighbor ranking that orders by index, then a~$\{\textsc{Bland}\}$-ranking-based pivot rule~$\Pi$ may select the improving variable with the second smallest index.
If~$\Pi$ is also memory-based it may, e.g., alternatingly select the improving variable with the smallest and the largest index.
We analyze the class of memory-based~$\{\textsc{Bland}\}$-ranking-based rules, which we call \emph{index-based} rules for short, in Section~\ref{sec:memory}.

\section{Lower bounds for index-based pivot rules}
\label{sec:memory}
Recall that Bland's rule always picks the improving variable that has the lowest index. In this section, we consider~$\{\textsc{Bland}\}$-ranking-based rules with memory and without ties. 
Recalling \cref{def:rankingbased}, these rules take into consideration only the order of the indices of the improving variables, and the current memory state. The purpose of this section is to show that~$\{\textsc{Bland}\}$-ranking-based rules need a superpolynomial number of steps, even when equipped with some memory. We can let the memory only depend on~$n$ without losing generality, since we always have~$m\leq n$, so the number of memory states is always bounded by~$\max_{m\in [n]}H_{m,n}$ if~$n$ is constant. We denote the allowed memory of a pivot rule~$\Pi$ by~$H_{m,n}=\ell(n)$. 

Suppose that for some LP with~$100$ variables, there is a feasible basis~$B$ from which there are~$20$ improving moves, and that a~$\{\textsc{Bland}\}$-ranking-based rule picks the improving move with the 12\textsuperscript{th} lowest index for this basis when the memory state is~$5$. The second condition (symmetry requirement) in \cref{def:rankingbased} then forces this rule to always pick the 12\textsuperscript{th}-lowest index whenever there are 100 variables, 20 improving moves, and the memory state is~5. For this reason, the choices of the pivot rule are determined by only three parameters: the number~$k$ of improving variables, the total number~$n$ of variables, and the memory state~$h$.
This means we can write~$\Pi$ in terms of a function~$P \colon \N^3 \to \N^2$, whose inputs represent these three parameters and whose outputs represent the chosen index and the memory update, so~$P(k,n,h)\in\{1,2,\ldots,k\}\times\{1,2\ldots,\ell(n)\}$, for all~$k,n,h\in \N^3$. Then~$\Pi_{m,n}((A,\vec b,\vec c), B,h)$ is given as follows: if~$P(|\Gamma^{+}(B)|,n,h)=(i',h')$, and~$i$ is the~$i'$-th smallest index in~$\Gamma^{+}(B)$, then~$\Pi_{m,n}((A,\vec b,\vec c), B,h)=(i,h')$. We refer to a function~$P$ of this form as an~$\ell$-index-selector function, and we say~$P$ \emph{induces} the~$\ell$-index-based pivot rule~$\Pi^P$ in this manner. The function~$P$ does not need~$m$ as input: for different~$m$ the output must still be the same by the second condition of \cref{def:rankingbased}.
 Like the so-called \emph{normalized-weight} rules in~\cite[Cor. 1.4]{black2023polyhedral}, index-based rules are universal in the following sense.
\begin{observation}\label[observation]{obs:n_memory}
    Given some~\eqref{eq:LP}, if there exists a pivot rule that\break reaches the optimum in~$\ell(n)$ iterations, then there also exists an~$\ell$-index-based rule reaching the optimum in~$\ell(n)$ iterations. 
    In particular, any parity game with~$n$ player 0 states can be solved in linear time by some~$n$-index-based improvement rule.
\end{observation}
\begin{proof}
    Let~$(B_0,B_1,\ldots,B_{\ell(n)})$ be the sequence of bases of~\eqref{eq:LP} visited by the pivot rule, where~$B_{\ell(n)}$ denotes the optimal basis.
    Then, we can choose a memory state~$h_i$ for every~$i\in\{0,1,\ldots,\ell(n)-1\}$ such that, if the current memory state is~$h_i$ and the current basis is~$B_i$, the index-based rule transitions to the basis~$B_{i+1}$ and updates the memory state to~$h_{i+1}$. 
    For parity games, the optimum can always be reached in a linear number of improving switches, given perfect choices of the improvement rule \cite[Lem. 4.2]{Friedmann2011ExponentialPrograms}. 
    \qed
\end{proof}
 
Recalling that the polyhedra with the largest known diameters are linear in the dimension, this illustrates the power of this class of pivot rules.
This contrasts in some sense with the following theorem, where the rest of this section is dedicated to its proof.

\begin{theorem}\label{thm:blandmemory}
Let~$\ell(n)=o\left(n/\log (n)\right)$. For every~$\ell$-index-based rule~$\Pi^{P}$, there exists a family~$(\LP_n)_{n\in\N}$ of linear programs~$\LP_n$ in~$n$ variables and at most~$n$ constraints such that the simplex method with pivot rule~$\Pi^{P}$ requires~$n^{\omega(1)}$ iterations to find the optimal basis.
\end{theorem}

Because of the combinatorial nature of index-based pivot rules, they can also be interpreted as improvement rules for strategy iteration in parity games. Then, the input~$(p,m,h_1)$ of an~$\ell$-index-selector function~$P$ consists of the number~$p$ of improving edges available at the current strategy, the number~$m$ of player 0 controlled edges, and the current memory state~$h_1$. Note that here, and in the remainder of this section,~$m$ denotes the number of player 0 edges in a sink parity game, which corresponds by \cref{thm:PGreduction} to the number of variables in its related LP (which was denoted by~$n$ before). 
This allows one to infer \cref{thm:blandmemory} from the following statement for parity games.

\begin{theorem} \label{thm:PGlordering}
Suppose~$\ell(m)=o\left(\frac{m}{\log (m)}\right)$. Then, for any~$\ell$-index-based improvement rule~$\Pi^P$, there is a family~$(\mathcal{G}_m)_{m\in\N}$ of sink parity games~$\mathcal{G}_m$ with~$m$ player 0 edges, such that strategy improvement with pivot rule~$\Pi^{P}$ requires~$m^{\omega(1)}$ iterations to find the optimal strategy.
\end{theorem}

Since the number~$N$ of non-zero input values is polynomial in~$n$ in \cref{thm:blandmemory}, and polynomial in~$m$ in \cref{thm:PGlordering}, the algorithms also take~$N^{\omega(1)}$ iterations, implying \cref{thm:blandmemoryInformal}. We construct these parity games in several steps, starting from special cases and gradually introducing further gadgets to cover the general setting. We start with the basic case of Bland's rule. This well-known pivot rule in LPs picks the improving variable with the lowest index to enter the basis. To use this rule in a sink parity game, we need to assign each player 0 edge a \emph{Bland number}, indicating its index.
Using the notation for index-based improvement rules, Bland's rule is~$\Pi^P$, where~$P$ is a selector function such that the first element of its output is always~$1$, independent of the input.

The games that we use to prove \cref{thm:PGlordering} all have the same base structure. For every~$n\in\N$, we define a sink parity game~$G_n$, which is shown in Figure~\ref{fig:BjorklundCounter}.

\begin{lemma}
\label[lemma]{lem:alternate}
Strategy improvement with Bland's rule takes~$2^n-1$ iterations to solve the sink parity game~$G_n$ with initial strategy~$\sigma_0$.
Moreover, the order of the valuations of the nodes~$a_1$ and~$b_1$ changes in every iteration. 
\end{lemma}

\begin{figure}[ht] 
    \centering
    \resizebox{\linewidth}{!}{%
    \begin{tikzpicture}[node distance = 3cm,
    p0/.style={circle, draw=p0Blue!60, fill=p0Blue!40, thick, minimum size=15mm,inner sep=0pt, outer sep=0pt}, 
    every path/.style={draw=p0Blue,line width=1.3pt,{>=Latex}}, 
      p1/.style={rectangle, draw=p1Red!60, fill=p1Red!40, thick, minimum size=15mm,inner sep=0pt, outer sep=0pt}, 
      phan/.style={draw=none, fill=none, inner sep=0pt, outer sep=0pt}] 

    \node[p0] (a1) at (0,0) {$\begin{array}{c}a_1 \\ 3 \end{array}$};
    \node[p1] (b1) at (0,4) {$\begin{array}{c}b_1 \\ 4 \end{array}$};
    \node[p0] (a2) at (4,0) {$\begin{array}{c}a_2 \\ 5 \end{array}$};
    \node[p1] (b2) at (4,4) {$\begin{array}{c}b_2 \\ 6 \end{array}$};
    \node[phan] (phan1) at (8,0) {}; 
    \node[phan] (phan2) at (8,4) {}; 
    \node[phan] (phan3) at (10,0) {}; 
    \node[phan] (phan4) at (10,4) {}; 
    \node[p0] (a5) at (14,0) {$\begin{array}{c}a_n \\ 2n+1 \end{array}$};
    \node[p1] (b5) at (14,4) {$\begin{array}{c}b_n \\ 2n+2 \end{array}$};
    \node[p0] (a6) at (18,0) {$\top$};
    \node[p1] (b6) at (18,4) {$\begin{array}{c}b_{n+1} \\ 2n+4 \end{array}$};
    \draw (a1) edge[->, line width=3] node [pos=.3,below] {$1$} (a2);
    \draw (a1) edge[->] node [pos=.2,right] {$2$} (b2);
    \draw (b1) edge[->,p1Red] (b2);
    \draw (b1) edge[->,p1Red] (a2);
    \draw (a2) edge[->, line width=3] node [pos=.3,below] {$3$} ($(a2)!0.8!(phan1)$);
    \draw (a2) edge[->] node [pos=.2,right] {$4$} ($(a2)!0.8!(phan2)$);
    \draw (b2) edge[->,p1Red] ($(b2)!0.8!(phan1)$);
    \draw (b2) edge[->,p1Red] ($(b2)!0.8!(phan2)$);

    \draw (phan1) edge[-, black, dotted] (phan3);
    \draw (phan2) edge[-, black, dotted] (phan4);

    \draw ($(phan3)!0.2!(a5)$) edge[->, line width=3] (a5);
    \draw ($(phan3)!0.2!(b5)$) edge[->] (b5);
    \draw ($(phan4)!0.2!(a5)$) edge[->, p1Red] (a5);
    \draw ($(phan4)!0.2!(b5)$) edge[->, p1Red] (b5);
    \draw (a5) edge[->, line width=3] node [pos=.3,below] {$2n-1$} (a6);
    \draw (a5) edge[->] node [pos=.2,right] {$2n$} (b6);
    \draw (b5) edge[->,p1Red] (b6);
    \draw (b5) edge[->,p1Red] (a6);
    \draw (b6) edge[->,p1Red] (a6);
    \draw (a6) edge[loop right, line width=3] (a6);
    
\end{tikzpicture}}
    \caption{Binary counter parity game~$G_n$, adapted from \cite{bjorklund_combinatorial_2007}. Circles represent player 0 nodes and the sink, and squares represent player 1 nodes. Priorities are shown below the names of the nodes. The initial strategy~$\sigma_0$ is shown in bold, and the Bland numbers are given on the edges.}
    \label{fig:BjorklundCounter}
\end{figure}
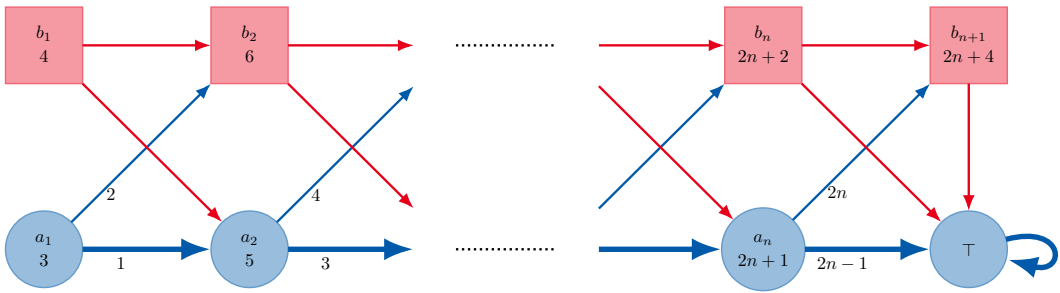

\begin{proof}
    For every~$n\in\N$, we have the sink parity game~$G_n$, see Figure~\ref{fig:BjorklundCounter}, that consists of~$n$ levels such that
\begin{itemize}
    \item each level~$i\in[n]$ contains a player 0 node~$a_i$ with priority~$2i+1$ and a player 1 node~$b_i$ with priority~$2i+2$,
    \item there is a player 1 node~$b_{n+1}$ with priority~$2n+4$ whose only edge ends in the sink~$\top$,
    \item for every~$i\in[n]$, the node~$a_i$ has an outgoing edge with Bland number~$2i$ that leads to~$b_{i+1}$ and an outgoing edge with Bland number~$2i-1$ that leads to~$a_{i+1}$, where we write~$a_{n+1}=\top$,
    \item for every~$i\in[n]$, the node~$b_i$ has two outgoing edges leading to~$b_{i+1}$ and~$a_{i+1}$, respectively.
\end{itemize}
Define a player 0 strategy~$\sigma_0$ for~$G_n$ by~$\sigma_0(a_i)=a_{i+1}$ for every~$i\in[n-1]$ and~$\sigma_0(a_n)=\top$.
    We prove the following two statements by induction on~$n$:
    \begin{itemize}
        \item Strategy improvement with Bland's rule takes~$2^n-1$ iterations to arrive at the optimum.
        \item After an even number of switches,~$\val_{\sigma}(a_1)\tl \val_{\sigma}(b_1)$, and after an odd number of switches,~$\val_{\sigma}(a_1)\tr \val_{\sigma}(b_1)$.
    \end{itemize}
    For~$n=1$, strategy improvement takes one iteration to reach the optimum (by switching~$(a_n, b_{n+1})$, as then we collect an extra~$6$-priority). For both of the possible strategies, we have~$\bar{\sigma}(b_1)=\top$. For the initial strategy (after 0 switches), we have~$\val_{\sigma}(a_1)=\{3\}\tl \{4\} = \val_{\sigma}(b_1)$, and for the optimum (after 1 iteration) we have~$\val_{\sigma}(a_1)=\{3, 6\}\tr \{4\} = \val_{\sigma}(b_1)$. This completes the induction basis.
    
    Now suppose we have proven the lemma for~$G_{n-1}$, and now we want to prove it for~$G_{n}$. If we would remove~$a_1$ and~$b_1$ from~$G_n$ we are left with a parity game equivalent to~$G_{n-1}$. Moreover, the choices in the nodes~$a_1$ and~$b_1$ have no effect on the rest of the game. We conclude that the switches of strategy improvement on the remaining part of the graph are equivalent to a run of the algorithm on~$G_{n-1}$. From the induction hypothesis, we know then that strategy improvement makes~$2^{n-1}-1$ iterations on that part of the game graph, and moreover the ordering of the valuations of~$a_2$ and~$b_2$ changes in each of these iterations. In particular, at the start,~$b_2$ is better valued than~$a_2$, so there is an improving move in~$a_1$. Like before, we see that currently~$b_1$ is better valued than~$a_1$, since~$\sigma(a_1)=\bar{\sigma}(b_1)=a_2$. Since Bland's rule prefers the edges from~$a_1$, it switches~$(a_1,b_2)$ first, after which~$a_1$ becomes better valued than~$b_1$. Then strategy improvement makes a switch in the rest of the game, which, as we know from the induction hypothesis, changes the highest valued of~$a_2$ and~$b_2$. This makes~$b_1$ higher valued than~$a_1$ again, and creates a new improving move from~$a_1$. After making that switch,~$a_1$ will have better value than~$b_1$. This pattern continues until there are no more improving moves from~$\{a_2, a_3, \ldots\}$, at which point~$a_1$ makes a final improving move to arrive at the optimum. From the induction hypothesis we know there were~$2^{n-1}-1$ improving moves from nodes other than~$a_1$, so there are~$2^{n-1}+(2^{n-1}-1)=2^n-1$ switches in total, and the order of the values of~$a_1$ and~$b_1$ alternates. This completes the induction proof.
    \qed
\end{proof}
\subsection{Clustered output}
We now fix some notation to argue asymptotically about the number of memory states allowed. In \cref{thm:PGlordering} we allow for~$o(\frac{m}{\log(m)})$ memory states for parity games with~$m$ player 0 edges. By definition of little-o notation, this implies that there exists an increasing sequence~$(m_i)_{i\in\N}$ of natural numbers~$m_i\in\N$ such that, if~$\ell_i$ denotes the number of memory states allowed for parity games with~$m_i$ player 0 edges, then we have~$\ell_i\leq \frac{m_i}{3i\log (m_i)}$ for all~$i\in\mathds{N}$. We can additionally assume that~$m_i$ is divisible by 3 for every~$i$, and that~$m_1\geq 16$ such that~$m_i\geq 12\ell_i$ for all~$i$. We use this notation throughout the rest of this section.

We assume without loss of generality that the initial state of the memory of~$P$ is 1. We define the following sequence:
\begin{equation*}
        h_1=1, \qquad (g_j,h_{j+1})=P\left(\frac{m_i}{3},m_i,h_j\right), \quad \forall\; 1\leq j< \ell_i',
\end{equation*}
where~$\ell_i'$ is the largest possible number such that the sequence~$(h_1,h_2,\ldots,h_{\ell_i'})$ contains only unique elements. This means that~$P(\frac{m_i}{3},m_i,h_{\ell_i'})=(s,h_{j})$ for some~$j\leq \ell_i'$ and some~$s$. 
For now, we assume that this~$j$ is equal to 1, in which case, if the number of improving moves stays constant at~$\frac{m_i}{3}$, the chosen indices of the improvement rule will be~$(g_1,g_2,\ldots,g_{\ell_i'},g_1,g_2,\ldots,g_{\ell_i'},g_1,g_2,\ldots)$. The other case, where~$P(\frac{m_i}{3},m_i,h_{\ell_i'})\neq(s,h_{1})$ for all~$s$, is dealt with in \cref{sec:outsidecycle}. Before continuing, we need a technical definition.

\begin{definition}\label[definition]{def:clustered}
    Let~$m,\ell\in\mathds{N}$ with~$m\geq 4\ell$. Let~$\mathcal{S}=(i_1, i_2, \ldots, i_{\ell'})\in[m]^{\ell'}$ be a sequence of length~$\ell'\leq \ell$, and let~$I=\{i_1, i_2, \ldots, i_{\ell'}\}$. 
    We say that~$\mathcal{S}$ is~\emph{$(m,\ell)$-clustered} if there exists some~$k\in\mathds{N}$ and integers~$p_1, p_2, \ldots, p_k\in [m]$ and~$q_1, \ldots, q_k\in \mathds{N}$ such that
        \begin{enumerate}
            \item[(i)] the intervals~$[p_1, p_1+q_1], [p_2, p_2+q_2], \ldots,[p_k, p_k+q_k]$ are pairwise disjoint and contained in the interval~$[1,m]$, and the union of these intervals contains~$I$,
            \item[(ii)] for every~$c\in[k]$, and for~$K_c \coloneqq [p_c,p_c+q_c]\cap I$, we have
            \[
            q_c+1\geq \left\lfloor\frac{m}{2\ell}\right\rfloor\cdot \min\left(\max(K_c)-p_c+1,p_c+q_c-\min(K_c)+1\right),
            \]
            \item[(iii)]~$            \sum_{j=1}^{k}(q_j+1)+2(\ell-\ell')\leq m.
           ~$
    \end{enumerate}
\end{definition}

We use this notion to define a class of improvement rules.
We say that the~$\ell$-index-selector function~$P$ is~$(m_i,\ell_i)$-clustered if the sequence~$(g_1,g_2, \ldots, g_{\ell_i'})$ is~$(\frac{m_i}{3},\ell_i)$-clustered\footnote{This is well-defined since~$m_i\geq 12\ell_i\geq 12\ell_i'$ and therefore~$\frac{m_i}{3}\geq 4\ell_i'$.}. Intuitively, this means the preferences of the improvement rule lie in small clusters. That allows us to state the following result.

\begin{lemma}\label[lemma]{lem:clustered}
    Suppose we have an~$\ell$-index-based improvement rule~$\Pi^P$, suppose~$m_i$ is divisible by 3,~$\ell_i:=\ell(m_i)$, and~$m_i\geq 12\ell_i$.  If~$P$ is~$(m_i,\ell_i)$-clustered, then there exists a parity game with~$m_i$ player 0 edges, which takes at least~$2^{\frac{m_i}{12\ell_i}-1}$ iterations to solve for SI with improvement rule~$\Pi^P$.
\end{lemma}

The proof idea of \cref{lem:clustered} is as follows: since~$P$ is~$(m_i,\ell_i)$-clustered, this means that the outputs of~$P$ are contained in intervals~$[p_c,p_c+q_c]$, where the outputs of~$P$ are either all close to the start or to the end of the interval. We make a copy of the binary counter from \cref{fig:BjorklundCounter} for each interval, and we replace each player 0 edge with a fixed number of copies of itself. This artificially increases the number of improving moves, so that it suffices for the pivot rule to choose the indices \emph{close} to the smallest index in each copy of the binary counter.

    Furthermore, we need to make sure that the total number of improving moves stays constant. This is done by replacing each player 0 edge~$(x,y)$ by the controller gadget in \cref{fig:ControllerFiller}. Finally, to make sure we have exactly~$m_i$ player 0 edges in total, we add a number of filler gadgets from \cref{fig:ControllerFiller}, carefully giving them the right Bland numbers. Now we formally prove the lemma.

    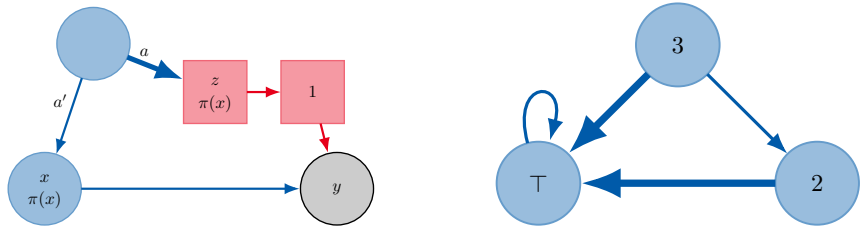
\begin{figure}[htbp]
        \centering
        \resizebox{0.35\linewidth}{!}{%
        \begin{tikzpicture}[node distance = 3cm,
    p0/.style={circle, draw=p0Blue!60, fill=p0Blue!40, thick, minimum size=15mm,inner sep=0pt, outer sep=0pt}, 
    gen/.style={circle, draw=black, fill=black!20, thick, minimum size=15mm}, 
    p0small/.style={circle, draw=p0Blue!60, fill=p0Blue!40, thick, minimum size=6mm}, 
    every path/.style={draw=p0Blue,line width=1.3pt,{>=Latex}}, 
      p1/.style={rectangle, draw=p1Red!60, fill=p1Red!40, thick, minimum size=13mm,inner sep=0pt, outer sep=0pt}, 
      phan/.style={draw=none, fill=none, inner sep=0pt, outer sep=0pt}] 

    \node[p0] (x) at (0,0) {$\begin{array}{c}x\\ \pi(x)\end{array}$};
    \node[gen] (y) at (6,0) {$y$};
    \node[p0] (z) at (1,3) { };   
    \node[p1] (v) at (3.5,2) {$\begin{array}{c}z\\ \pi(x)\end{array}$};
    \node[p1] (w) at (5.5,2) {$1$};

    \draw (x) edge[->] (y);
    \draw (z) edge[->] node[pos=.3,left] {$a'$} (x);
    \draw (z) edge[->, line width=3] node[pos=.3,above] {$a$}(v);
    \draw (v) edge[->,p1Red] (w);
    \draw (w) edge[->,p1Red] (y);

\end{tikzpicture}
        }\hspace{40pt}
    \resizebox{0.35\linewidth}{!}{%
        \begin{tikzpicture}[node distance = 3cm,
    p0/.style={circle, draw=p0Blue!60, fill=p0Blue!40, thick, minimum size=12mm,inner sep=0pt, outer sep=0pt}, 
    gen/.style={circle, draw=black, fill=black!20, thick, minimum size=12mm,inner sep=0pt, outer sep=0pt}, 
    p0small/.style={circle, draw=p0Blue!60, fill=p0Blue!40, thick, minimum size=6mm}, 
    every path/.style={draw=p0Blue,line width=1.3pt,{>=Latex}}, 
      p1/.style={rectangle, draw=p1Red!60, fill=p1Red!40, thick, minimum size=11mm}, 
      phan/.style={draw=none, fill=none, inner sep=0pt, outer sep=0pt}] 

    \node[p0] (s) at (0,0) {$\top$};
    
    \node[p0] (x) at (4,0) {$2$};
    \node[p0] (y) at (2,2) {$3$};

     \draw (s) edge[loop above] (s);
    \draw (y) edge[->, line width=3] (s);
    \draw (x) edge[->, line width=3] (s);
    \draw (y) edge[->] (x);

\end{tikzpicture}
        }
        \caption{Left: the controller gadget, which ensures a constant number of improving switches.~$x,y$ and~$z$ are node labels, all other labels in the nodes denote priorities. Right: The filler gadget (the~$2$ and~$3$ denote priorities). The bold edges are always part of the initial strategy.}
         \label{fig:ControllerFiller}
    \end{figure}

\begin{proof}
    Note that we keep using the notation~$\ell_i$ and~$m_i$. By definition of clusteredness, we get the intervals~$[p_1, p_1+q_1], [p_2, p_2+q_2], \ldots,[p_k, p_k+q_k]$ with the desired properties. Assume w.l.o.g. that the intervals are in order. 

    We first consider~$[p_1,p_1+q_1]$. By property (ii) of \cref{def:clustered}, all the elements of~$g_1,\ldots, g_{\ell_i'}$ are either contained in the first
    \[
    \alpha\coloneqq\left\lfloor\frac{q_1+1}{\lfloor \frac{m_i}{6\ell_i} \rfloor}\right\rfloor
    \]
    elements of the interval or in the last~$\alpha$ elements. Suppose they are in the first~$\alpha$ elements, and assume for convenience~$p_1=1$. Then, take the binary counter game~$G_M$, with
    \[
    M=\ML,
    \]
    from Figure \ref{fig:BjorklundCounter}. This game has~$2M$ player 0 controlled edges, and takes~$2^M-1$ iterations to solve for Bland's rule, according to \cref{lem:alternate}.
    We turn this game into~$G_{M}'$ by replacing each player 0 controlled edge by the multiplier gadget shown in Figure~\ref{fig:duplicate}.
    \begin{figure}
            \centering
            \resizebox{0.4\linewidth}{!}{%
            \begin{tikzpicture}[node distance = 3cm,
    p0/.style={circle, draw=p0Blue!60, fill=p0Blue!40, thick, minimum size=12mm}, 
    gen/.style={circle, draw=black, fill=black!20, thick, minimum size=12mm}, 
    p0small/.style={circle, draw=p0Blue!60, fill=p0Blue!40, thick, minimum size=6mm}, 
    every path/.style={->,draw=p0Blue,line width=1.3pt,{>=Latex}}, 
      p1/.style={rectangle, draw=p1Red!60, fill=p1Red!40, thick, minimum size=11mm}, 
      phan/.style={draw=none, fill=none, inner sep=0pt, outer sep=0pt}] 

    \node[p0] (s1) at (0,0) {$s_1$};
    \node[gen] (s2) at (4,0) {$s_2$};

    \draw (s1) -- (s2);

    \draw[black] (2,-1) -- (2,-2);

    \node[p0] (s) at (-1,-5) {$s_1$};
    \node[p1] (h1) at (2,-3) {0};
    \node[p1] (h2) at (2,-4.5) {0};
    \node[p1] (h3) at (2,-7) {0};
    \node[gen] (t) at (5,-5) {$s_2$};

    \node[phan] (phan1) at (2,-5.4) {};
    \node[phan] (phan2) at (2,-6.1) {};
    \draw[black] (phan1) edge[-, black, dotted] (phan2);
    
    \draw (s) edge[] (h1);
    \draw (s) edge[] (h2);
    \draw (s) edge[] (h3);
    \draw[] (h1) edge[p1Red] (t);
    \draw[] (h2) edge[p1Red] (t);
    \draw[] (h3) edge[p1Red] (t);

\end{tikzpicture}
        }
            \caption{The multiplier gadget, which turns every player 0 edge into multiple edges.}
            \label{fig:duplicate}
        \end{figure}
    To be precise, we do this in such a way that each player~0 edge is replaced by~$\alpha$ edges. The Bland ordering of~$G_{M}'$ is chosen according to the original game; that is, if edge~$e_1$ had a lower Bland number than~$e_2$ in the old graph, then the edges coming from~$e_1$ have lower Bland numbers than those edges coming from~$e_2$. The graph~$G_{M}'$ has 
    \begin{equation}
    2M\alpha=2\cdot \ML\cdot\left\lfloor\frac{q_1+1}{\lfloor \frac{m_i}{6\ell_i} \rfloor}\right\rfloor \leq q_1+1\label{eq:q1}
    \end{equation}
    player 0 edges.

    As a result of the increase in edges, as long as there are exactly~$\frac{m_i}{3}$ improving moves in the game, whenever the index selector function picks an index from the interval~$[p_1,p_1+q_1]$, the pivot rule will necessarily pick an edge with one of the~$\alpha$ lowest Bland numbers. But since that number is exactly the number of times we multiplied our edges, this is equivalent to just performing Bland's rule in the original graph. Hence, if we somehow can keep the number of improving moves constant, this means that our pivot rule needs~$2^M-1$ iterations to solve the game~$G_{M}'$.\footnote{Adding this gadget does give us a number of paths of identical length/value. However, the construction still works if we shift the priorities to be unique with the standard transformation. Say~$e_1$ is an edge from the original graph~$G_{M}$, and say we switch in~$G_{M}'$ to an edge originating from~$e_1$. It could then be that there is another edge related to~$e_1$ that is slightly better. However, this just adds a few improving moves with lower Bland numbers: they could be picked later, but it has no effect on the course of the strategy improvement algorithm.}

    This is done by another modification. We now replace every player 0 edge in~$G_{M}'$ by the controller gadget from \cref{fig:ControllerFiller}.         
    This gives us the graph~$G^1$, which has 
    \[
    6M\alpha=6\cdot \ML\cdot\left\lfloor\frac{q_1+1}{\lfloor \frac{m_i}{6\ell_i} \rfloor}\right\rfloor \leq 3(q_1+1)
    \] player 0 edges. We claim that this will keep the number of improving moves of~$G^1$ constant as long as no edge~$a'$ is switched. This is shown in Lemma~\ref{lemma:controller}.
    
    We pick our Bland numbering as follows: we keep the original edges~$(x,y)$ the same, and give the edges~$a$ and~$a'$ a higher Bland number than all of the original edges. In this way, our improvement rule~$P$ will still pick the edges from~$G_{M}'$, so no edge of the form~$a'$ will be switched until we have gone through the entire binary counter.

    Finally, it could be that (\ref{eq:q1}) did not hold with equality, so~$G_{M}'$ had less than~$q_1+1$ player~$0$ edges. In that case, we add a few copies of the filler gadget from \cref{fig:ControllerFiller} until the total number of player~$0$ edges equals~$3(q_1+1)$.
    We pick the Bland numbers of these filler gadgets such that they get the largest numbers in~$G^1$. Summarizing, the graph~$G^1$ has then~$3(q_1+1)$ nodes, has a constant number of~$q_1+1$ improving moves, and takes at least~$2^{M}-1$ iterations to be solved by improvement rule~$\Pi^{P}$ (ignoring what happens outside of~$G_1$ for now).

    Now we are nearing the end of the proof. Recall that we made two assumptions on the interval~$[p_1,p_1+q_1]$: that~$p_1=1$ and that the intersection of~$g_1,\ldots, g_{\ell_i'}$ with the interval~$[p_1,p_1+q_1]$ is at the start of the interval. For the latter assumption, note that we could simply reverse the Bland order on~$G^1$ to make the construction work for this case. And if~$p_1>1$, then we can add~$p_1-1$ filler gadgets like in \cref{fig:ControllerFiller}, each with lower Bland numbers than all of~$G^1$. These gadgets will not be switched before the binary counter is done, since~$[p_1,p_1+q_1]$ is the first interval, and since~${g_1\ldots,g_{\ell_i'}}$ is contained in the union of the intervals from \cref{def:clustered}.

    To complete our construction, we add binary counters~$G^2,G^3,\ldots, G^k$ to the game. Here, the counter~$G^j$ is based on the interval~$[p_j,p_j+q_j]$, constructed analogously to~$G^1$. For each of the integers from~$[1,\frac{m_i}{3}]$ that are not contained in such an interval, we add a filler gadget from \cref{fig:ControllerFiller}. Naturally, the filler gadget corresponding to number~$r$, with~$p_j+q_j<r<p_{j+1}$ will have Bland number higher than nodes from~$G^j$ and lower than nodes from ~$G^{j+1}$. The whole construction together forms the game~$G_P(m_i)$, which has~$m_i$ player 0 edges.

    What we have shown for~$G^1$ also holds for the other sections of~$G_P(m_i)$.
    We conclude that it takes~$\Pi^P$ at least~$2^M-1$ iterations to solve~$G_P(m_i)$, and that in the first~$2^M-1$ iterations, the number of improving moves is constant at~$\frac{m_i}{3}$.

    Finally, we have
    \[
        M=\ML\geq\frac{\left\lfloor\frac{m_i}{6\ell_i}\right\rfloor-1}{2}>\frac{\frac{m_i}{6\ell_i}-2}{2}=\frac{m_i}{12\ell_i}-1.
    \]
    Hence the algorithm takes at least~$2^M-1\geq 2^{\frac{m_i}{12\ell_i}-1}$ iterations, which completes the proof.
    \qed
    \end{proof}
    
    The following Lemma proves our claim from the previous proof.
    \begin{lemma}\label[lemma]{lemma:controller}
        Suppose that in a parity game the smallest priority of the nodes reachable from player 0 node~$x$ is~$2$. Suppose we replace player 0 edge~$(x,y)$ by the controller gadget shown in \cref{fig:ControllerFiller}, and assume the edge~$a$ is used in strategy~$\sigma$.
        Then~$a'$ is improving if and only if~$a$ is not improving.
    \end{lemma}
    \begin{proof}    
        Assume~$a'$ is improving. Then~$\val_{\sigma}(z)\triangleleft\val_{\sigma}(x)$, by definition.
        This means that the valuation of~$x$ is better than that of~$z$, and since the smallest priority in the remaining reachable game is~$2$, this means that the valuation of~$x$ is at least as good as the valuation of~$z$ when the 1 is removed from its valuation. But that is the same as saying~$\val_{\sigma}(x)\trianglerighteq \val_{\sigma}(y)\uplus\pi(x)$, which is equivalent to saying~$(x,y)$ is not an improving edge.
    
        On the other hand, assume~$a$ is not improving. This means that\[\val_{\sigma}(x)\trianglelefteq \val_{\sigma}(z)=\val_{\sigma}(y)\uplus \{\pi(x),1\}\triangleleft \val_{\sigma}(y)\uplus \{\pi(x)\},\] which yields that~$(x,y)$ is improving.
    \qed
    \end{proof}

    \begin{example}
    We demonstrate our construction from \cref{lem:clustered} with an example. Let~$m_i=39$ and~$l_i=4$. In this case, we don't have~$m_i\geq 12l_i$, but we make the construction anyway for the sake of example (if we followed the exact construction, we would need a very large graph to make an interesting example). We just make our construction slightly different from what is prescribed by the proof of \cref{lem:clustered}. Suppose that we have a~$(\frac{m_i}{3},4)$-clustered improvement rule. Let in this case~$g_1=1,g_2=13,g_3=1$ and~$g_4=2$. This is indeed a~$(13,4)$-clustered sequence: we can pick~$[p_1,p_1+q_1]=[1,8]$, and~$[p_2,p_2+q_2]=[10,13]$. These intervals contain all the~$g_j$, and also satisfy the second condition of clusteredness:
    \begin{eqnarray*}
        8=q_1+1&\geq& \left\lfloor\frac{m_i}{6l_i}\right\rfloor\cdot \min\left[\max(K_1)-p_1+1,p_1+q_1-\min(K_1)+1\right]\\
        &=&\left\lfloor \frac{39}{24}\right\rfloor \cdot \min(2,8),\\
        4=q_2+1&\geq& \left\lfloor\frac{m_i}{6l_i}\right\rfloor\cdot \min\left[\max(K_2)-p_2+1,p_2+q_2-\min(K_2)+1\right]\\
        &=&\left\lfloor \frac{39}{24}\right\rfloor\cdot \min(4,1)
    \end{eqnarray*}
     The lemma would prescribe using binary counter~$G_{M}$ as a starting point for~$G^1$ and~$G^2$. However,~$M=0$, but we instead start with~$G_2$. The proof still works for this example, since in our example~$g_1,g_2,g_3,g_4$ are particularly nicely clustered. Next, the edges are multiplied~$\lfloor\frac{q_c+1}{\lfloor\frac{m_i}{6l_i}\rfloor}\rfloor$ times. Again, since our example is particularly nicely clustered, we only have to multiply by~$4$ times less: that is, double the edges for the first counter and leave the edges for the second counter the same.

    Finally, we add the gadgets from \cref{fig:ControllerFiller} and a filler gadget for the 9th improving move, which is not in an interval. We reverse the Bland order of~$G^2$, since~$g_2$ is at the end of the interval~$[p_2,p_2+q_2]$. The final construction can be found in \cref{fig:clusteredexample}. It follows from the construction that it takes the improvement rule at least~$2^2-1$ iterations to solve this game.
    
    \begin{figure}[htbp]
        \centering
        \includestandalone[width=\linewidth]{Clusteredexample}
        \caption{An example of a construction as in the proof of \cref{lem:clustered}. The top part of the game is~$G^1$, and the bottom part is~$G^2$.}
        \label{fig:clusteredexample}
    \end{figure}
\end{example}

\subsection{Dispersed output}
We need another technical definition.
\begin{definition}\label[definition]{def:dispersed}
    Let~$m,\ell\in\mathds{N}$ with~$m\geq 4\ell$. Let~$\mathcal{S}=(i_1, i_2, \ldots, i_{\ell'})\in[m]^{\ell'}$ be a sequence of length~$\ell'\leq \ell$, and let~$I=\{i_1, i_2, \ldots, i_{\ell'}\}$.
    We say that ~$\mathcal{S}$ is~\emph{$(m,\ell)$-dispersed} if there exists some~$k\in\mathds{Z}_{\geq 0}$ and integers~$\psi, p_1, p_2, \ldots, p_k\in [m]$ and~$\xi , q_1, \ldots, q_k\in \mathds{N}$ such that, for~$K=I\cap [\psi,\psi+\xi]$,
    \begin{enumerate}
        \item[(i)] the intervals~$[\psi, \psi+\xi],[p_1, p_1+q_1], [p_2, p_2+q_2], \ldots,[p_k, p_k+q_k]$ are pairwise disjoint and contained in~$[1,m]$, and the union of these intervals contains~$I$,
        \item[(ii)] for~$j=1,2, \ldots, k$, we have~$q_j\geq 2\cdot \left|\{c\in [\ell']:i_c\in [p_j,p_j+q_j]\}\right|-1$,
        \item[(iii)] either~$K=\{\psi\}$ or~$K=\{\psi+\xi\}$,
        \item[(iv)]~$\xi\geq \lfloor\frac{m}{2\ell}\rfloor|\{c:i_c\in K\}|-1$,
        \item[(v)]~$(\xi+1)+\sum_{j=1}^{k}(q_j+1)+2(\ell-\ell')\leq m$.
    \end{enumerate} 
\end{definition}

We say that the~$\ell$-index-selector function~$P$ is~$(m_i,\ell_i)$-dispersed if the sequence~$(g_1,g_2, \ldots, g_{\ell_i'})$ is~$(\frac{m_i}{3},\ell_i)$-dispersed. Intuitively, this means there is one big interval~$[\psi,\psi+\xi]$ in which the pivot rule has almost no outputs. We prove the following lemma.

\begin{lemma}\label[lemma]{lem:dispersed}
    Suppose we have an~$\ell$-index-based improvement rule~$\Pi^P$, suppose~$m_i$ is divisible by 3,~$\ell_i:=\ell(m_i)$, and~$m_i\geq 12\ell_i$.  If~$P$ is~$(m_i,\ell_i)$-dispersed, then there exists a parity game with~$m_i$ player 0 edges, which takes at least~$2^{\frac{m_i}{12\ell_i}-1}$ iterations to solve for SI with improvement rule~$\Pi^P$.
\end{lemma}

The main proof idea for \cref{lem:dispersed} is as follows: the dispersedness tells us that there is a large interval~$[\psi,\psi+\xi] \subseteq [1,m]$, in which the outputs of the improvement rule are either all at the start or at the end. We take one copy of the binary counter from \cref{fig:BjorklundCounter}, and then use the delayer gadget as seen in \cref{fig:GadgetConsecutive}: each player 0 node with two outgoing edges towards nodes~$x$ and~$y$ is replaced by this gadget. By doing this, it essentially delays a switch by~$k$ moves: it now takes~$k+1$ iterations to get the same result as usually one iteration would, see Lemma~\ref{lem:consecutive}. We pick the Bland numbers such that this counter will be identified with~$[\psi,\psi+\xi]$, and we pick the parameters~$k$ from the gadgets such that each iteration of the original binary counter takes exactly as long as it takes the improvement rule to repeat its memory state. We call the time until reaching the same memory state a cycle.

In each cycle,~$P$ can give many different outputs, so we need to deal with the outputs that are outside~$[\psi,\psi+\xi]$. By \cref{lem:alternate}, the (replacements of)~$a_1$ and~$b_1$ alternate which of them is the best, every cycle. We attach additional copies of the delayer gadget to our game, identifying~$x$ with~$b_1$ and~$y$ with~$a_1$ in each gadget. This creates new improving moves every cycle for every output of~$P$ outside the interval.
Finally, like for \cref{lem:clustered}, we use controller and filler gadgets from \cref{fig:ControllerFiller} to make sure the number of improving moves is constant and the total number of edges equals~$m_i$. 

\begin{lemma}
    \label[lemma]{lem:consecutive}
    In the gadget from \cref{fig:GadgetConsecutive}, suppose that in the strategy~$\sigma$ we have~$\val_{\sigma}(x)\tr\val_{\sigma}(y)$ and that these two valuations differ by at least a priority~$3$. Assume furthermore that~$\sigma$ uses edges~$l_0,l_1,\ldots,l_k$. 
    Let~$\sigma^{(1)}$ be obtained from~$\sigma$ by some switch outside the gadget, such that~$\val_{\sigma^{(1)}}(x)\tl\val_{\sigma^{(1)}}(y)$ (again differing by at least a 3). Then the following are true:
    \begin{itemize}
        \item There are no improving moves inside the gadget for strategy~$\sigma$.
        \item Starting at~$\sigma^{(1)}$, the gadget shown in the figure needs exactly~$k+1$ improving switches to guarantee reaching~$y$. In every iteration, there is exactly one improving switch possible in the gadget.
    \end{itemize}
    By symmetry, this lemma is also true if the roles of~$x$ and~$y$ are reversed and~$\sigma$ uses~$r_0,r_1,\ldots,r_k$.
\end{lemma}
\begin{figure}[tbp]
    \centering
    \resizebox{\linewidth}{!}{%
    \begin{tikzpicture}[node distance = 3cm,
    p0/.style={circle, draw=p0Blue!60, fill=p0Blue!40, thick, minimum size=15mm,inner sep=0pt, outer sep=0pt}, 
    every path/.style={draw=p0Blue,line width=1.3pt,{>=Latex}}, 
      p1/.style={rectangle, draw=p1Red!60, fill=p1Red!40, thick, minimum size=11mm}, 
      undef/.style={diamond, draw=black, fill=black!20, thick, minimum size=11mm},
      phan/.style={draw=none, fill=none, inner sep=0pt, outer sep=0pt}] 
    {\LARGE
    \node[p0] (a1) at (0,0) {$\begin{array}{c}
         z_1  \\
         1 
    \end{array}$};
    \node[p0] (a2) at (4,0) {$\begin{array}{c}
         z_2  \\
         1 
    \end{array}$};
    \node[phan] (phan1) at (8,0) {}; 
    \node[phan] (phan2) at (12,0) {}; 
    \node[p0] (a3) at (16,0) {$\begin{array}{c}
         z_k  \\
         1 
    \end{array}$};
    \node[p0] (a4) at (20,0) {$\begin{array}{c}
         z_{k+1}  \\
         1 
    \end{array}$};
    \node[p1] (b1) at (0,4) {$2$};
    \node[p1] (b2) at (20,4) {$2$};
    \node[undef] (phan3) at (0,8) {$x$}; 
    \node[undef] (phan4) at (20,8) {$y$}; 

    \draw (a1) edge[harp] node [above] {$r_k$} node [below] {$l_1$} (a2);
    \draw (a2) edge[harp] node [above] {$r_{k-1}$} node [below] {$l_2$} (phan1);
    \draw ($(phan1)!0.3!(phan2)$) edge[-,black,dotted] ($(phan1)!0.7!(phan2)$);
    \draw (phan2) edge[harp] node [above] {$r_2$} node [below] {$l_{k-1}$} (a3);
    \draw (a3) edge[harp] node [above] {$r_1$} node [below] {$l_k$} (a4);

    \draw (a1) edge[->] node [left] {$l_0$} (b1);
    \draw (a4) edge[->] node [right] {$r_0$} (b2);

    \draw (b1) edge[->,bend left=10pt,p1Red] node [pos=.3,above] {$d_l$} (a4); 
    \draw (b2) edge[->,bend right=10pt,p1Red] node [pos=.3,above] {$d_r$} (a1);    
    \draw (b1) edge[->,p1Red] node [left] {$u_\ell$} (phan3); 
    \draw (b2) edge[->,p1Red] node [right] {$u_r$} (phan4);
    }
\end{tikzpicture}}
    \caption{The delayer gadget. It makes a significant switch take~$k+1$ times more iterations.}
    \label{fig:GadgetConsecutive}
\end{figure}
\begin{proof}
    Suppose~$x$ is higher valued than~$y$. Then there is only one strategy in this gadget for Even to force the play to go to~$x$, which is the strategy choosing~$l_0, l_1, \ldots, l_k$. For any other strategy, Odd can use edges~$d_{l}$ and~$u_r$, after which Even either creates a cycle with only priorities 1 or plays edge~$r_0$ to end up at~$y$. However, by playing~$l_0, l_1, \ldots, l_k$, Odd cannot use edge~$d_{l}$, as it would create a cycle with 2 as highest priority, so Odd has to play~$u_{l}$, ending up at~$x$.

    Now suppose suddenly~$y$ becomes better than~$x$. This means that Odd will now play~$d_r$ instead of~$u_r$, in order to end up at~$x$. This means that edge~$r_0$ becomes improving, as its head has an extra~$2$ in its valuation compared to the current choice. Any other switch would create a cycle with only~$1$-priorities, so it cannot be improving. After the first switch, we can argue analogously that~$r_1$ is the only improving switch, and after that~$r_2$, etc, until~$r_k$ has been switched. Now Odd is unable to play~$d_r$ and has to play~$u_r$ towards~$y$, and there is no more improvement possible within the gadget. The case where then~$x$ becomes better than~$y$ again is the same, so this completes the proof.
    \qed
\end{proof}

Now we are ready to prove Lemma~\ref{lem:dispersed}.

\begin{proof}
    We again write
    \[
    M:=\ML.
    \]
    We start our construction with the game $G_{M}$, as defined in \cref{fig:BjorklundCounter}. Our goal is to identify the edges of~$G_M$ with elements from the interval~$K:=[\psi,\psi+\xi]$, in a similar way we did for~$G^1$ in the proof for the clustered case. However, in this case we have much less restrictions on the output of the improvement rule outside the interval~$[\psi,\psi+\xi]$, so we need some mechanism to deal with outputs outside the interval~$K$.

    Consider the interval~$[\psi,\xi+\psi]$. The values~$g_i$ in this interval are either equal to the first element or all to the last. We assume w.l.o.g. they are equal to the first. We modify our game~$G_M$ in such a way that making a significant improving move takes many small steps. To do so, we introduce a new gadget, called the delayer gadget. It is shown in \cref{fig:GadgetConsecutive}.

Every node~$a_i$ has two outgoing edges. Consider the quantity~$|\{c:g_c\in K\}|$. If it is equal to 1, we let~$G_M'=G_M$. Otherwise, we place around every~$a_i$ a gadget as in \cref{fig:GadgetConsecutive}, picking~$k=|\{c:g_c\in K\}|-1$ (note: this is a different~$k$ than in the definition of dispersedness). We do this replacement from high to low (first~$a_n$, then~$a_{n-1}$, and so on). When placing the gadget around~$a_i$, we pick~$b_{i+1}$ as the node~$x$ (so edge~$u_l$ will go to~$b_{i+1}$), and node~$a_{i+1}$ as our node~$y$ (or, if~$i=n$,~$\top$ will be~$y$). Finally, we turn~$a_i$ into a player 1 node, and add only an edge from~$a_i$ to~$z_1$ in the new gadget. The Bland numbers in the graph~$G_M'$ are chosen in increasing order along the gadgets: that is, any edge in the gadget of~$a_i$ has lower Bland number than any edge in the gadget of~$a_{i+1}$ for all~$i$.
For the initial strategy, we take the one equivalent to the original initial strategy from~$G_M$: that is the strategy using the edges~$r_0,r_1,\ldots,r_k$ in each gadget.

In both cases, the graph~$G_M'$ will have~$|\{c:g_c\in K\}|\cdot 2M$ player 0 edges. Moreover, it follows from \cref{lem:consecutive} that, if we perform strategy improvement on~$G_M$ with Bland's rule, the run of the algorithm will be equivalent to the run of Bland's rule on~$G_M$; the only difference is that it takes~$|\{c:g_c\in K\}|$ times more steps to reach the optimum. Moreover, in~$G_M$, which of the nodes~$a_1$ and~$b_1$ is the highest, switches every iteration: but in~$G_M'$, this switches every~$|\{c:g_c\in K\}|$ iterations.

Finally, like in the clustered case, we replace every player 0 edge with the controller gadget from \cref{fig:ControllerFiller} to make sure that the number of improving moves stays constant, as long as the edges of the form~$a'$ are not switched. We give the edges of the form~$a'$ therefore a higher Bland number than the original edges. Call the new game~$G_M''$, which will have~$|\{c:g_c\in K\}|\cdot 6M$ player 0 edges. The number of improving moves will be constant at~$2M\cdot |\{c:g_c\in K\}|$, and
\[
    \xi\geq \lfloor\frac{m_i}{6\ell_i}\rfloor|\{c:g_c\in K\}|-1\geq 2M\cdot |\{c:g_c\in K\}|-1
\]
so there is enough space in the interval~$[\psi,\psi+\xi]$. Additionally, we add some filler gadgets from \cref{fig:ControllerFiller} to make sure the game~$G_M''$ has exactly~$3(\psi+1)$ player 0 edges and~$\psi+1$ improving moves. Of course, the case~$K=\{\psi+\xi\}$ is analogous: we just reverse the Bland ordering of the game~$G_M''$.

Next, we need to deal with the elements of the sequence~$(g_1,g_2,\ldots, g_{\ell_i(p)'})$ that are outside the interval~$[\psi,\psi+\xi]$. We need to make sure that, whenever the improvement rule outputs these elements, this does not interfere with the game~$G_M''$. By property (ii) of \cref{def:dispersed}, they are contained in intervals~$[p_j,p_j+q_j]$, where each interval has double the length of the number of sequence elements~$g_j$ in it. Our goal is to add a number of decoy switches for each interval.

For each of the intervals~$[p_j,p_j+q_j]$, we add a gadget to the game~$G_M''$. This depends on the quantity~$Q_j:=|\{c:g_c\in [p_j,p_j+q_j]\}|$. If~$Q_j=1$, we add just one player 0 node~$v_j$, with two outgoing edges:~$(v_j,a_1)$ and~$(v_j,b_1)$. If~$Q_j>1$, then we add the gadget from \cref{fig:GadgetConsecutive}, picking~$k=Q_j-1$, identifying~$x$ with~$b_1$ and~$y$ with~$a_1$. Hence in both cases, we have~$2Q_j$ player 0 edges. Next, we replace every player 0 edge from this new gadget with the gadget from \cref{fig:ControllerFiller}, resulting in the gadget~$G^j$, with~$6Q_j$ player 0 edges. Moreover, by \cref{lem:consecutive}, there are always~$2Q_j$ improving moves in~$G^j$, as long as no edge of the form~$a'$ is switched. It is, of course, no coincidence that, by \cref{def:dispersed}, we have 
\[
q_j\geq 2Q_j-1
\]
so there is enough space in the interval~$[p_j,p_j+q_j]$. Next, we add some filler gadgets from \cref{fig:ControllerFiller} to~$G^j$ to make sure that~$G^j$ has exactly~$3(q_j+1)$ player~0 edges.

Finally, we need to carefully pick the Bland numbers and the initial strategy in~$G^j$, to make sure that there always is an available switch for strategy improvement. To gain some intuition for this, we reorder the sequence of outputs of~$P$ as
\begin{equation}\label{eq:reorder}
    (g_{\phi+1},\ldots, g_{\ell_i(p)'}, g_1,g_2,\ldots, g_\phi)
\end{equation}
where~$\phi=\max\{c\in [\ell_i(p)']:g_c\in [\psi,\psi+\xi]\}$.
If the outputs of~$P$ are cycling through this sequence, after~$g_{\phi}$ is output, we switch so that the best valued between~$a_1$ and~$b_1$ changes. This then allows for new switches within~$G^j$, hence we want all the possible switches within~$G^j$ to be made between these occurrences of~$g_{\phi}$.

To be more precise, let~$h_1,h_2,\ldots,h_{Q_j}$ be the subsequence of \cref{eq:reorder} containing exactly the elements of~$[p_j,p_j+q_j]$. We want the first switch ($r_0$ or~$l_0$) to be in the~$h_1$-th position in the Bland ordering when it is improving, the second switch ($r_1$ or~$l_1$) in the~$h_2$-th position, and so on. This fixes our choice of initial strategy in~$G^j$: we pick edges~$r_0,r_1,\ldots, r_{|\{c<\phi: g_c\in [p_j,p_j+q_j]\}|-1}$, and pick the edges~$l_0,l_1,\ldots, l_{k-|\{c<\phi: g_c\in [p_j,p_j+q_j]\}|}$. Since initially~$b_1$ is higher valued, this means that there are~$|\{c<\phi: g_c\in [p_j,p_j+q_j]\}|$ switches to be made within~$G^j$ before we arrive at~$g_{\phi}$.

The required Bland ordering for~$G^j$ can be found in the following way: first, we assign Bland numbers to the edges of the form~$a$ and~$a'$ and of the filler gadgets. This can be done arbitrarily, as long as they are chosen such that they will lie in the interval~$[p_1,p_1+q_1]$ at the end. Next, we know that there will only be one of the edges~$r_0,r_1,\ldots, r_k,l_0,l_1,\ldots,l_k$ improving at a time (\cref{lem:consecutive}), and that when it is improving, then its related two edges~$a$ and~$a'$ are not improving. This means we can pick the Bland numbers of~$r_0,\ldots, r_k,l_0,\ldots,l_k$ independently without them interfering with each other. We determined that edges~$r_s$ and~$l_s$ must be the~$h_{s+1}$-th in the Bland ordering, and we can achieve this by letting its Bland number be in the~$h_{s+1}-p_1+1$-th position relative to the other moves that would be improving at the same time.

Finally, to complete the construction, we add a few more filler gadgets from \cref{fig:ControllerFiller} for all the numbers from~$[\frac{m_i}{3}]$ that were not in an interval~$[p_j,q_j]$ or~$[\psi,\psi+\xi]$, with the appropriate Bland number.

Summarizing, every time that~$g_{\phi}$ is output, the best valued of~$a_1$ and~$b_1$ switches, creating~$Q_j$ new improving moves within~$G^j$. All of these improving moves are then switched in between, whenever~$h_1,h_2,\ldots,h_{Q_j}$ is output, which `resets' the gadget again. In the end, this allows our binary counter~$G_M''~$ to operate normally, and we have already seen in the proof of \cref{lem:clustered} that~$\Pi^P$ therefore takes a superpolynomial number of iterations to solve the resulting game.
    \qed
\end{proof}

\begin{example}
    Now we give an example of our construction for a specific pivot rule. Let~$m_i=48$ and~$l_i=4$, and let~$g_1=1$,~$g_2=g_3=5$ and~$g_4=3$. This sequence is~$(\frac{m_i}{3},l_i)$-dispersed.
    In particular, the intervals~$[p_1,p_1+q_1]=[1,4]$ and~$[\psi,\psi+\xi]=[5,12]$ satisfy the properties of \cref{def:dispersed}.

    First, we construct the binary counter game~$G_M$. In this case,~$M=1$. However, because we picked our numbers nicely, we can take~$G_2$ instead (we do this just for the sake of example, to get an interesting construction without needing a very large graph). We add the gadgets from \cref{fig:GadgetConsecutive} to the game~$G_2$ with~$k=|c:g_c\in K|-1=1$, and then add the controller gadgets from \cref{fig:ControllerFiller} and any filler gadgets (which we don't need in this case). This results in the parity game shown in \cref{fig:dispex}, excluding the part left of~$a_1$ and~$b_1$. The game~$G_M''$ must have improving moves in the interval~$[5,12]$, so its edges have Bland numbers ranging from~$13$ to~$48$.

    Next, we look at the interval~$[p_1,p_1+q_1]$. Since there are two elements in this interval~($g_1,g_4$), we have~$Q_1=2$, and we add the gadget from \cref{fig:GadgetConsecutive} with~$k=1$, and connect it to~$a_1$ and~$b_1$. Since we have~$\phi=3$, we obtain that~$|\{c<\phi:g_c\in [p_1,p_1+q_1]\}|=1$. Therefore the initial strategy in~$G^1$ will have~$r_0$ and~$\ell_0$. 
    Since~$q_1=2Q_1-1$, we do not have to introduce any filler gadgets for~$G^1$.

    For the Bland numbering, first we arbitrarily number the~$8$ edges in~$G^1$ that are not~$r_0,r_1,\ell_0,\ell_1$. We number them~$3,4,6,8,9,10,11,12$ (see \cref{fig:dispex}).\footnote{Of course, in this example we picked these nicely, if one would generate this in practice one would need to shift some numbers around to make sure all edges have unique, consecutive, integer Bland numbers.}
    We have~$h_1=3$ and~$h_2=1$. If~$r_0$ is improving, it needs to be the~$h_1$-th improving move. The other improving moves that are active at the same time have Bland numbers~$3,4,6$. So edge~$r_0$ must have a Bland number between~$4$ and~$6$ to be the third, and we pick~$5$. Similarly, for edge~$r_1$, this must be the~$h_2$-th, and since~$3,6,8$ are active at the same time, it follows that~$r_1$ must have Bland number below~$3$. In a similar manner, we can pick Bland numbers for~$\ell_0$ and~$\ell_1$, resulting in the numbers shown in the figure.
    \begin{figure}[htbp]
        \centering
        \includestandalone[width=1.5\linewidth, angle=90]{DispersedExample}
        \caption{The parity game constructed based on the pivot rule~$P$. Bland numbers are shown on the edges, priorities on the vertices, and bold edges represent the initial strategy.}
        \label{fig:dispex}
    \end{figure}
\end{example}

\subsection{Completing the proof} \label{sec:outsidecycle}
We combine the results of \cref{lem:clustered,lem:dispersed} with the following lemma.
\begin{lemma}\label[lemma]{lem:clustereddispersed}
    Let~$m,\ell\in \mathds{N}$ such that~$m\geq 4\ell$. Let~$\mathcal{S}=(i_1, i_2, \ldots, i_{\ell'})\in[m]^{\ell'}$ be a sequence of length~$\ell'\leq \ell$. Then~$\mathcal{S}$ is~$(m,\ell)$-clustered or~$(m,\ell)$-dispersed.
\end{lemma}

Before we continue, we first need a small intermediate result.
\begin{lemma}\label[lemma]{lem:doubleintervals}
    Let~$m,\ell\in\mathds{N}$ with~$m\geq 2\ell$, and let~$i_1,i_2,\ldots,i_{\ell}\in [m]$.
    Then, there exist~$k\in\mathds{N}$ and~$p_1,p_2,\ldots,p_k,q_1,q_2,\ldots,q_k\in [m]$ such that
    \begin{itemize}
        \item the intervals~$[p_1,p_1+q_1],\ldots,[p_k,p_k+q_k]$ are disjoint and contain the set~$I\coloneqq\{i_c:c\in [\ell]\}$,
        \item we have~$q_j=2|\{c:i_c\in [p_j,p_j+q_j]\}|-1$ for all~$j\in\{1,2,\ldots,k\}$.
    \end{itemize}
\end{lemma}
\begin{proof}
    We fix an arbitrary~$m$ and perform an induction on~$\ell$. For~$\ell=1$, the statement is obvious. Now assume that the statement is true for~$\ell-1$, where~$1<\ell\leq \frac{m}{2}$, then we want to show it for~$\ell$. By our induction hypothesis, we get intervals~$[p_1,p_1+q_1],\ldots,[p_k,p_k+q_k]$ that satisfy the conditions for the sequence~$(i_1,\ldots, i_{\ell-1})$. Assume w.l.o.g. that~$p_1<p_2<\ldots<p_k$. Now we want to add~$i_{\ell}$ to the sequence. 
    We consider two cases.
    
    First, assume that~$i_{\ell}$ and one of~$\{i_{\ell}\pm1\}$ is not contained in any interval. Then, we add the interval~$[p_{k+1},p_{k+1}+q_{k+1}]$, which is either equal to~$[i_{\ell},i_{\ell}+1]$ or~$[i_{\ell}-1,i_{\ell}]$.
        
    Otherwise, we know that~$i_{\ell}$ is inside or adjacent to an interval~$[p_j,p_j+q_j]$. Then, we add~$i_{\ell}$ to the interval~$[p_j,p_j+q_j]$, and we increase the length of that interval by 2 (we can always do this, if we merge all intervals that are adjacent to each other between each addition).        
    This completes the induction proof.
    \qed
\end{proof}

Now we are ready to prove \cref{lem:clustereddispersed}.

\begin{proof}
    Recalling the definitions of clusteredness and dispersedness, the statement of the lemma is equivalent to saying one of the following two cases holds:
    \begin{enumerate}
        \item[(a)]  There exist~$k\in\mathds{N}$, and integers~$p_1, p_2, \ldots, p_k\in [m]$ and~$q_1, \ldots, q_k\in \mathds{N}$ such that the following three statements hold:
        \begin{enumerate}
            \item[(i)] The intervals~$[p_1, p_1+q_1], [p_2, p_2+q_2], \ldots,[p_k, p_k+q_k]$ are pairwise disjoint and contained in~$[1,m]$. Moreover, the union of these intervals contains~$I$.
            \item[(ii)] Let~$K_c=[p_c,p_c+q_c]\cap I$ for~$c\in [k]$. Then for every~$c\in[k]$ the following holds:
            \[
            q_c+1\geq \lfloor\frac{m}{2\ell}\rfloor\cdot \min\left[\max(K_c)-p_c+1,p_c+q_c-\min(K_c)+1\right]
            \]
            \item[(iii)] 
           ~$\sum_{j=1}^{k}(q_j+1)+2(\ell-\ell')\leq m~$
    \end{enumerate}
        \item[(b)] There exists some~$k\in\mathds{Z}_{\geq 0}$ as well as integers~$\psi, p_1, p_2, \ldots, p_k\in [m]$ and~$\xi , q_1, \ldots, q_k\in \mathds{N}$ such that the following five statements hold:
        \begin{enumerate}
        \item[(i)] The intervals~$[\psi, \psi+\xi],[p_1, p_1+q_1], [p_2, p_2+q_2], \ldots,[p_k, p_k+q_k]$ are pairwise disjoint and contained in~$[1,m]$. Moreover, the union of these intervals contains~$I$.
        \item[(ii)] For all~$j\in\{1,2, \ldots, k\}$, we have~$q_j\geq 2|\{c\in [\ell']:i_c\in [p_j,p_j+q_j]\}|-1$.
        \item[(iii)] Let~$K=I\cap [\psi,\psi+\xi]$. Then either~$K=\{\psi\}$ or~$K=\{\psi+\xi\}$.
        \item[(iv)] We have~$\xi\geq \lfloor\frac{m}{2\ell}\rfloor|\{c:i_c\in K\}|-1$.
        \item[(v)] We have~$(\xi+1)+\sum_{j=1}^{k}(q_j+1)+2(\ell-\ell')\leq m$.
    \end{enumerate} 
    \end{enumerate}
    Throughout this proof, let~$d=\lfloor \frac{m}{2\ell}\rfloor$. Since~$m\geq 4\ell$, we have~$d\geq 2$. We prove the lemma by induction on~$\ell'$.
    
    As induction basis, we have~$\ell'=1$. We show that (b) is true. Note that, in this case,~$\ell\geq 1$ so~$m\geq 2d$. If~$i_1\leq d$, then pick~$[\psi, \psi+\xi]=[i_1, i_1+d-1]$ and~$k=0$. Then we have
    \[
    \xi\geq d\cdot 1-1=\lfloor\frac{m}{2\ell}\rfloor|\{c:i_c\in K\}|-1
    \]
    so condition (iv) is satisfied. Moreover, since~$m\geq 4\ell$, we have~$d\geq 2$, and thus
    \[
    (\xi+1)+2(\ell-\ell') \leq d\cdot \ell'+d(\ell-\ell')=d\ell<m
    \]
    which proves (v). It is easy to see that the other three conditions are also satisfied. In the other case, if~$i_1\geq d+1$, we pick~$[\psi, \psi+\xi]=[i_i-d+1, i_1]$ instead. Then we again have
    \[
    \xi= d\cdot 1-1=\lfloor\frac{m}{2\ell}\rfloor|\{c:i_c\in K\}|.
    \]
    So (iv) is satisfied, and the other conditions are also satisfied.

    Now suppose we have shown that the statement is true for all sequences of \break length~$1,2, \ldots, \ell'-1$, given some~$\ell'$ with~$1<\ell'\leq \ell$. Then we want to show the correctness of the lemma for~$\ell'$. We define~$\lambda=|\{c:i_c=\min(I)\}|$ as well as~$\lambda'=|\{c:i_c=\max(I)\}|$. We distinguish four cases:
     \begin{itemize}
         \item Suppose~$\max(I)\leq 2\ell'$ or~$\min(I)\geq m-2\ell'+1$. Then (a) holds, because we can pick~$k=1$, and either~$[p_1, p_1+q_1]=[1,2d\ell']$ in the first case or~$[p_1,p_1+q_1]=[m-2d\ell'+1,m]$ in the second case. Condition (i) trivially holds. If~$\max(I)\leq 2\ell'$, then (ii) follows from the fact that
         \[
         q_1+1= 2d\ell'=d(2\ell'-1+1) \geq \lfloor\frac{m}{2\ell}\rfloor\left(\max(K_1)-p_1+1\right)
         \]
         and likewise, if~$\min(I)\geq m-2\ell+1$, then
         \[
         q_1+1\geq \lfloor\frac{m}{2\ell}\rfloor\left(p_1+q_1-\min(K_1)+1\right)
         \] Condition (iii) also holds, since:
         \[\sum_{j=1}^{k}(q_j+1)+2(\ell-\ell')=2d\ell'+2(\ell-\ell')\leq 2d\ell'+2d(\ell-\ell') =2d\ell \leq m\]
        \item Suppose~$\lambda d<\min(I)< m-2\ell'+1$. We want to show that (b) holds. We first pick~$k=1$, and~$[\psi, \psi+\xi]=[\min(I)-\lambda d+1,\min(I)]$. Note that the interval~$[\min(I)+1,m]$ contains at least~$2\ell'$ integers. We can then apply \cref{lem:doubleintervals} to the interval~$[\min(I)+1,m]$ and the remaining sequence elements to find valid intervals~$[p_1,p_1+q_1],\ldots,[p_k,p_k+q_k]$ such that~$q_c+1=2|\{c:i_c\in[p_c,p_c+q_c]\}|$. The conditions (i)-(iv) now hold by definition. For condition (v), using~$d\geq 2$, we get 
        \begin{align*}
            (\xi+1)+\sum_{j=1}^{k}(q_j+1)+2(\ell-\ell')&=\lambda d +2(\ell'-\lambda) +2(\ell-\ell')\\&=\lambda d+2(l-\lambda) \leq d\ell < m.
        \end{align*}  
        \item The case~$2\ell'<\max(I)< m-\lambda' d+1$ is analogous to the previous one.
        \item Suppose~$\min(I)\leq \lambda d$ and~$\max(I)\geq m-\lambda' d+1$. Our goal for this case is to split the interval~$[1,m]$ into two intervals, and then use the induction hypothesis on both of them. To determine how to split, we consider the function~$f:[\ell]\to [\ell]$ defined by
            \[
                f(\eta) = |\{c:i_c\leq 2d\eta\}|
            \]
        We can make the following three observations about~$f$:
        \begin{itemize}
        \item[$\circ$] The function~$f$ is nondecreasing and integer-valued, by definition.
        \item[$\circ$] We have~$f(\lambda)\geq \lambda$, since~$\min(I)\leq \lambda d<2\lambda d$ and by definition of~$\lambda$.
        \item[$\circ$] We have~$f(\ell'-\lambda')\leq \ell'-\lambda'$. This is since~$\max(I)\geq m-\lambda' d +1$ and~$m\geq 2d\ell'$ imply~$\max(I)>2d(\ell'-\lambda')$, and by definition of~$\lambda'$.
    \end{itemize}
         It follows immediately from these three facts that~$f$ has a fixed point~$\eta^*$ with~$\lambda\leq \eta^*\leq \ell'-\lambda'$ (note that~$\lambda\leq \ell'-\lambda'$ by definition).

         Now, we split~$[m]$ into two intervals~$[1,2d\eta^*]$ and~$[2d\eta^*,m]$, where the first interval contains~$\eta^*$ elements from the sequence~$i_1,i_2, \ldots, i_{\ell}$. Call this subsequence~$i_1^1,i_2^1, \ldots, i_{\eta^*}^1$. The second interval contains the other~$\ell'-\eta^*$ elements, call these~$i_1^2, \ldots, i_{\ell'-\eta^*}^2$. None of the two sequences are empty by definition of~$\eta^*$. Applying the induction hypothesis to the first interval, we conclude that the sequence~$i_1^1,\ldots,i_{\eta^*}^1$ is~$(2d\eta^*,\eta^*)$-clustered or -dispersed. Likewise, the sequence~$i_1^2, \ldots, i_{\ell'-\eta^*}^2$ is~$(m-2d\eta^*,\ell-\eta^*)$-clustered or -dispersed. To be able to use this, we observe that~$\lfloor \frac{2d\eta^*}{2\eta^*}\rfloor = \lfloor\frac{m}{2\ell}\rfloor$, and that~$\lfloor \frac{m-2d\eta^*}{2(\ell-\eta^*)}\rfloor \geq \lfloor \frac{2d(\ell-\eta^*)}{2(\ell-\eta^*)}\rfloor  = \lfloor\frac{m}{2\ell}\rfloor$. We distinguish two cases.
         
         First assume that case (b) holds for at least one of the two intervals~$[1,2d\eta^*]$ and~$[2d\eta^*,m]$. 
         Without loss of generality (b) holds for the first interval, so we have integers $\psi, p_1, p_2, \ldots, p_k\in [m]$ and~$\xi , q_1, \ldots, q_k\in \mathds{N}$ that satisfy the conditions (i)-(iv). Now add $[p_{k+1},p_{k+1}+q_{k+1}]=[2d\eta^*,m]$. 
         Going back to the original problem in the interval~$[m]$ and with~$i_1, i_2, \ldots, i_{\ell}$, one can easily check that~$\psi, p_1, p_2, \ldots, p_k, p_{k+1}\in [m]$ and~$\xi , q_1, \ldots, q_k, q_{k+1}\in \mathds{N}$ satisfy conditions (i)-(iv) (here we need that~$\lfloor \frac{2d\eta^*}{2\eta^*}\rfloor = \lfloor\frac{m}{2\ell}\rfloor$ for condition~(iv)). 
         For condition (v), note that from either (a-iii) or (b-v) from the induction hypothesis for the interval~$[2d\eta^*,m]$ we know that there are at least~$2\left((\ell-\eta^*)-(\ell'-\eta^*)\right)$ elements in~$[2d\eta^*,m]$ that are not contained in an interval. 
         From this immediately follows that (b-v) holds for the complete sequence~$i_1,\ldots, i_{\ell'}$. 
        
         Otherwise, only case (a) holds for both the intervals~$[1,2d\eta^*]$ and~$[2d\eta^*,m]$. Then we show that for the original~$[m]$ also condition (a) holds. By the induction basis, we have intervals~$[p_1, p_1+q_1], [p_2,p_2+q_2], \ldots, [p_j, p_j+q_j]$ that are in~$[1,2d\eta^*]$, and intervals~$[p_{j+1}, p_{j+1}+q_{j+1}], \ldots, [p_k, p_k+q_k]$ for the interval~$[2d\eta^*,m]$. But then~$[p_1, p_1+q_1], \ldots, [p_k, p_k+q_k]$ is a sequence of intervals that satisfies (i) and (ii) (again, using the bounds on~$\lfloor \frac{2d\eta^*}{2\eta^*}\rfloor$ and~$\lfloor \frac{m-2d\eta^*}{2(\ell-\eta^*)}\rfloor$ to prove (ii)). Again, (iii) follows immediately from statement (iii) of the induction hypothesis for the interval~$[2d\eta^*,m]$.
    \end{itemize}
    \qed
\end{proof}

Finally, we drop the assumption that~$P(\frac{m_i}{3},m_i,h_{\ell_i'})=(s,h_1)$ for some~$s$, which we made at the start. In general, the memory states visited by the improvement rule do not just form one repeating sequence; if~$\ell_i''$ is such that~$P(\frac{m_i}{3},m_i,h_{\ell_i'})=(s,h_{\ell_i''})$ and the number of improving moves stays~$\frac{m_i}{3}$, then the choices output by~$P$ are
\[
    (g_1,g_2,\ldots,g_{\ell_i''},g_{\ell_i''+1}, \ldots, g_{\ell_i'},g_{\ell_i''},g_{\ell_i''+1},\ldots,g_{\ell_i'},g_{\ell_i''},g_{\ell_i''+1},\ldots)
\]
In this case, we can modify the constructions from \cref{lem:clustered,lem:dispersed}, by simply adding some filler gadgets that account for the one-time switches related to~$g_1,g_2,\ldots, g_{\ell_i''-1}$.
\begin{lemma}
     Let~$\Pi^P$ be an~$\ell$-index-based improvement rule, let~$m_i$ be divisible by 3, let~$\ell_i\coloneqq\ell(m_i)$, and let~$m_i\geq 12\ell_i$. There exists a parity game with~$m_i$ player 0 edges, which takes at least~$2^{\frac{m_i}{12\ell_i}-1}$ iterations to solve for strategy improvement with improvement rule~$\Pi^P$.
\end{lemma}
\begin{proof}
    
Obviously, if the improvement rule started at memory state~$h_{\ell_i''}$, we could make the constructions from the previous cases, and this would let~$\Pi^P$ take a superpolynomial number of iterations. To show that this also works starting from state~$h_1$, we make a small modification to make sure that the first~$\ell_i''-1$ steps do not interfere with the binary counter.

We know by \cref{lem:clustereddispersed} that the sequence~$(g_{\ell_i''}, g_{\ell_i''+1},\ldots, g_{\ell_i'})$ is~$(\frac{m_i}{3},\ell_i)$-clustered or~-dispersed. Suppose that it is~$(\frac{m_i}{3},\ell_i)$-clustered. Then, by property~(iii) from \cref{def:clustered}, we can assume that at least $2\left(\ell-(\ell_i'-\ell_i''+1)\right)$ elements of~$[m]$ are not in an interval. This means that, in our construction from the proof of \cref{lem:clustered}, we have at least~$2\left(\ell-(\ell_i'-\ell_i''+1)\right)$ filler gadgets (from \cref{fig:ControllerFiller}). Call the constructed game~$CG_1$.

Since~$\left(\ell-(\ell_i'-\ell_i''+1)\right)\geq 2(\ell_i''-1)$, this gives us enough space to do the following procedure for~$j\in\{\ell_i''-1,\ell_i''-2,\ldots,1\}$:
\begin{itemize}
    \item Choose two filler gadgets in~$CG_{\ell_i''-j}$ that correspond to numbers of~$[m]$ outside any interval~$[p_c,p_c+q_c]$.
    \item Replace them by the double filler gadget from \cref{fig:gadgetdoublepath}, giving the Bland numbers from one filler gadget to edges~$a_1,a_2,a_3$, and from the other to~$b_1,b_2,b_3$.
    \item Change the Bland number of~$a_1$ such that it is the~$g_{j}$-th improving move\footnote{This will involve some shifting. For example, if~$b$ is the Bland number of the~$(g_{j}-1)$-th improving move, we can give~$a_1$ Bland number~$b+1$, and increase all other Bland numbers above~$b$ by 1.}. Call the game obtained in this manner~$CG_{\ell_i''-j+1}$.
\end{itemize}
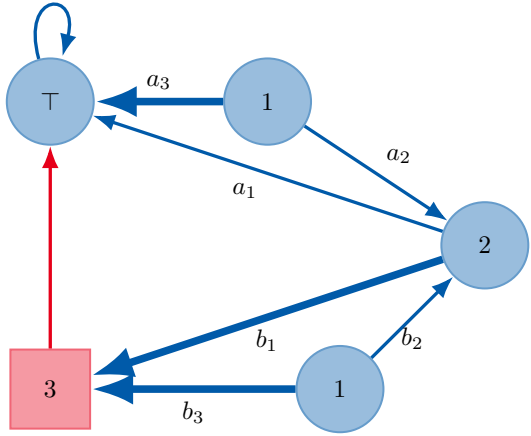
\begin{figure}[tbp]
    \centering
    \resizebox{0.5\linewidth}{!}{%
    \begin{tikzpicture}[node distance = 3cm,
    p0/.style={circle, draw=p0Blue!60, fill=p0Blue!40, thick, minimum size=12mm}, 
    gen/.style={circle, draw=black, fill=black!20, thick, minimum size=12mm}, 
    p0small/.style={circle, draw=p0Blue!60, fill=p0Blue!40, thick, minimum size=6mm}, 
    every path/.style={draw=p0Blue,line width=1.3pt,{>=Latex}}, 
      p1/.style={rectangle, draw=p1Red!60, fill=p1Red!40, thick, minimum size=11mm}, 
      phan/.style={draw=none, fill=none, inner sep=0pt, outer sep=0pt}] 

    \node[p0] (x) at (6,2) {$2$};
    \node[p0] (y) at (3,4) {$1$};
    \node[p1] (z) at (0,0) {$3$};
    \node[p0] (s) at (0,4) {$\top$}; 
    \node[p0] (w) at (4,0) {$1$};
    \draw (s) edge[loop above] (s);
    \draw (y) edge[->] node [above right] {$a_2$} (x);
    \draw (y) edge[->,line width=3pt] node[above] {$a_3$} (s);
    \draw (x) edge[->,line width=3pt] node [below] {$b_1$} (z);
    \draw (x) edge[->] node [below left] {$a_1$} (s);
    \draw (z) edge[->,p1Red] node [] {} (s);
    \draw (w) edge[->] node[below] {$b_2$} (x);
    \draw (w) edge[->, line width=3] node[below] {$b_3$} (z);

\end{tikzpicture}
    }
    \caption{Double filler gadget that deals with non-repeating memory states.}
    \label{fig:gadgetdoublepath}
\end{figure}
Then, by construction, the first~$\ell_i''-1$ iterations of strategy improvement will switch the edges~$a_1$ from these new gadgets. Moreover, after the switch, the Bland numbers of the remaining switches are equivalent to what they were in~$CG_1$. Hence strategy improvement with improvement rule~$\Pi^P$ takes a superpolynomial number of steps to solve the game~$CG_{\ell_i''}$. The case where the sequence~$g_{\ell_i''+1},\ldots, g_{\ell_i'}$ is~$(\frac{m_i}{3},\ell_i)$-dispersed is analogous.
    \qed
\end{proof}

By definition, we have~$\frac{m_i}{\ell_i}\geq 3i\log(m_i)$.
Thus, we conclude 
\[
2^{\frac{m_i}{12\ell_i}-1}\geq 2^{\frac{i}{4}\log(m_i)-1}-1=\frac{1}{2}m_i^{\frac{i}{4}}-1.
\]
Since~$i\to\infty$, this yields that~$\Pi^P$ takes a superpolynomial number of steps in the worst case to solve sink parity games, completing the proof of~\cref{thm:PGlordering}.

Finally, to complete the proof of \cref{thm:blandmemory}, it suffices to show that we can assume that all priorities in the game are unique by \cref{thm:PGreduction}. One can quickly check that all the constructions and the gadgets used are not significantly\footnote{There may be slight changes in the order of non-significant switches due to edge copying in the proof of \cref{lem:clustered}.} affected by the standard transformation. 
This completes the proof of \cref{thm:blandmemory}: we proved that no strongly polynomial index-based pivot rule exists when fewer than~$o(n/ \log n)$ memory states are available. The term `strongly' cannot be left out here, since the LP that is implicitly constructed via \cref{thm:PGreduction} has doubly exponential coefficients. It remains open whether this lower bound extends to rules with larger memory.
\begin{question}
    Is there a super-polynomial lower bound on the running time of index-based pivot rules when at least~$\Omega(n/\log n)$ memory states are available?
\end{question}

\section{Lower bounds for a memoryless subclass of ranking-based rules}
\label{sec:rankings}
Three classical simplex pivot rules are Dantzig's original pivot rule, Bland's rule, and the largest-increase rule, which greedily select the improving variable with the largest reduced cost, the smallest global index, and the largest associated objective improvement, respectively; see Observation~\ref{obs:greedy}.
We use the names \textsc{Bland}, \textsc{Dantzig}, and \textsc{LargestIncrease} to refer both to the pivot rules and to their underlying neighbor rankings, as in Observation~\ref{obs:greedy}.
Conceptually, as we will see in the following, the inefficient worst-case behavior of these rules~\cite{avis_notes_1978,jeroslow_simplex_1973,KleeMinty} arises not merely from their greedy nature but already from the fact that the information they rely on is too limited.
More precisely, our main result in this section is that every~$\{ \textsc{Bland},\allowbreak \textsc{Dantzig},\allowbreak \textsc{LargestIncrease} \}$-ranking-based simplex pivot rule (see Definition~\ref{def:rankingbased}) is at least subexponential.
That is, every deterministic, memoryless simplex pivot rule that bases its decision solely on the number of improving variables and their relative rankings by index, reduced cost, and objective improvement requires subexponential time in the worst-case.
This includes, for example, the pivot rule that deliberately picks the second-best variable according to \textsc{Dantzig} unless it is ranked last by \textsc{Bland}, in which case it picks the best variable according to \textsc{LargestIncrease}.

\begin{theorem}\label{thm:rankingsSimplex}
    Let~$\Pi$ be a~$\{\textup{\textsc{Bland}}, \textup{\textsc{Dantzig}}, \textup{\textsc{LargestIncrease}}\}$-ranking-based pivot rule.
    Then, there exists a family of linear programs with~$N$ non-zero entries in the constraint matrix such that the simplex method with pivot rule~$\Pi$ takes~$\Omega(2^{\sqrt N})$ iterations.
\end{theorem}

Two other classical simplex pivot rules are \textsc{SteepestEdge} and \textsc{Shadow\-Vertex}, see Observation~\ref{obs:greedy}.
They are closely connected~\cite{black2025exponential}, and it is natural to ask whether our results can be extended to include these rules as well.

\begin{question}
    Does there exist a polynomial-time~$\{\textsc{Steepest\-Edge},\allowbreak \textsc{Shadow}\-\textsc{Vertex},\allowbreak \textsc{Bland},\allowbreak \textsc{Dan\-tzig},\allowbreak \textsc{Lar}\-\textsc{gest}\-\textsc{In\-crease}\}$-\allowbreak ranking-based pivot rule?
\end{question}

The remainder of this section is devoted to proving that no polynomial-time $\{ \textsc{Bland},\allowbreak \textsc{Dantzig},\allowbreak \textsc{LargestIncrease} \}$-ranking-based pivot rule exists for \textsc{PolicyIteration}.
Although ranking-based pivot rules were defined only for the simplex method, the definitions carry over directly via the correspondence between \textsc{PolicyIteration} and the simplex method, see Theorem~\ref{thm:MDPreduction}.
Explicitly, a neighbor ranking assigns to each policy of the Markov decision process a total preorder on the set of its improving switches. 
Consequently, the following lower bound for \textsc{Policy\-Iteration} implies Theorem~\ref{thm:rankingsSimplex} since the number of non-zero matrix entries in the induced linear program is linear in the number of transition probabilities in the Markov decision process, see~\cite{puterman1994markov} or~\cite[Sec. 2.4]{hansen2012worst}. 
Both theorems together prove \cref{thm:rankingsSimplexInformal}.

\begin{theorem}\label{thm:rankingsPI}
    Let~$\Pi$ be a~$\{\textup{\textsc{Bland}}, \textup{\textsc{Dantzig}}, \textup{\textsc{LargestIncrease}}\}$-ranking-based pivot rule.
    Then, there exists a family of weak unichain Markov decision processes with~$N$ non-zero transition probabilities such that \textup{\textsc{PolicyIteration}} with~$\Pi$ takes~$\Omega(2^{\sqrt N})$ iterations.
\end{theorem}

In this section, we may assume that all considered pivot rules break ties by a pre-defined (\textsc{Bland}) ordering of the variables, as discussed in Section~\ref{sec:pivotrules}.

Let~$R$ be a fixed neighbor ranking and let~$\Pi$ be a memoryless~$R$-ranking-based rule.
By Definition~\ref{def:rankingbased}, the decisions of~$\Pi$ only depend on the number of improving variables and their relative positions in the ranking according to~$R$; in particular, its decisions are independent of the given linear program or the current basis.
Therefore, associated to~$\Pi$ is a function~$f_\Pi \colon \N \to \N$, where~$f_\Pi(k)\in\{1,2,\ldots,k\}$ for all~$k\in\N$, such that, whenever there are~$k$ improving variables, the pivot rule~$\Pi$ chooses the~$f_\Pi(k)$\textsuperscript{th}-lowest index according to~$R$.
Thus, the behavior of~$\Pi$ is fully determined by the function~$f_\Pi$.
For example, if~$R=\textsc{Dantzig}$ and there are~$k$ improving variables at the current solution, then~$\Pi$ chooses the improving variable with the~$f_\Pi(k)$\textsuperscript{th}-smallest reduced costs; in particular, Dantzig's original pivot rule is the~$\textsc{Dantzig}$-ranking-based rule with~$f_\Pi(k)=k$ for all~$k\in\N$.

We can extend the notion of associated functions a bit beyond ranking-based rules that only rely on a single neighbor ranking.
Let~$\mathcal{R} = \{ R_1, \dots, R_k \}$ be a fixed set of neighbor rankings and let~$\Pi$ be a memoryless~$\mathcal{R}$-ranking-based rule.
Consider the special case where the simplex method is at a basis~$B$ such that all neighbor rankings in~$\mathcal{R}$ induce the same total preorder~$\preceq_{B}$ on the set of improving neighbors~$\Gamma^+(B)$, that is, we have~$\preceq^{R_1}_{B} = \cdots = \preceq^{R_k}_{B} = \preceq_{B}$.
Since~$\Pi$ bases its decisions only on the neighbor rankings in~$\mathcal{R}$, in this situation, the only information that~$\Pi$ has is the number of improving variables and their relative positions in the common preorder~$\preceq_{B}$. 
That is, $\Pi$ effectively only sees a single ranking of the improving variables.
Hence, as before,~$\Pi$ is associated with a function~$f_\Pi \colon \N \to \N$ such that~$f_\Pi(k)\in\{1,2,\ldots,k\}$ and, whenever given~$k$ improving variables and all ranking rules agree, the pivot rule~$\Pi$ chooses the~$f_\Pi(k)$\textsuperscript{th}-lowest index in $\preceq_{B}$.
More formally, for all linear program data~$m,n,A,\vec b,\vec c$ and for all bases~$B$ in which all rankings agree, we have~$\Pi_{m,n}((A,\vec b,\vec c), B)=i$, where $i$ is the index of the~$f_{\Pi}(|\Gamma^{+}(B)|)$-th smallest improving variable in~$\Gamma^{+}(B)$ according to~$\preceq_B$.

Fix some arbitrary~$\{ \textsc{Bland},\allowbreak \textsc{Dantzig},\allowbreak \textsc{LargestIncrease} \}$-ranking-based pivot rule~$\Pi$.
We analyze the running time of \textsc{Policy\-Iteration} with pivot rule~$\Pi$ based on the structure of the function~$f_\Pi$, which determines the choice of~$\Pi$ at policies where all three neighbor rankings agree.
More precisely, depending on the behavior of~$f_{\Pi}$, we construct four different families of Markov decision processes.
Then, we argue that, on each of these families, \textsc{Policy\-Iteration} with pivot rule~$\Pi$ visits a (sub-)exponential number of policies when started at a suitable initial policy.
The common key feature of the constructed families is that the three neighbor rankings agree at each of the visited policies, such that the behavior of~$\Pi$ is fully determined by~$f_\Pi$ during the run of the algorithm.

First, we consider the case~$f_\Pi\equiv1$.
That is, whenever the neighbor rankings \textsc{Bland}, \textsc{Dantzig}, and \textsc{Largest\-Increase} agree, the pivot rule~$\Pi$ chooses the least-preferred improving switch in the common ranking.
We prove an exponential lower bound for such~$\Pi$ in \cref{lem:BDL_with_fequals1}.
The core structure of the constructed family of Markov decision processes is based on a lower bound construction for parity games given in~\cite{MaatThesis}.
As in Section~\ref{sec:memory}, we divide the processes into levels, or \emph{bits}, and prove that the algorithm simulates a \emph{binary counter}.
Since~$f_\Pi$ selects the least-preferred switch, the main feature of our processes is that improving switches in low levels are less significant than those in higher levels, similar to how the bits in a binary counter differ in significance.

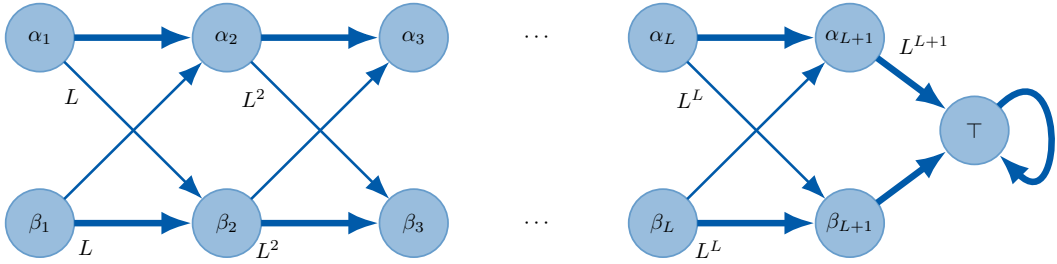
\begin{figure}
    \centering
    \resizebox{\linewidth}{!}{%
    \begin{tikzpicture}[node distance = 3cm,
    p0/.style={circle, draw=p0Blue!60, fill=p0Blue!40, thick, minimum size=11mm,inner sep=0pt, outer sep=0pt}, 
    every path/.style={draw=p0Blue,line width=1.2pt,{>=Latex}}, 
      p1/.style={rectangle, draw=p1Red!60, fill=p1Red!40, thick, minimum size=9mm,inner sep=0pt, outer sep=0pt}, 
      phan/.style={draw=none, fill=none, inner sep=0pt, outer sep=0pt}, 
      ]
    
      \foreach \i in {1,2,3} {
        \pgfmathsetmacro\x{3*(\i - 1)}
        \node[p0] (a\i) at (\x,3) {$\alpha_{\i}$};
        \node[p0] (b\i) at (\x,0) {$\beta_{\i}$};
      }
    
      \foreach \i/\label in {1/L,2/L^2} {
        \pgfmathtruncatemacro{\j}{\i + 1}
        \draw (a\i) edge [-{Latex[length=5mm]}, line width=3pt] (a\j);
        \draw (a\i) edge [-{Latex[length=4mm]}] node [near start, below left=-7pt and 5pt] {$\label$} (b\j);
        \draw (b\i) edge [-{Latex[length=4mm]}] (a\j);
        \draw (b\i) edge [-{Latex[length=5mm]}, line width=3pt] node [near start,below left=3pt and 0pt] {$\label$} (b\j);
      }
    
      \node[phan] at (8,3) {$\cdots$};
      \node[phan] at (8,0) {$\cdots$};
    
      \foreach \i/\label in {4/L, 5/{L+1}} {
        \pgfmathsetmacro\x{3*\i-2}
        \node[p0] (a\i) at (\x,3) {$\alpha_{\label}$};
        \node[p0] (b\i) at (\x,0) {$\beta_{\label}$};
      }
    
      \foreach \i/\label in {4/L} {
        \pgfmathtruncatemacro{\j}{\i + 1}
        \draw (a\i) edge [-{Latex[length=5mm]}, line width=3pt] (a\j);
        \draw (a\i) edge [-{Latex[length=4mm]}] node [near start, below left=-7pt and 4pt] {$\label^{\label}$} (b\j);
        \draw (b\i) edge [-{Latex[length=4mm]}] (a\j);
        \draw (b\i) edge [-{Latex[length=5mm]}, line width=3pt] node [near start,below left=3pt and -3pt] {$\label^{\label}$} (b\j);
      }
    
        \node[p0] (c) at (15,1.5) {$\top$};
        \draw (c) edge [out=45,in=-45, loop, looseness=5,-{Latex[length=5mm]}, line width=3pt] (c);
        \draw (a5) edge [-{Latex[length=5mm]}, line width=3pt] node [near start, above right=3pt and -3pt] {$L^{L+1}$} (c);
        \draw (b5) edge [-{Latex[length=5mm]}, line width=3pt] (c);
    
    \end{tikzpicture}}
    \caption{The Markov decision process~$\mathcal{M}_L$ from the proof of Lemma~\ref{lem:BDL_with_fequals1}. Arc labels denote rewards, actions without a label yield no reward. Bold edges indicate the initial policy~$\sigma_0$.}
    \label{fig:BDL_with_fequals1}
\end{figure}

\begin{lemma}\label[lemma]{lem:BDL_with_fequals1}
    Given a~$\{\textup{\textsc{Bland}}, \textup{\textsc{Dantzig}}, \textup{\textsc{LargestIncrease}}\}$-ranking-based pivot rule~$\Pi$ with~$f_{\Pi}(m)=1$ for all~$m\in\N$, there exists a family of weak unichain Markov decision processes with~$N$ non-zero transition probabilities such that \textup{\textsc{PolicyIteration}} with pivot rule~$\Pi$ takes~$\Omega(2^{N})$ iterations.
\end{lemma}
\begin{proof}
    Fix some~$L\in\N_{\geq 2}$ and consider the Markov decision process~$\mathcal{M}_L$, see Figure~\ref{fig:BDL_with_fequals1}, that consists of~$L+1$ levels such that 
    \begin{itemize}
        \item each level~$\ell\in\{1,2,\ldots,L+1\}$ contains two states, called~$\alpha_\ell$ and~$\beta_\ell$,
        \item each~$\alpha_\ell$ or~$\beta_\ell$ with~$\ell\leq L$ admits two actions, one of which points to~$\alpha_{\ell+1}$ without reward, the other points to~$\beta_{\ell+1}$ yielding a reward of~$L^\ell$, 
        \item~$\alpha_{L+1}$ admits a single action, which yields a reward of~$L^{L+1}$ and points to the sink~$\top$,
        \item~$\beta_{L+1}$ also admits only one action, which points to~$\top$ without reward.
    \end{itemize}
    Note that~$\mathcal{M}_L$ is weak unichain since it is acyclic (if we ignore the sink).

    Consider some~$b\in\{0,1,\ldots,2^L-1\}$.
    For $\ell\in\{1,2\ldots,L\}$, let~$b_\ell\in\{0,1\}$ such that~$b=\sum_{\ell=1}^L b_\ell 2^{\ell-1}$.
    We call a policy~$\sigma$ for~$\mathcal{M}_L$ \emph{canonical for~$b$} if
    \begin{itemize}\label{def_canonical}
        \item~$\sigma(\alpha_\ell) = \sigma(\beta_\ell)$ if and only if~$b_\ell=1$,
        \item~$\val_\sigma(\sigma(\alpha_\ell)) = \max\left\{ \val_\sigma(\alpha_{\ell+1}),\, L^\ell+\val_\sigma(\beta_{\ell+1}) \right\}$,
    \end{itemize}
    for all~$\ell\in\{1,2,\ldots,L\}$.
    By the second condition, there is no improving switch for~$\sigma$ in any~$\alpha_\ell$.
    Hence, by the first condition, there is no improving switch in the states~$\beta_\ell$ with~$b_\ell=1$ either.    
    It is easy to see that the policy~$\sigma_0$ determined by~$\sigma_0(\alpha_\ell)=\alpha_{\ell+1}$ and~$\sigma_0(\beta_\ell)=\beta_{\ell+1}$, for all~$\ell\in\{1,2,\ldots,L\}$, as shown in Figure~\ref{fig:BDL_with_fequals1}, is canonical for~$b=0$.
    
    Given some~$\{\textsc{Bland},\allowbreak \textsc{Dantzig},\allowbreak \textsc{LargestIncrease}\}$-\allowbreak\! ranking-based pivot rule~$\Pi$ with~$f_{\Pi} \equiv 1$, we will prove that \textsc{PolicyIteration} with pivot rule~$\Pi$ and initial policy~$\sigma_0$ visits a canonical policy for each~$b\in\{0,1,\ldots,2^L-1\}$, by induction on~$b$.
    The induction basis, where~$b=0$, is trivial by our choice of the initial policy.
    So assume that \textsc{PolicyIteration} with pivot rule~$\Pi$ visits a policy~$\sigma$ that is canonical for some fixed~$b\in\{0,1,\ldots,2^L-2\}$.
    
    Since~$\sigma$ is canonical for~$b$, all improving switches for~$\sigma$ must be in states~$\beta_\ell$ with~$b_\ell=0$.
    Observe that, for every policy~$\tau$ and for all~$\ell\in\{1,2,\ldots,L\}$,
    \begin{equation}
    \vert \val_{\tau}(\alpha_{\ell+1}) - \left( L^\ell+\val_{\tau}(\beta_{\ell+1}) \right) \vert \geq L^\ell,\label{eq:fequals1_mindiff}      
    \end{equation}
    because all possible values of~$\alpha_{\ell+1}$ and~$\beta_{\ell+1}$ are multiples of~$L^{\ell+1}$. 
    In particular, as the difference in \eqref{eq:fequals1_mindiff} is nonzero, the states~$\beta_\ell$ with~$b_\ell=0$ are exactly the states with improving switches for~$\sigma$.

    Let~$\Pi_{\textsc{Bland}}$ be a~$\{\textsc{Bland}\}$-ranking-based rule with~$f_{\Pi_{\textsc{Bland}}}(m)=1$ for all~$m\in\N$.
    Since, by definition, \textsc{Bland} selects the improving switch with the \emph{smallest} index, the pivot rule~$\Pi_{\textsc{Bland}}$ always picks the improving switch with the \emph{largest} index in a pre-defined ordering of the actions.    
    We will show that \textsc{PolicyIteration} with pivot rule~$\Pi_{\textsc{Bland}}$ transforms~$\sigma$ into a canonical policy for~$b+1$, and afterwards argue that~$\Pi$ mimics this behaviour of~$\Pi_{\textsc{Bland}}$.
    Let the \textsc{Bland} ranking of the actions in~$\mathcal{M}_L$ be such that~$\Pi_{\textsc{Bland}}$ prefers improving switches in~$\alpha_\ell$ over switches in~$\beta_\ell$, and switches in~$\alpha_\ell$ or~$\beta_\ell$ over switches in~$\alpha_k$ or~$\beta_k$, for all~$\ell\in\{1,2,\ldots,L\}$ and~$k\in\{\ell+1,\ell+2,\ldots,L+1\}$.

    By definition of the \textsc{Bland} ranking, when given the canonical policy~$\sigma$, the pivot rule~$\Pi_{\textsc{Bland}}$ applies the improving switch in the state~$\beta_j$ that satis\-fies~$j=\min\{\ell\in\{1,2,\ldots,L\}\colon\allowbreak b_\ell=0\}$.
    If~$j=1$, then the resulting policy is already canonical for~$b+1$.
    Note that this is the case if~$b$ is even, where~$b$ and~$b+1$ only differ in the first bit.
    Now assume~$j>1$.

    Since~$\sigma$ is canonical for~$b$, we have~$\sigma(\beta_{j-1}) = \sigma(\alpha_{j-1})$ and~$\sigma(\beta_{j}) \neq \sigma(\alpha_{j})$.
    Further, there is no improving switch for~$\sigma$ in~$\{\alpha_{\ell}\,\colon\, \ell\in\{1,2,\ldots,L\}\}$, so equation~\eqref{eq:fequals1_mindiff} yields~$\val_\sigma(\alpha_j) - \val_\sigma(\beta_j) \geq L^j > L^{j-1}$.
    Therefore, we obtain that~$\sigma(\beta_{j-1}) = \sigma(\alpha_{j-1}) = \alpha_j$, and thus~$\val_\sigma(\beta_{j-1}) - \val_\sigma(\alpha_{j-1}) = L^{j-1}$.
    By analogous arguments we obtain
    \begin{equation}\label{eq:fequals1_sigma}
        \sigma(\beta_{\ell}) = \sigma(\alpha_{\ell}) = \beta_{\ell+1}, \qquad \text{for all } \ell\in\{1,2,\ldots,j-2\},
    \end{equation}    
    where we use the notion that~\eqref{eq:fequals1_sigma} is the empty statement if~$j=2$.

    The policy~$\tau$ obtained from~$\sigma$ by switching~$\beta_j$ satisfies~$\val_{\tau}(\alpha_j)=\val_{\tau}(\beta_j)$.
    Due to the behavior of~$\sigma$ in the first~$j-1$ levels, this yields that the next improving switches that~$\Pi_{\textsc{Bland}}$ applies are in~$\alpha_{j-1}$,~$\alpha_{j-2}$,$\,\ldots\,$,~$\alpha_{2}$, and~$\alpha_{1}$ -- in this order.
    That is, starting at the policy~$\sigma$, \textsc{PolicyIteration} with~$\Pi_{\textsc{Bland}}$ applies exactly one switch in each of the first~$j$ levels, and the resulting policy~$\sigma'$ does not admit an improving switch in any~$\alpha_\ell$ with~$\ell\in\{1,2,\ldots,j\}$.
    This yields that~$\sigma'$ is canonical for~$b+1$, which concludes the induction step for~$\Pi_{\textsc{Bland}}$.

    It remains to show that~$\Pi$ mimics the behavior of~$\Pi_{\textsc{Bland}}$.
    Let~$\mathcal S$ denote the sequence of policies that~$\Pi_{\textsc{Bland}}$ visits during the transformation of~$\sigma$ into~$\sigma'$, where~$\sigma\in\mathcal S$ is included and~$\sigma'\not\in\mathcal S$ is excluded.
    Fix some~$\sigma''\in\mathcal{S}$ and let~$\preceq^{\textsc{B}}_{\sigma''}$,~$\preceq^{\textsc{D}}_{\sigma''}$, and~$\preceq^{\textsc{LI}}_{\sigma''}$ denote the total preorders induced by \textsc{Bland}, \textsc{Dantzig}, and \textsc{LargestIncrease}, respectively, on the set of all improving switches for~$\sigma''$.
    If~$\preceq^{\textsc{B}}_{\sigma''}\, =\, \preceq^{\textsc{D}}_{\sigma''}\, =\, \preceq^{\textsc{LI}}_{\sigma''}$, then~$\Pi$ picks the same improving switch for~$\sigma''$ as~$\Pi_{\textsc{Bland}}$ since~$f_{\Pi}(m) = f_{\Pi_{\textsc{Bland}}}(m)=1$ for all~$m\in\N$.
    Thus, for the induction step, it suffices to prove~$\preceq^{\textsc{B}}_{\sigma''}\, =\, \preceq^{\textsc{D}}_{\sigma''}\, =\, \preceq^{\textsc{LI}}_{\sigma''}$ for the arbitrarily chosen~$\sigma''\in\mathcal{S}$.

    For every~$\ell\in\{1,2,\ldots,L\}$, we define
    \[
    \lambda(\ell)\coloneqq\min\left\{ k\in\{ \ell+1,\ell+2,\ldots,L+1\}\colon\, b_k=1 \text{ or } k=L+1 \right\}
    \]
    and
    \[
    \mu(\ell)\coloneqq \max\left\{ k\in\{0,\ldots,\ell\}\colon\, b_{\ell-i}=0,\,\forall i\in\{0,\ldots,k-1\} \right\} \geq 0.
    \]     
    We consider the case~$\sigma''=\sigma\in\mathcal{S}$ first.
    Recall that the improving switches for the canonical policy~$\sigma$ are the switches in the states~$\beta_\ell$ with~$b_\ell=0$.
    Fix some arbitrary~$\ell_0\in\{1,2,\ldots,L\}$ with~$b_{\ell_0}=0$, and let~$\lambda\coloneqq\lambda(\ell_0)$.
    
    By definition of~$\lambda$, we have~$b_\lambda = 1$ or~$\lambda=L+1$.
    In both cases, as~$\sigma$ is canonical for~$b$, we get~$\sigma(\alpha_{\lambda}) = \sigma(\beta_{\lambda})$.
    Thus, since there is no improving switch in~$\alpha_{\lambda-1}$ and since~$b_{\lambda-1}=0$, we have~$\sigma(\alpha_{\lambda-1})=b_{\lambda}$ and~$\sigma(\beta_{\lambda-1})=\alpha_{\lambda}$.
    Hence~$\val_\sigma(\alpha_{\lambda-1}) - \val_\sigma(\beta_{\lambda-1}) = L^{\lambda-1}$.
    Additionally, we have~$b_\ell=0$ for all~$\ell\in\{ \ell_0,\ell_0+1,\dots,\lambda-2\}$, which gives us~$\sigma(\beta_{\ell}) \neq \sigma(\alpha_{\ell}) = \alpha_{\ell+1}$ for these values of~$\ell$.
    Thus, starting from $\beta_{\ell_0}$ one gains the rewards $L^{\ell_0},L^{\ell_0+1},\ldots,L^{\lambda_0-2}$ before reaching $\beta_{\lambda-1}$, while from $\alpha_{\ell_0}$ one gains no reward before reaching~$\alpha_{\lambda-1}$. 
    Therefore, we obtain that the improving switch in the state~$\beta_{\ell_0}$ has reduced costs of~$\rc(\beta_{\ell_0}) \coloneqq L^{\lambda-1}-\sum_{\ell=\ell_0}^{\lambda-2} L^\ell \in \{ L^{\ell_0},L^{\ell_0}+1,\ldots,L^{\lambda-1} \}$.

    Next, we consider the objective improvement induced by this switch.
    Recall that the objective in Markov decision processes is given by the sum over all state values.
    As the game~$\mathcal{M}_L$ is ayclic, the switch in~$\beta_{\ell_0}$ increases the value of~$\beta_{\ell_0}$ by~$\rc(\beta_{\ell_0})$.
    In particular, it increases the values of all states from which~$\sigma$ reaches~$\beta_{\ell_0}$ by~$\rc(\beta_{\ell_0})$.
    
    Since~$\sigma(\alpha_{\ell_0})\neq \sigma(\beta_{\ell_0})$ equation~\eqref{eq:fequals1_mindiff} yields~$\val_\sigma(\alpha_{\ell_0}) - \val_\sigma(\beta_{\ell_0}) \geq L^{\ell_0}$.
    By an easy inductive argument, this implies~$\sigma(\alpha_{\ell_0-i})=\alpha_{\ell_0-i+1} \neq \sigma(\beta_{\ell_0-i})$ for all~$i\in\{1,2,\ldots,\mu(\ell_0)-1\}$.
    Thus, if~$\mu(\ell_0)=\ell_0$, then~$\sigma$ reaches~$\beta_{\ell_0}$ from the~$\mu(\ell_0)-1\geq 0$ states in~$\{\beta_\ell\,\colon\,\ell\in\left\{1,2,\ldots,\ell_0-1\right\}\}$.
    Otherwise, we have~$b_{\ell_0-\mu(\ell_0)}=1$, which yields~$\sigma(\beta_{\ell_0-\mu(\ell_0)}) = \sigma(\alpha_{\ell_0-\mu(\ell_0)}) = \alpha_{\ell_0-\mu(\ell_0)+1}$.
    So, again, the policy~$\sigma$ reaches~$\beta_{\ell_0}$ from~$\mu(\ell_0)-1\geq 0$ other states, namely the states in the set~$\{\beta_\ell\,\colon\,\ell\in\left\{\ell_0-\mu(\ell_0)+1,\ell_0-\mu(\ell_0)+2,\ldots,\ell_0-1\}\right\}$.
    We conclude that the objective improvement of the switch in~$\beta_{\ell_0}$ is given by~$\mu(\ell_0) \cdot \rc(\beta_{\ell_0})$.

    Fix some arbitrary~$k\in\{\ell_0+1,\ell_0+2,\ldots,L\}$ with~$b_k=0$, if such a~$k$ exists.
    We obtain~$\preceq^{\textsc{B}}_{\sigma}\, =\, \preceq^{\textsc{D}}_{\sigma}\, =\, \preceq^{\textsc{LI}}_{\sigma}$ if we can show that the reduced costs and the objective improvement of the switch in~$\beta_k$ are larger than the corresponding values in~$\beta_{\ell_0}$.
    
    First, assume~$k\leq\lambda-1$.
    Then, the reduced costs of the switch in~$\beta_k$ are
    \[
    \rc(\beta_k)\coloneqq L^{\lambda-1}-\sum_{\ell=k}^{\lambda-2} L^\ell > L^{\lambda-1}-\sum_{\ell=\ell_0}^{\lambda-2} L^\ell = \rc(\beta_{\ell_0}).
    \]
    Further, we have~$\mu(k)>\mu(\ell_0)$, which yields~$\mu(k)\cdot\rc(\beta_k) > \mu(\ell_0) \cdot \rc(\beta_{\ell_0})$.
    Hence, the reduced costs and the objective improvement of the switch in~$\beta_k$ are both larger than those of the switch in~$\beta_{\ell_0}$.
    
    Now assume~$k\geq\lambda$, which yields~$\lambda<L+1$.
    Then, we have~$b_k=0\neq b_\lambda$ and thus~$k\neq\lambda$.
    Further,
    \[
    \rc(\beta_k) = L^{\lambda(k)-1}-\sum_{\ell=k}^{\lambda(k)-2} L^\ell \geq L^{k} > L^{\lambda-1}-\sum_{\ell=\ell_0}^{\lambda-2} L^\ell = \rc(\beta_{\ell_0}).
    \] 
    We have~$\mu(k)\geq 1$ and~$k>\lambda$ and~$\mu(\ell_0)\leq L$.
    Therefore, 
    \[
    \mu(k)\rc(\beta_k) \geq L^{k} > \mu(\ell_0)\cdot \left( L^{\lambda-1}-\sum_{\ell=\ell_0}^{\lambda-2} L^\ell \right) = \mu(\ell_0)\rc(\beta_{\ell_0}).
    \]
    We conclude~$\preceq^{\textsc{B}}_{\sigma''}\, =\, \preceq^{\textsc{D}}_{\sigma''}\, =\, \preceq^{\textsc{LI}}_{\sigma''}$ for the case~$\sigma''=\sigma$.
    Therefore,~$\Pi$ also applies the improving switch in~$\beta_j$ to~$\sigma$.
    Now let~$\sigma''$ be the policy that we obtain from~$\sigma$ by switching~$\beta_j$.
    
    The switch in~$\beta_j$ does not affect the reduced costs of the remaining improving switches for~$\sigma$ since all of them are in higher levels.
    Further, it decreases the value of~$\mu(\ell)$ by~$1$ for all levels~$\ell\in\{j+1,j+2,\ldots,\lambda(j)-1\}$.
    This decreases the objective improvement of the switches in these levels, but it does not change their relative ranking by objective improvement.
    The objective improvement of switches above level~$\lambda(j)-1$ remains unchanged.
    In particular, the switch in~$\beta_j$ does not increase the objective improvement of the remaining improving switches above level~$j$, and it does not change their relative ranking by objective improvement.   
    
    However, the switch creates two new improving switches, which are the switches in the states~$\alpha_{j-1}$ and~$\beta_{j-1}$.
    Recall that~$\sigma''(\beta_{j-1}) = \sigma''(\alpha_{j-1}) = \alpha_j$ and~$\sigma''(\beta_{\ell}) = \sigma''(\alpha_{\ell}) = \beta_{\ell+1}$ for all~$\ell\in\{1,2,\ldots,j-2\}$, see~\eqref{eq:fequals1_sigma}.

    Therefore, the switches in the states~$\alpha_{j-1}$ and~$\beta_{j-1}$ have reduced costs of~$\rc(\alpha_{j-1}) = \rc(\beta_{j-1}) = L^{j-1}$.
    At this point, we could assume that~$\Pi$ breaks ties by using \textsc{Bland}.
    Instead, to give a proof of this lemma that does not rely on a specific tie-breaking, we modify the binary counter~$\mathcal{M}_L$ a little.
    We replace every action~$(\alpha_\ell,x)\in\mathcal{M}_L$, where~$\ell\in\{1,2,\ldots,L\}$ and~$x\in\{\alpha_{\ell+1},\beta_{\ell+1}\}$, by a randomized action that starts in~$\alpha_\ell$ and leads with probability~$1-\epsilon$ to the state~$x$ and with probability~$\epsilon$ to the sink~$\top$, where~$\epsilon>0$ is chosen small enough.
    Note that this transformation alters values in~$\mathcal{M}_L$ at most by an amount that is proportional to~$\epsilon$.
    Since all values were integral before and since there are no ties in \textsc{Bland} or during the switch in~$\beta_j$, this modification of~$\mathcal{M}_L$ does not affect the previous analysis if we choose~$\epsilon$ small enough.

    This modification scales the reduced costs~$\rc(\alpha_{j-1})$ of the switch in~$\alpha_{j-1}$ by a factor of~$1-\epsilon$, that is, we have~$\rc(\alpha_{j-1}) = (1-\epsilon)L^{j-1} <  L^{j-1} = \rc(\beta_{j-1})$.
    By the structure of~$\sigma''$ in the first~$j-1$ levels,~$\sigma''$ does not reach~$\alpha_{j-1}$ from any other state.
    Hence, the objective improvement of the switch in~$\alpha_{j-1}$ equals its reduced cost.
    Further, the objective improvement of the switch in~$\beta_{j-1}$ is~$(2\cdot\max\{0,j-2\}+1)\rc(\beta_{j-1}) < 2L\rc(\beta_{j-1}) < L^{j+1}$, while the reduced costs of switches in states above level~$j$ are at least~$L^{j+1}$.
    Thus, we obtain that~$\preceq^{\textsc{B}}_{\sigma''}\, =\, \preceq^{\textsc{D}}_{\sigma''}\, =\, \preceq^{\textsc{LI}}_{\sigma''}$.
    Therefore,~$\Pi$ again mimics~$\Pi_{\textsc{Bland}}$ and it applies the switch in~$\alpha_{j-1}$ to~$\sigma''$.

    Finally, let~$\sigma''$ denote the policy that results from~$\sigma$ by applying the improving switches in~$\beta_j$ and~$\alpha_{j-1}$,~$\alpha_{j-2}$,~$\ldots$,~$\alpha_{j-k}$ for some~$k\in\{1,2,\ldots,j-2\}$.
    Analogous arguments as above yield that the reduced costs and the objective improvement of the switch in~$\alpha_{j-k-1}$ are smaller than the corresponding values of the switch in~$\beta_{j-k-1}$, which are smaller than the corresponding values of the switch in~$\beta_{j-k}$.    
    This yields that~$\Pi$ also applies the improving switch in~$\alpha_{j-k-1}$ next. 
    We obtain that~$\Pi$ mimics~$\Pi_{\textsc{Bland}}$ until reaching the canonical policy for~$b+1$, which concludes the induction.
    
    We obtain that it takes \textsc{PolicyIteration} with pivot rule~$\Pi$ at least~$2^L-1$ iterations to find the optimum.
    Since all~$4L+3$ actions in~$\mathcal{M}_L$ are deterministic, we also have~$4L+3$ non-zero transition probabilities, so we obtain the desired exponential lower bound for~$\Pi$ on the family~$(\mathcal{M}_L)_{L\in\N_{\geq 2}}$.
    \qed
\end{proof}

Second, we prove a subexponential lower bound for the case~$f_{\Pi}(k)=o(\sqrt{k})$ in \cref{lem:Sqrt}.
The key observation for the proof is the following: if~$\mathcal{M}$ is one of our processes from the first case, where~$f_\Pi\equiv1$, and we \emph{copy} each of its~$n$ actions~$n$ times, then we obtain a new process~$\mathcal{M}'$ with~$n'=n^2$ actions such that switching the least-preferred action in~$\mathcal{M}$ corresponds to switching one of its~$n=\sqrt{n'}$ copies in~$\mathcal{M}'$.
Thus, if~$n$ is large enough, the condition~$f_{\Pi}(k)=o(\sqrt{k})$ yields that~$\Pi$ on the new process~$\mathcal{M}'$ mimics the behavior from the first case.
Since the number~$n'$ of actions in~$\mathcal{M}'$ is quadratic in~$n$, we obtain a subexponential lower bound of~$\Omega(2^n)=\Omega(2^{\sqrt{n'}})$.

\begin{lemma}\label[lemma]{lem:Sqrt}
    Given a~$\{\textup{\textsc{Bland}}, \textup{\textsc{Dantzig}}, \textup{\textsc{LargestIncrease}}\}$-ranking-based pivot rule~$\Pi$ with~$f_{\Pi}(m)=o(\sqrt{m})$, there exists a family of weak unichain Markov decision processes with~$N$ non-zero transition probabilities such that \textup{\textsc{PolicyIteration}} with pivot rule~$\Pi$ takes~$\Omega(2^{\sqrt N})$ iterations.
\end{lemma}
\begin{proof}
By~$f_\Pi(m)=o(\sqrt m)$, there exists~$m_0\in\N$ such that~$f_\Pi(m)\leq\sqrt m$ for all~$m\geq m_0$.
Let~$\mathcal{M}'$ be one of the Markov decision processes from the proof of Lemma~\ref{lem:BDL_with_fequals1} such that there are at least~$m_0$ actions in~$\mathcal{M}'$.
We denote the number of actions in~$\mathcal{M}'$ by~$k\geq m_0$.

Let~$\mathcal{M}$ be the Markov decision process that we obtain from~$\mathcal{M}'$ by copying all actions, except the action in the sink,~$k$ times.
That is, for every action~$(x,y)$ in~$\mathcal{M}'$, which is not the action in the sink and which yields a reward of~$\rew(x,y)$, we introduce~$k$ artificial states in~$\mathcal{M}$, add a deterministic action with reward~$\rew(x,y)$ from~$x$ to each of these new states, add a deterministic action without reward from each of the new states to~$y$, and finally delete the action~$(x,y)$.
Observe that~$\mathcal{M}$ is acyclic and, thus, weak unichain.

Now let~$\sigma$ denote a non-optimal policy visited during a run of~$\Pi$ on~$\mathcal{M}$, and let~$m_\sigma$ denote the number of improving switches for~$\sigma$. 
Note that~$m_\sigma$ is an integer multiple of~$k$ since, for each action in~$\mathcal{M}'$, either all or none of its copies are improving in~$\mathcal{M}$.
Furthermore, at most~$k(k-1)$ actions can, potentially, be improving in~$\mathcal{M}$.
Thus, we obtain~$m_0\leq k\leq m_\sigma\leq k^2$, which yields~$f_\Pi(m_\sigma)\leq\sqrt{m_\sigma}\leq k$.

Let~$\sigma'$ denote the policy for~$\mathcal{M}'$ that \emph{corresponds} to~$\sigma$, that is, an action is active in~$\sigma'$ if and only if one of its copies is active in~$\sigma$.
Then, since all copies of an action have the same reduced costs and objective improvement (with respect to~$\sigma$) as the original action (with respect to~$\sigma'$), and since we use \textsc{Bland} as a tiebreaker, we obtain that the three rankings still coincide in~$\sigma$.
Due to~$f_\Pi(m_\sigma)\leq k$, the rule~$\Pi$ selects one of the~$k$ least-preferred improving switches for~$\sigma$, so it chooses a copy of the least-preferred improving switch for~$\sigma'$.

Therefore, the run of~$\Pi$ on~$\mathcal{M}$ mimics the run of the rules analyzed in Lemma~\ref{lem:BDL_with_fequals1} on~$\mathcal{M}'$ if we choose corresponding initial policies.
Since~$\mathcal{M}'$ has~$k$ actions and all~$1+2k(k-1)$ actions in~$\mathcal{M}$ are deterministic, we obtain the desired subexponential lower bound.
    \qed
\end{proof}

Third, we consider the case~$f_{\Pi}(k)=k$, for all~$k\in\N$, and give an exponential lower bound in \cref{lem:BDL_with_fequalsm}.
Our proof strengthens the exponential lower bound of~\cite{disser_unified_2023} by ensuring that all three neighbor rankings agree at every visited policy, which is necessary for our~$f_\Pi$-based analysis.
The main approach for our construction is to take the processes from the first case, where~$f_\Pi\equiv1$, and replace each action by a \emph{gadget} that was introduced in~\cite{disser_unified_2023}.
Informally, this gadget allows us to scale the reduced costs and objective improvements of improving switches at the expense of introducing additional improving switches, whose influence must be controlled in the construction.
By choosing the parameters in these gadgets such that switches in lower levels have larger reduced costs and associated objective improvements than those in higher levels, we obtain the desired exponential bound.

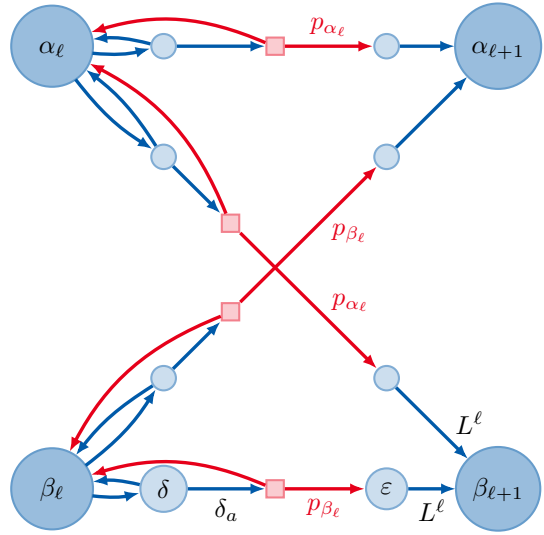
\begin{figure}
    \centering
    \resizebox{.5\linewidth}{!}{%
    \begin{tikzpicture}[node distance = 3cm,
    p0/.style={circle, draw=p0Blue!60, fill=p0Blue!40, thick, minimum size=11mm,inner sep=0pt, outer sep=0pt}, 
    every path/.style={draw=p0Blue,line width=1.3pt,{>=Latex}}, 
      p1/.style={rectangle, draw=p1Red!60, fill=p1Red!40, thick, minimum size=9mm,inner sep=0pt, outer sep=0pt}, 
      phan/.style={draw=none, fill=none, inner sep=0pt, outer sep=0pt}, 
      ]

    \node[p0] (a1) at (0,6) {$\begin{array}{c}\alpha_\ell \end{array}$};
    \node[p0] (b1) at (0,0) {$\begin{array}{c}\beta_\ell \end{array}$};
    \node[p0] (a2) at (6,6) {$\begin{array}{c}\alpha_{\ell+1} \end{array}$};
    \node[p0] (b2) at (6,0) {$\begin{array}{c}\beta_{\ell+1} \end{array}$};
    
    \LIGadget[above]{a1}{a2}{}{$p_{\alpha_\ell}$}{}{0.5}{}{}
    \LIGadget[above right=-8pt and 5pt]{a1}{b2}{$L^\ell$}{$p_{\alpha_\ell}$}{}{0.4}{}{}
    \LIGadget[below right=-8pt and 5pt]{b1}{a2}{}{$p_{\beta_\ell}$}{}{0.4}{}{}
    \LIGadget[below]{b1}{b2}{$L^\ell$}{$p_{\beta_\ell}$}{$\delta_a$}{0.5}{$\delta$}{$\epsilon$}
\end{tikzpicture}}
    \caption{Level~$\ell\in\{1,2,\ldots,L\}$ of the Markov decision process~$\mathcal{M}'_L$ from the proof of Lemma~\ref{lem:BDL_with_fequalsm}.
    Red squares indicate randomized actions; they should not be confused with states.
    Labels of arcs starting in squares denote probabilities.
    We omit most state labels~$\delta(\cdot,\cdot)$,~$\epsilon(\cdot,\cdot)$ and action lables~$\delta_a(\cdot,\cdot)$, and write~$\delta=\delta(\beta_\ell,\beta_{\ell+1})$,~$\epsilon=\epsilon(\beta_\ell,\beta_{\ell+1})$, and~$\delta_a=\delta_a(\beta_\ell,\beta_{\ell+1})$.
    Note that~$\delta_a$ is the action's name, and not, as usual, a reward.}
    \label{fig:BDL_with_fequalsn}
\end{figure}

\begin{lemma}\label[lemma]{lem:BDL_with_fequalsm}
    Given a~$\{\textup{\textsc{Bland}}, \textup{\textsc{Dantzig}}, \textup{\textsc{LargestIncrease}}\}$-ranking-based pivot rule~$\Pi$ with~$f_{\Pi}(m)=m$ for all~$m\in\N$, there exists a family of weak unichain Markov decision processes with~$N$ non-zero transition probabilities such that \textup{\textsc{PolicyIteration}} with pivot rule~$\Pi$ takes~$\Omega(2^{N})$ iterations.
\end{lemma}
\begin{proof}
    Let~$\mathcal{M}_L$ be the Markov decision process described in the first paragraph of the proof of Lemma~\ref{lem:BDL_with_fequals1}, as drawn in Figure~\ref{fig:BDL_with_fequals1}.
    Further, let~$\mathcal{M}'_L$ denote the Markov decision process that we obtain from~$\mathcal{M}_L$ as follows\footnote{An analogous construction was used in~\cite{disser_unified_2023}.}: for every pair of two states~$x$ and~$y$ such that~$x\in\{\alpha_\ell,\beta_\ell\}$ and~$y\in\{\alpha_{\ell+1},\beta_{\ell+1}\}$ for some~$\ell\in\{1,2,\ldots,L\}$,
    \begin{itemize}
        \item add two states, denoted by~$\delta(x,y)$ and~$\epsilon(x,y)$,
        \item add a deterministic action without reward from~$x$ to~$\delta(x,y)$,
        \item add a deterministic action without reward from~$\delta(x,y)$ to~$x$,
        \item add a probabilistic action, called~$\delta_a(x,y)$, without reward that starts in $\delta(x,y)$, and leads to~$\epsilon(x,y)$ with a probability of~$p_x\in(0,1)$ and to~$x$ with a probability of~$(1-p_x)$, where~$p_x$ only depends on the state~$x$,
        \item add a deterministic action from~$\epsilon(x,y)$ to~$y$ that yields the same reward as the action~$(x,y)$,
        \item delete the action~$(x,y)$.
    \end{itemize}
    Consider Figure~\ref{fig:BDL_with_fequalsn} for a drawing of one level of~$\mathcal{M}'_L$.
    The probabilities~$p_x$ are specified later.
    An action is called \emph{incident} to state~$x$ if it starts in~$x$, or leads directly to~$x$ with positive probability.
    
    In the proof of Lemma~\ref{lem:BDL_with_fequals1} we argued that the~$\{\textsc{Bland}\}$-ranking-based rule~$\Pi_{\textsc{Bland}}$ with~$f_{\Pi_{\textsc{Bland}}}\equiv 1$ is exponential on~$\mathcal{M}_L$ if, for every pair of indices~$\ell,k\in\{1,2,\ldots,L\}$ with~$\ell<k$, it prefers improving switches in~$\alpha_\ell$ over those in~$\beta_\ell$, and prefers any switch in~$\alpha_\ell$ or~$\beta_\ell$ over any switch in~$\alpha_k$ or~$\beta_k$.    
     
    Analogously, we consider a \textsc{Bland} ranking on the actions of~$\mathcal{M}'_L$ such that, for all~$\ell<k$, the~$\{\textsc{Bland}\}$-ranking-based rule~$\Pi'_{\textsc{Bland}}$ with~$f_{\Pi'_{\textsc{Bland}}}(m)\equiv m$ prefers switches incident to~$\alpha_\ell$ over those incident to~$\beta_\ell$, and switches incident to~$\alpha_\ell$ or~$\beta_\ell$ over those incident to~$\alpha_k$ or~$\beta_k$.\footnote{For notational convenience, the symbols~$\alpha_\ell$ and~$\beta_\ell$ are used both for states in~$\mathcal{M}_L$ and in~$\mathcal{M}'_L$. The intended meaning is always clear from context.}
    Note that such a ranking exists since every action that might become improving is incident to exactly one state in~$\bigcup_{\ell=1}^L\{\alpha_\ell,\, \beta_\ell\}$. The relative ranking of actions incident to the same state in~$\bigcup_{\ell=1}^L\{\alpha_\ell,\, \beta_\ell\}$ is irrelevant, as we will see that at most one of them is improving at any time.
    We fix these \textsc{Bland} rankings on~$\mathcal{M}_L$ and~$\mathcal{M}'_L$ for the rest of this proof.
 
    We also establish a connection between policies for~$\mathcal{M}_L$ and those for~$\mathcal{M}'_L$.
    A policy~$\sigma'$ for~$\mathcal{M}'_L$ is said to \emph{correspond} to a policy~$\sigma$ for~$\mathcal{M}_L$ if we have
    \begin{equation}\label{eq:correspond}
      \sigma(x) = y \quad\Longleftrightarrow\quad 
      \left[\sigma'(x) = \delta(x,y) \ \text{and}\ \delta_a(x,y)\ \text{is active in}\ \sigma'\right],      
    \end{equation}
    for every pair of states~$x,y\in\bigcup_{\ell=1}^{L+1}\{\alpha_\ell,\, \beta_\ell\}$.
    In other words, both policies move from~$x$ to~$y$ with probability one, collecting the same reward along the way.
    Consequently, we have
    \begin{equation}\label{eq:sameValues}
    \val_\sigma(x) = \val_{\sigma'}(x), \qquad \text{for all~} x\in\bigcup\nolimits_{\ell=1}^{L+1}\{\alpha_\ell,\, \beta_\ell\}.
    \end{equation}
    It is not hard to see that the optimum policy for~$\mathcal{M}'_L$ corresponds to the optimum policy for~$\mathcal{M}_L$.
    Thus, since every policy for~$\mathcal{M}'_L$ that corresponds to a policy for~$\mathcal{M}_L$ is weak unichain, the Markov decision process~$\mathcal{M}'_L$ is weak unichain.

    Fix an arbitrary policy~$\sigma$ for~$\mathcal{M}_L$ and a corresponding policy~$\sigma'$ for~$\mathcal{M}'_L$.
    Let~$k\in\{1,2,\ldots,L\}$ be the level and~$x\in\{\alpha_k,\, \beta_k\}$ be the state of~$\mathcal{M}_L$ in which \textsc{Bland} applies the improving switch for~$\sigma$, and denote the policy resulting from this switch by~$\Bar\sigma$.
    We will argue that \textsc{Bland} on~$\mathcal{M}'_L$ transforms~$\sigma'$ into a policy that corresponds to~$\Bar{\sigma}$ via a sequence of one or more improving switches. 
    Since~$\sigma$ and~$\sigma'$ were chosen arbitrarily, Lemma~\ref{lem:BDL_with_fequals1} then yields an exponential lower bound for \textsc{Bland} on~$\mathcal{M}'_L$ if we choose the initial policy corresponding to the initial policy from Lemma~\ref{lem:BDL_with_fequals1}; note that the number of actions in~$\mathcal{M}_L'$ is linear in the number of actions in~$\mathcal{M}_L$, where all probabilistic actions only introduce two non-zero transition probabilities.

    Let~$y,z\in\{ \alpha_{k+1}, \beta_{k+1} \}$ such that~$\sigma(x)=y\neq z$.
    Since the switch in~$x$ is improving, we have~$\rew(x,z) + \val_\sigma(z) > \rew(x,y) + \val_\sigma (y)$.
    Thus, by~\eqref{eq:sameValues}, we get
    \[
    \val_{\sigma'}(\epsilon(x,z)) = \rew(x,z) + \val_{\sigma'}(z) > \rew(x,y) + \val_{\sigma'}(y) = \val_{\sigma'}(\epsilon(x,y)).
    \]
    Hence, if~$\delta_a(x,z)$ is not already active, then it is improving for~$\sigma'$.
    Otherwise, the inactive action~$(x,\delta(x,z))$ is improving for~$\sigma'$.
    However, \textsc{Bland} does not necessarily apply these switches to~$\sigma'$ since there might be more preferred improving switches in the \textsc{Bland} ranking.

    Let~$S$ denote the set of states of~$\mathcal{M}_L$ whose actions \textsc{Bland} prefers over actions of~$x$, that is,~$S=\bigcup_{\ell=1}^{k-1}\{\alpha_\ell,\, \beta_\ell\}$ if~$x=\alpha_k$, and~$S=\bigcup_{\ell=1}^{k-1}\{\alpha_\ell,\, \beta_\ell\}\cup\{\alpha_k\}$ if~$x=\beta_k$.
    Then, there is no improving switch for~$\sigma$ in any state in~$S$.
    Let~$u\in S$ and~$v,w\in\bigcup_{\ell=1}^{k+1}\{\alpha_\ell,\, \beta_\ell\}$ such that~$\sigma(u)=v\neq w$, and~$v$ and~$w$ are in the same level.    
    By~\eqref{eq:correspond}, this gives us that~$\sigma'(u)=\delta(u,v)$, and~$\delta_a(u,v)$ is active in~$\sigma'$.
    
    Since the switch in~$u$ is not improving for~$\sigma$, the only action incident to~$u$ in~$\mathcal{M}'_L$ that might be improving for~$\sigma'$ is the action~$(\delta(u,w),u)$.
    In particular, this action is improving if and only if it is not active.
    Furthermore, its activation does not affect the other states since~$\delta(u,w)$ can not be reached from any other state when policy~$\sigma'$ is used.
    This yields that \textsc{Bland} applies all such switches~$(\delta(u,w),u)$ that are not active already.
    
    Afterwards, the improving switch incident to~$x$ is the most preferred improving switch in the \textsc{Bland} ranking.
    More precisely, if the action~$\delta_a(x,z)$ is already active in~$\sigma'$, \textsc{Bland} applies the improving switch~$(x,\delta(x,z))$.
    Otherwise, it activates~$\delta_a(x,z)$, which makes~$(x,\delta(x,z))$ improving and does not affect any other states.
    Thus, \textsc{Bland} applies~$(x,\delta(x,z))$ afterwards.
    So, in both cases, we end up with a policy that corresponds to~$\Bar{\sigma}$.
    It remains to choose the initial policy.

    Let~$\sigma_0$ be the initial policy from Lemma~\ref{lem:BDL_with_fequals1}.
    Let~$\sigma'_0$ be the unique policy for~$\mathcal{M}'_L$ that corresponds to~$\sigma_0$ and, for all states~$x,y\in\bigcup_{\ell=1}^{L+1}\{\alpha_\ell,\, \beta_\ell\}$ such that~$\delta(x,y)$ exists, satisfies~$\sigma'_0(\delta(x,y))=x$ if~$\sigma_0'(x) \neq \delta(x,y)$.
    Then, as discussed above, we obtain that \textsc{Bland} is exponential on~$\mathcal{M}_L'$ if we choose~$\sigma'_0$ as the initial policy.

    Let~$\Pi$ be a~$\{\textsc{Bland},\allowbreak \textsc{Dantzig},\allowbreak \textsc{LargestIncrease}\}$-\allowbreak ranking-based rule with~$f_{\Pi}(m)=m$ for all~$m\in\N$.
    We argue that the rankings for all three pivot rules agree on every policy that is visited by \textsc{Bland} on~$\mathcal{M}'_L$ with initial policy~$\sigma_0'$.
    Then~$\Pi$ applies the same sequence of improving switches as \textsc{Bland}, which we have already proven to be exponential.
    This yields the desired lower bound for~$\Pi$.

    It is time to specify the probabilities in~$\mathcal{M}'_L$.
    We set~$p_{\alpha_\ell}=M^{-2\ell+2}$ and~$p_{\beta_\ell}=M^{-2\ell+1}$ for every~$\ell\in\{1,2,\ldots,L\}$, where~$M\in\N$ is sufficiently large.
    Fix~$x,y\in \bigcup_{\ell=1}^{L+1}\{\alpha_\ell,\, \beta_\ell\}$ such that~$\delta(x,y)$ exists, and fix a weak unichain policy~$\sigma$ for~$\mathcal{M}'_L$.
    Then, the reduced costs of the actions~$\delta_a(x,y)$ and~$(\delta(x,y),x)$ are
    \begin{equation}\label{eq:rcDelta_a}
    \rc(\delta_a(x,y))=p_x(\val_{\sigma}(\epsilon(x,y)) - \val_\sigma(x))
    \end{equation}
    and
    \[
    \rc(\delta(x,y),x) = p_x( \val_\sigma(x) -\val_{\sigma}(\epsilon(x,y)) ),
    \]
    whenever they are inactive in~$\sigma$, respectively.
    Observe that the values of~$x$,~$y$, and~$\epsilon(x,y)$ do not depend on any probability~$p_{\alpha_\ell}$ or~$p_{\beta_\ell}$, since one always reaches~$\epsilon(x,y)$ from~$x$ with probability either~$0$ or~$1$.
    Let~$u,v \in \bigcup_{\ell=1}^{L}\{\alpha_\ell,\, \beta_\ell\}$ such that~$u$ is more preferred than~$v$ in the \textsc{Bland} ranking on~$\mathcal{M}'_L$.
    Since rewards and values are finite, we can choose~$M$ so large that the reduced costs of improving actions in the two states of the form~$\delta(u,\cdot)$ are larger than those of improving actions in the states~$\delta(v,\cdot)$.

    Further, observe that if there is an improving switch in~$\delta(x,y)$ for~$\sigma$, then~$(x,\delta(x,y))$ is inactive in~$\sigma$.
    Therefore, the reduced costs of an improving action in~$\delta(x,y)$ are equal to the objective improvement that its application yields, and there is only one improving switch incident to a state of the form~$\delta(x,\cdot)$.    
    We conclude that, given some weak unichain policy, the \textsc{Bland}, \textsc{Dantzig}, and \textsc{Largest\-Increase} rankings agree on the set of improving switches in states of the form~$\delta(\cdot,\cdot)$.

    We claim that, whenever an action of the form~$(x,\delta(x,y))$ is improving during the run of \textsc{Bland} with initial policy~$\sigma_0'$, then it is the most-preferred improving switch in all three rankings.
    This claim yields the statement since it implies that the three rankings agree for every policy visited during the run of \textsc{Bland}, which we have already shown to be exponential, and since~$\Pi$ selects the same switch as \textsc{Bland} whenever all rankings agree.
    So it only remains to prove our claim.

    First observe that there is no improving switch of the form~$(x,\delta(x,y))$ for~$\sigma'_0$.
    We now argue that, during the run of \textsc{Bland} with initial policy~$\sigma'_0$, an action of the form~$(x,\delta(x,y))$ can only become improving due to the activation of~$\delta_a(x,y)$.
    
    Let~$\sigma$ be the first policy visited during the run of \textsc{Bland} with initial policy~$\sigma'_0$ such that, for some pair of states~$x,y\in \bigcup_{\ell=1}^{L+1}\{\alpha_\ell,\, \beta_\ell\}$, the action~$(x,\delta(x,y))$ just became improving but the last switch was not the activation of~$\delta_a(x,y)$.
    Then, the last switch must have been applied in some state~$z\in \bigcup_{\ell=1}^{L}\{\alpha_\ell,\, \beta_\ell\}$ in a level above~$x$ since other switches do not affect state~$x$.
    By our assumption on~$\sigma$, the switch in~$z$ followed directly on the activation of an action of the form~$\delta_a(z,\cdot)$.
    This is a contradiction since the action~$(\delta(x,y),x)$ was improving at the time the action of the form~$\delta_a(z,\cdot)$ was applied -- as~$(x,\delta(x,y))$ was inactive but not improving --, and \textsc{Bland} prefers the former action over the latter.

    So we obtain that, during the run of \textsc{Bland}, the action~$(x,\delta(x,y))$ can only become improving due to the activation of~$\delta_a(x,y)$.
    Assume that~$\delta_a(x,y)$ gets applied to the current policy~$\sigma$ during the run of \textsc{Bland}.
    Then, since the \textsc{Bland}, \textsc{Dantzig}, and \textsc{Largest\-Increase} rankings agree on switches in the states~$\delta(\cdot,\cdot)$, the improving switch~$\delta_a(x,y)$ for~$\sigma$ is most-preferred by all three rules.
    
    The activation of~$\delta_a(x,y)$ only makes~$(x,\delta(x,y))$ improving and does not affect the other states. 
    The reduced costs of~$(x,\delta(x,y))$ are then
    \[
    \rc(x,\delta(x,y)) = p_x(\val_{\sigma}(\epsilon(x,y)) - \val_\sigma(x)),
    \]
    which equals the reduced costs of the previous switch~$\delta_a(x,y)$, see~\eqref{eq:rcDelta_a}.
    Therefore, the induced objective improvement is also at least as large as the previous improvement due to~$\delta_a(x,y)$. 
    This yields that, after the activation of~$\delta_a(x,y)$, the improving switch~$(x,\delta(x,y))$ is the most-preferred switch in all three rankings.
    This proves our claim.
    \qed
\end{proof}

Finally, we prove a subexponential lower bound for the case~$f_{\Pi}(k)\neq o(\sqrt{k})$ in \cref{lem:notInSqrt}, by extending the Markov decision processes that were used for the third case.
By assumption, there exists a monotone sequence~$(k_i)_{i\in\N}\in\N^{\N}$ such that~$f_{\Pi}(k_i)\geq c\sqrt{k_i}$, for some constant~$c>0$ and for all~$i\in\N$.
For each~$i\in\N$, we take one of the processes from the third case, and equip it with gadgets such that, for a suitable initial policy and during a sufficient number of iterations, the number of improving switches is exactly~$k_i$, and the~$f_\Pi(k_i)$-th least-preferred improving switch in the new process corresponds to the most-preferred switch in the original process.
This yields a subexponential bound since~$\Pi$ mimics the behavior from the third case on the new process, whose size is quadratic in the size of the original process.

\begin{figure}
    \centering
    \resizebox{.5\linewidth}{!}{%
    \begin{tikzpicture}[node distance = 3cm,
    p0/.style={circle, draw=p0Blue!60, fill=p0Blue!40, thick, minimum size=12mm,inner sep=0pt, outer sep=0pt}, 
    every path/.style={draw=p0Blue,line width=1.3pt,{>=Latex}}, 
      p1/.style={rectangle, draw=p1Red!60, fill=p1Red!40, thick, minimum size=5mm,inner sep=0pt, outer sep=0pt}, 
      phan/.style={draw=none, fill=none, inner sep=0pt, outer sep=0pt}, 
      ]

    \node[p0] (x) at (0,0) {$x$};
    \node[p0] (x') at (4,0) {$v(x,y)$};
    \node[p0] (y) at (7,0) {$y$};
    \node[p1] (z) at (1,2) {};   
    \node[p0] (g) at (4,3) {$u(x,y)$}; 

    \draw (x) edge[->] node [below] {$0$} (x');
    \draw (x') edge[->] node [below] {$\rew(x,y)$} (y);
    \draw[p1Red] (z) edge[p1Red,->] node [above left] {$p(x,y)$} (x);
    \draw[p1Red] (z) edge[p1Red,->] node [above right=3pt and -13pt] {$1-p(x,y)$} (x');
    \draw (g) edge[->] node [above right=2pt and -20pt] {$q(x,y)$} (z);
    \draw (g) edge[->] node [right] {$0$} (x');

\end{tikzpicture}}
    \caption{The gadget described in the proof of Lemma~\ref{lem:notInSqrt}. The action with reward~$q(x,y)$ is called~$u_a(x,y)$.}
    \label{fig:notOsqrt}
\end{figure}
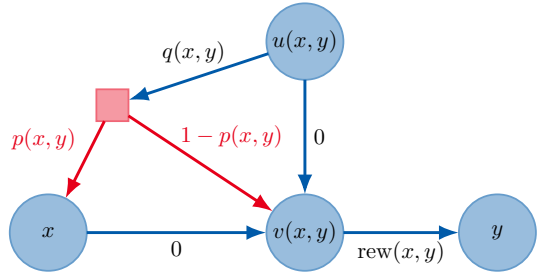

\begin{lemma}\label[lemma]{lem:notInSqrt}
    Given a~$\{\textup{\textsc{Bland}}, \textup{\textsc{Dantzig}}, \textup{\textsc{LargestIncrease}}\}$-ranking-based pivot rule~$\Pi$ with~$f_{\Pi}(m)\neq o(\sqrt{m})$, there exists a family of weak unichain Markov decision processes with~$N$ non-zero transition probabilities such that \textup{\textsc{PolicyIteration}} with pivot rule~$\Pi$ takes~$\Omega(2^{\sqrt N})$ iterations.
\end{lemma}
\begin{proof}    
    Since~$f_\Pi(m)\neq o(\sqrt{m})$, there exists a constant~$c>0$ and some strictly monotone sequence~$(m_i)_{i\in\N}\subseteq\N$ such that~$f_{\Pi}(m_i)\geq c\sqrt{m_i}$, for all~$i\in\N$.
    Let~$\left( \mathcal{M}'_L \right)_{L\in\N_{\geq 2}}$ be the family of Markov decision processes described in the first paragraph of the proof of Lemma~\ref{lem:BDL_with_fequalsm}.
    Fix~$i\in\N$ such that some processes in~$\left( \mathcal{M}'_L \right)_{L\in\N_{\geq 2}}$ have at most~$f_{\Pi}(m_i)+1$ actions.
    Let~$\mathcal{M}^{(3)}_i$ denote the largest of these processes, that is, $\mathcal{M}^{(3)}_i=\mathcal{M}'_{L_\text{max}}$, where
    \[ L_\text{max}=\arg\max\{ L\in\N_{\geq2}\, \colon\, \mathcal{M}'_L \text{ has at most } f_{\Pi}(m_i)+1 \text{ actions}\}.
    \]
    Let~$\mathcal{M}^{(1)}_i$ denote the Markov decision process that we obtain from~$\mathcal{M}^{(3)}_i$ as follows: 
    \begin{itemize}
        \item First, add new states whose single action leads to the sink without reward, until there are exactly~$f_{\Pi}(m_i)+1$ actions.
        Denote the resulting Markov decision process by~$\mathcal{M}^{(2)}_i$.
        \item Then, replace every action~$(x,y)$ of~$\mathcal{M}^{(2)}_i$, except the unique action in the sink, by the gadget drawn in Figure~\ref{fig:notOsqrt}.
        That is, we
        \begin{itemize}
            \item add two states, called~$u(x,y)$ and~$v(x,y)$,
            \item add a deterministic action without reward from~$x$ to~$v(x,y)$,
            \item add a deterministic action without reward from~$u(x,y)$ to~$v(x,y)$,
            \item add a deterministic action with the reward~$\rew(x,y)$ of action~$(x,y)$ from~$v(x,y)$ to~$y$,
            \item add a probabilistic action, which we denote with~$u_a(x,y)$, with reward~$q(x,y)$ that starts in~$u(x,y)$, and leads to the state~$x$ with probability~$p(x,y)\in(0,1)$ and to~$v(x,y)$ with probability~$1-p(x,y)$, where the values of~$p(x,y)$ and~$q(x,y)$ are specified below,
            \item delete the action~$(x,y)$.
        \end{itemize}
    \end{itemize}
    Observe that, given policies~$\sigma^{(3)}$ for~$\mathcal{M}^{(3)}_i$ and~$\sigma^{(1)}$  for~$\mathcal{M}^{(1)}_i$ with
    \[
    \sigma^{(3)}(x)=y \quad\Longleftrightarrow\quad \sigma^{(1)}(x)=v(x,y),
    \]
    for all actions~$(x,y)$ of~$\mathcal{M}^{(2)}_i$, we have~$\rc_{\sigma^{(3)}}(x,y) = \rc_{\sigma^{(1)}}(x,v(x,y))$, and the objective improvements induced by these switches is also the same.

    Fix an arbitrary action~$(x,y)$ of~$\mathcal{M}^{(2)}_i$, and some policy~$\sigma^{(1)}$ for~$\mathcal{M}^{(1)}_i$ such that~$\sigma^{(1)}(u(x,y))=v(x,y)$.
    We first argue that we can choose~$p(x,y)$ and~$q(x,y)$ such that
    \begin{itemize}
        \item the reduced costs of~$u_a(x,y)$ are smaller than the reduced costs of any improving switch that might occur in~$\mathcal{M}^{(3)}_i$,
        \item~$u_a(x,y)$ is improving~$\Longleftrightarrow$~$(x,v(x,y))$ is not improving.
    \end{itemize}
    Let~$L>0$ denote a lower bound on the reduced costs of all improving switches that could occur in~$\mathcal{M}^{(3)}_i$.
    Similarly, let~$U$ denote an upper bound on the absolute value of the objective change that is induced by any switch, improving or not, that could occur in~$\mathcal{M}^{(3)}_i$.

    Assume that~$p(x,y)\in(0,\frac{L}{L+U})$ and~$q(x,y)\in(0,p(x,y)\cdot L)$.
    Then, the reduced costs of the action~$u_a(x,y)$ satisfy
    \begin{align}
    \rc_{\sigma^{(1)}}(u_a(x,y)) &= q(x,y) + p(x,y)\cdot \left( \val_{\sigma^{(1)}}(x) - \val_{\sigma^{(1)}}(v(x,y)) \right)\label{eq:rc_of_ua} \\
    & = q(x,y) - p(x,y)\cdot \rc_{\sigma^{(1)}}(x,v(x,y)) \notag\\
    & \leq q(x,y) + p(x,y)\cdot U < p(x,y)\cdot\left( L+U \right) < L,\notag
    \end{align}
    where we used that~$U$ bounds the absolute value of~$\rc_{\sigma^{(1)}}(x,v(x,y))$.
    Further, if~$\rc_{\sigma^{(1)}}(x,v(x,y))$ is positive, then we have~$\rc_{\sigma^{(1)}}(x,v(x,y))\geq L$.
    This yields~$\rc_{\sigma^{(1)}}(x,v(x,y))>0$ if and only if~$\rc_{\sigma^{(1)}}(u_a(x,y))\leq 0$.   
    We conclude that, given a policy~$\sigma^{(1)}$ for~$\mathcal{M}^{(1)}_i$ such that~$\sigma^{(1)}(u(x,y))=v(x,y)$ for all actions~$(x,y)$ in~$\mathcal{M}^{(2)}_i$, the number of improving switches for~$\sigma^{(1)}$ is~$f_{\Pi}(m_i)$.
    
    Let~$\mathcal{M}_i$ denote the Markov decision process that we obtain from~$\mathcal{M}^{(1)}_i$ as follows:
    \begin{itemize}
        \item add a new state~$\gamma_0$ whose single action leads to the sink~$\top$ without reward,
        \item for every~$k\in\{1,2,\ldots,m_i-f_{\Pi}(m_i)\}$, add a new state~$\gamma_k$ that has two deterministic actions, one leading to~$\top$ without reward, the other leading to~$\gamma_0$ with reward~$U+k$.
    \end{itemize}
    We say that a policy~$\sigma$ for~$\mathcal{M}_i$ is \emph{induced} by a policy~$\sigma^{(3)}$ for~$\mathcal{M}^{(3)}_i$ if    
    \begin{itemize}
        \item~$\sigma(u(x,y))=v(x,y)$,\quad \text{for all actions~$(x,y)$ in~$\mathcal{M}^{(2)}_i$},
        \item~$\sigma(x)=v(x,y) \quad\Longleftrightarrow\quad \sigma^{(3)}(x)=y$,\quad \text{for all actions~$(x,y)$ in~$\mathcal{M}^{(3)}_i$},
        \item~$\sigma(\gamma_k)=\top$, for all~$k\in\{1,2,\ldots,m_i-f_{\Pi}(m_i)\}$.
    \end{itemize}
    Observe that every policy for~$\mathcal{M}^{(3)}_i$ induces a unique policy for~$\mathcal{M}_i$.

    We now define the \textsc{Bland} ranking on the actions of~$\mathcal{M}_i$.
    We divide the set of actions that can become improving in~$\mathcal{M}_i$ into three disjoint sets
    \begin{align*}
    S_1 &\coloneqq\left\{(\gamma_k,\gamma_0)\colon k\in\{1,2,\ldots,m_i-f_{\Pi}(m_i)\} \right\},\\
    S_2 &\coloneqq\left\{(x,v(x,y))\colon (x,y) \text{ is an action in } \mathcal{M}^{(3)}_i\right\},\\
    S_3 &\coloneqq\left\{u_a(x,y)\colon (x,y) \text{ is an action in } \mathcal{M}^{(2)}_i\right\}.       
    \end{align*}
    Then, \textsc{Bland} prefers switches in~$S_1$ over those in~$S_2$ or~$S_3$, and switches in~$S_2$ over those in~$S_3$.
    Further, it prefers~$(\gamma_k,\gamma_0)\in S_1$ over~$(\gamma_\ell,\gamma_0)\in S_1$ if~$k>\ell$.
    On~$S_2$, we use the \textsc{Bland} ranking for~$\mathcal{M}^{(3)}_i$ from the proof of Lemma~\ref{lem:BDL_with_fequalsm}, where we identify~$(x,v(x,y))\in\mathcal{M}_i$ with~$(x,y)\in\mathcal{M}^{(3)}_i$.
    Finally, on~$S_3$ we fix some arbitrary ranking.

    Let~$\sigma^{(3)}$ denote a non-optimal policy for~$\mathcal{M}^{(3)}_i$ visited during the run of \textsc{PolicyIteration} on~$\mathcal{M}^{(3)}_i$ with a pivot rule~$\Pi'$ as in Lemma~\ref{lem:BDL_with_fequalsm}.
    Let~$\sigma$ be the policy for~$\mathcal{M}_i$ induced by~$\sigma^{(3)}$.
    Let~$(\Bar x,\Bar y)$ denote the improving switch that~$\Pi'$ applies to~$\sigma^{(3)}$, and denote the resulting policy with~$\Bar\sigma^{(3)}$.
    We will argue that~$\Pi$ applies the switch~$(\Bar x,v(\Bar x,\Bar y))$ to~$\sigma$, which results in the policy induced by~$\Bar\sigma^{(3)}$.

    But first, for every action~$(x,y)$ in~$\mathcal{M}^{(2)}_i$, we specify the choice of the probability~$p(x,y)\in(0,\frac{L}{L+U})$ and the reward~$q(x,y)\in(0,p(x,y)\cdot L)$.
    Let~$u_a(x_i,y_i)$ denote the~$i$-th action in the \textsc{Bland} ranking of the actions in~$S_3$.
    That is,~$u_a(x_1,y_1)$ is the most-preferred action in~$S_3$.
    We fix~$p(x_1,y_1)\in(0,\frac{L}{L+U})$ and~$q(x_1,y_1)\in(0,p(x_1,y_1)\cdot L)$ arbitrarily.
    Then, for every~$k\in\{2,3,\ldots,\vert S_3\vert\}$, we set~$p(x_k,y_k)=M^{-k+1}\cdot p(x_1,y_1)$ and~$q(x_k,y_k)=M^{-k+1}\cdot q(x_1,y_1)$, for some~$M\in\N$.
    
    These choices yield that~$p(x,y)\in(0,\frac{L}{L+U})$ and~$q(x,y)\in(0,p(x,y)\cdot L)$ for every action~$(x,y)\in\mathcal{M}^{(2)}_i$.
    Further, by~\eqref{eq:rc_of_ua}, the reduced costs~$\rc(u_a(x_k,y_k))$ are scaled by~$M^{-k+1}$ for every~$k\in\{1,2,\ldots,\vert S_3\vert\}$.
    Thus, if~$M$ is sufficiently large, the \textsc{Bland} ranking and the \textsc{Dantzig} ranking agree on~$S_3$, with respect to any policy.
    They also agree with the \textsc{Largest\-Increase} ranking since the reduced costs of the edge~$u_a(x,y)$ equals the induced objective improvement.

    It is obvious that the three rankings also agree on the set~$S_1$, irrespective of the policy, since the reduced cost and objective improvement of the improving switch~$(\gamma_k,\gamma_0)$ is~$U+k$ for all~$k\in\{1,2,\ldots,m_i-f_{\Pi}(m_i)\}$.

    Further, we have~$\rc_{\sigma^{(3)}}(x,y) = \rc_{\sigma}(x,v(x,y))$ for every action~$(x,y)$ in the process~$\mathcal{M}^{(2)}_i$.
    Thus, we obtain from Lemma~\ref{lem:BDL_with_fequalsm} that the rankings by \textsc{Bland}, \textsc{Dantzig}, and \textsc{Largest\-Increase} also agree on~$S_2$ for~$\sigma$.

    We conclude that the three rankings agree on the set of improving switches for~$\sigma$.
    Thus,~$\Pi$ selects the~$f_{\Pi}(m_i)$-th least-preferred improving switch in the common ranking.
    Since there are~$m_i$ improving switches for~$\sigma$, of which exactly~$m_i-f_{\Pi}(m_i)$ are contained in~$S_1$, the selected switch is the most-preferred improving switch in~$S_2$, which is the switch~$(\Bar x,v(\Bar x,\Bar y))$.

    Therefore, the run of~$\Pi$ on~$\mathcal{M}_i$ mimics the run of~$\Pi'$ on~$\mathcal{M}^{(3)}_i$, up to the point where~$\Pi'$ has found the optimum.
    Due to Lemma~\ref{lem:BDL_with_fequalsm}, this takes a number of iterations that is exponential in the number of actions in~$\mathcal{M}^{(3)}_i$. 
    Since~$\mathcal{M}_i$ has~$\Theta(m_i)$ actions, all of wich are deterministic, and since~$\mathcal{M}^{(3)}_i$ has~$\Theta(f_{\Pi}(m_i))=\Omega(\sqrt{m_i})$ actions, this yields the desired subexponential lower bound.
    \qed
\end{proof}

Combining~\cref{lem:Sqrt,lem:notInSqrt} yields Theorem~\ref{thm:rankingsPI}, 
which in turn implies Theorem~\ref{thm:rankingsSimplex}; 
these two theorems establish our second main result, Theorem~\ref{thm:rankingsSimplexInformal}.

\newpage

\bibliographystyle{spmpsci}      
\bibliography{references}

\end{document}

%% file: ClusteredExample.tex
\definecolor{p0Blue}{RGB}{0, 90, 169} 
\definecolor{p1Red}{RGB}{230, 0, 26} 
\begin{tikzpicture}[node distance = 3cm,
    p0/.style={circle, draw=p0Blue!60, fill=p0Blue!40, thick, minimum size=12mm}, 
    every path/.style={draw=p0Blue,line width=1.3pt,{>=Latex}}, 
      p1/.style={rectangle, draw=p1Red!60, fill=p1Red!40, thick, minimum size=6mm}, 
      phan/.style={draw=none, fill=none, inner sep=0pt, outer sep=0pt}] 

    \node[p0] (a1) at (0,0) {$\begin{array}{c}a_1 \\ 3 \end{array}$};
    \node[p1] (b1) at (0,10) {$\begin{array}{c}b_1 \\ 4 \end{array}$};
    \node[p0] (a2) at (10,0) {$\begin{array}{c}a_2 \\ 5 \end{array}$};
    \node[p1] (b2) at (10,10) {$\begin{array}{c}b_2 \\ 6 \end{array}$};
    \node[p1] (b3) at (20,10) {$\begin{array}{c}b_3 \\ 8 \end{array}$};
    \node[p0] (top) at (20,0) {$\top$};

    \node[p1, minimum size=9mm] (d11) at (0,7) {};
    \node[p1, minimum size=9mm] (d12) at (3,6) {};
    \node[p0, minimum size=9mm] (e11) at (-1,2) {};
    \node[p1] (f11) at (-1,5) {$3$};
    \node[p1] (g11) at (-1,6) {$1$};
    \node[p0, minimum size=9mm] (e12) at (2,2) {};
    \node[p1] (f12) at (3,3.5) {$3$};
    \node[p1] (g12) at (3,4.5) {$1$};

    \node[p1, minimum size=9mm] (d13) at (7,1) {};
    \node[p1, minimum size=9mm] (d14) at (7,-1) {};
    \node[p0, minimum size=9mm] (e13) at (2.5,1.2) {};
    \node[p1] (f13) at (4,2) {$3$};
    \node[p1] (g13) at (5,2) {$1$};
    \node[p0, minimum size=9mm] (e14) at (2.5,-1.2) {};
    \node[p1] (f14) at (4,-2) {$3$};
    \node[p1] (g14) at (5,-2) {$1$};

    \node[p1, minimum size=9mm] (d21) at (10,7) {};
    \node[p1, minimum size=9mm] (d22) at (13,6) {};
    \node[p0, minimum size=9mm] (e21) at (9,2) {};
    \node[p1] (f21) at (9,5) {$5$};
    \node[p1] (g21) at (9,6) {$1$};
    \node[p0, minimum size=9mm] (e22) at (12,2) {};
    \node[p1] (f22) at (13,3.5) {$5$};
    \node[p1] (g22) at (13,4.5) {$1$};

    \node[p1, minimum size=9mm] (d23) at (17,1) {};
    \node[p1, minimum size=9mm] (d24) at (17,-1) {};
    \node[p0, minimum size=9mm] (e23) at (12.5,1.2) {};
    \node[p1] (f23) at (14,2) {$5$};
    \node[p1] (g23) at (15,2) {$1$};
    \node[p0, minimum size=9mm] (e24) at (12.5,-1.2) {};
    \node[p1] (f24) at (14,-2) {$5$};
    \node[p1] (g24) at (15,-2) {$1$};

    \draw (a1) edge[->] node [pos=.7,right] {$4$} (d11);
    \draw (a1) edge[->] node [pos=.7,left] {$3$} (d12);
    \draw (a1) edge[->] node [pos=.7,below] {$2$} (d13);
    \draw (a1) edge[->, line width=3] node [pos=.7,above] {$1$} (d14);
    \draw (e11) edge[->] node [pos=.5,left] {$12$} (a1);
    \draw (e11) edge[->, line width=3] node [pos=.5,left] {$16$} (f11);
    \draw (f11) edge[->,p1Red] (g11);
    \draw (g11) edge[->,p1Red] (d11);
    \draw (d11) edge[->,p1Red] (b2);
    \draw (e12) edge[->] node [pos=.5,above] {$11$} (a1);
    \draw (e12) edge[->, line width=3] node [pos=.5,left] {$15$} (f12);
    \draw (f12) edge[->,p1Red] (g12);
    \draw (g12) edge[->,p1Red] (d12);
    \draw (d12) edge[->,p1Red] (b2);

    \draw (e13) edge[->] node [pos=.5,above] {$10$} (a1);
    \draw (e13) edge[->, line width=3] node [pos=.5,above] {$14$} (f13);
    \draw (f13) edge[->,p1Red] (g13);
    \draw (g13) edge[->,p1Red] (d13);
    \draw (d13) edge[->,p1Red] (a2);
    \draw (e14) edge[->] node [pos=.5,below] {$9$} (a1);
    \draw (e14) edge[->, line width=3] node [pos=.5,below] {$13$} (f14);
    \draw (f14) edge[->,p1Red] (g14);
    \draw (g14) edge[->,p1Red] (d14);
    \draw (d14) edge[->,p1Red] (a2);

    \draw (b1) edge[->,p1Red] (b2);
    \draw (b1) edge[->,p1Red] (a2);

     \draw (a2) edge[->] node [pos=.7,right] {$8$} (d21);
    \draw (a2) edge[->] node [pos=.7,left] {$7$} (d22);
    \draw (a2) edge[->] node [pos=.7,below] {$6$} (d23);
    \draw (a2) edge[->, line width=3] node [pos=.7,above] {$5$} (d24);
    \draw (e21) edge[->] node [pos=.5,left] {$20$} (a2);
    \draw (e21) edge[->, line width=3] node [pos=.5,left] {$24$} (f21);
    \draw (f21) edge[->,p1Red] (g21);
    \draw (g21) edge[->,p1Red] (d21);
    \draw (d21) edge[->,p1Red] (b3);
    \draw (e22) edge[->] node [pos=.5,above] {$19$} (a2);
    \draw (e22) edge[->, line width=3] node [pos=.5,left] {$23$} (f22);
    \draw (f22) edge[->,p1Red] (g22);
    \draw (g22) edge[->,p1Red] (d22);
    \draw (d22) edge[->,p1Red] (b3);

    \draw (e23) edge[->] node [pos=.5,above] {$18$} (a2);
    \draw (e23) edge[->, line width=3] node [pos=.5,above] {$22$} (f23);
    \draw (f23) edge[->,p1Red] (g23);
    \draw (g23) edge[->,p1Red] (d23);
    \draw (d23) edge[->,p1Red] (top);
    \draw (e24) edge[->] node [pos=.5,below] {$17$} (a2);
    \draw (e24) edge[->, line width=3] node [pos=.5,below] {$21$} (f24);
    \draw (f24) edge[->,p1Red] (g24);
    \draw (g24) edge[->,p1Red] (d24);
    \draw (d24) edge[->,p1Red] (top);
    
    \draw (b2) edge[->,p1Red] (b3);
    \draw (b2) edge[->,p1Red] (top);
    \draw (b3) edge[->, p1Red] (top);
    \draw (top) edge[loop right, line width=3] (top);

    \node[p0] (x1) at (23,6) {$2$};
    \node[p0, minimum size=9mm] (x3) at (23,3) {$3$};

    \draw (x1) edge[->] node[pos=.5, left] {$25$} (top);
    \draw (x1) edge[->, line width=3] node[pos=.5, right] {$26$} (x3);
    \draw (x3) edge[->, line width=3] node[pos=.5, above] {$27$} (top);

    \node[p0] (a4) at (0,-15) {$\begin{array}{c}a_4 \\ 3 \end{array}$};
    \node[p1] (b4) at (0,-5) {$\begin{array}{c}b_4 \\ 4 \end{array}$};
    \node[p0] (a5) at (10,-15) {$\begin{array}{c}a_5 \\ 5 \end{array}$};
    \node[p1] (b5) at (10,-5) {$\begin{array}{c}b_5 \\ 6 \end{array}$};
    \node[p1] (b6) at (20,-5) {$\begin{array}{c}b_6 \\ 8 \end{array}$};

    \node[p1, minimum size=9mm] (d32) at (3,-9) {};
    \node[p1] (f32) at (3,-11.5) {$3$};
    \node[p1] (g32) at (3,-10.5) {$1$};
    \node[p0, minimum size=9mm] (e32) at (2,-13) {};

    \node[p1, minimum size=9mm] (d33) at (7,-14) {};
    \node[p0, minimum size=9mm] (e33) at (2.5,-13.8) {};
    \node[p1] (f33) at (4,-13) {$3$};
    \node[p1] (g33) at (5,-13) {$1$};

    \node[p1, minimum size=9mm] (d42) at (13,-9) {};
    \node[p0, minimum size=9mm] (e42) at (12,-13) {};
    \node[p1] (f42) at (13,-11.5) {$5$};
    \node[p1] (g42) at (13,-10.5) {$1$};

    \node[p1, minimum size=9mm] (d43) at (17,-14) {};
    \node[p0, minimum size=9mm] (e43) at (12.5,-13.8) {};
    \node[p1] (f43) at (14,-13) {$5$};
    \node[p1] (g43) at (15,-13) {$1$};  
    
    \draw (a4) edge[->] node [pos=.7,left] {$38$} (d32);
    \draw (a4) edge[->, line width=3] node [pos=.7,below] {$39$} (d33);
    \draw (e32) edge[->] node [pos=.5,above] {$34$} (a4);
    \draw (e32) edge[->, line width=3] node [pos=.5,left] {$32$} (f32);
    \draw (f32) edge[->,p1Red] (g32);
    \draw (g32) edge[->,p1Red] (d32);
    \draw (d32) edge[->,p1Red] (b5);

    \draw (e33) edge[->] node [pos=.5,above] {$35$} (a4);
    \draw (e33) edge[->, line width=3] node [pos=.5,above] {$33$} (f33);
    \draw (f33) edge[->,p1Red] (g33);
    \draw (g33) edge[->,p1Red] (d33);
    \draw (d33) edge[->,p1Red] (a5);    
    
    \draw (b4) edge[->,p1Red] (b5);
    \draw (b4) edge[->,p1Red] (a5);

    \draw (a5) edge[->] node [pos=.7,left] {$36$} (d42);
    \draw (a5) edge[->, line width=3] node [pos=.7,below] {$37$} (d43);
    \draw (e42) edge[->] node [pos=.5,above] {$30$} (a5);
    \draw (e42) edge[->, line width=3] node [pos=.5,left] {$28$} (f42);
    \draw (f42) edge[->,p1Red] (g42);
    \draw (g42) edge[->,p1Red] (d42);
    \draw (d42) edge[->,p1Red] (b6);

    \draw (e43) edge[->] node [pos=.5,above] {$31$} (a5);
    \draw (e43) edge[->, line width=3] node [pos=.5,above] {$29$} (f43);
    \draw (f43) edge[->,p1Red] (g43);
    \draw (g43) edge[->,p1Red] (d43);
    \draw (d43) edge[->,p1Red] (top);
    
    \draw (b5) edge[->,p1Red] (b6);
    \draw (b5) edge[->,p1Red, bend right=5] (top);
    \draw (b6) edge[->, p1Red] (top);

\end{tikzpicture}

%% file: DispersedExample.tex
\definecolor{p0Blue}{RGB}{0, 90, 169} 
\definecolor{p1Red}{RGB}{230, 0, 26} 
\begin{tikzpicture}[node distance = 3cm,
    p0/.style={circle, draw=p0Blue!60, fill=p0Blue!40, thick, minimum size=12mm}, 
    every path/.style={draw=p0Blue,line width=1.3pt,{>=Latex}}, 
      p1/.style={rectangle, draw=p1Red!60, fill=p1Red!40, thick, minimum size=11mm}, 
      phan/.style={draw=none, fill=none, inner sep=0pt, outer sep=0pt}] 

    \node[p1] (a1) at (-4,12) {$\begin{array}{c}
        a_1\\
         5 
    \end{array}$};
    \node[p1] (b1) at (-4,18) {$\begin{array}{c}
        b_1\\
         6 
    \end{array}$};
    \node[p0] (z13) at (0,0) {$3$};
    \node[p0] (z12) at (0,6) {$3$};
    \node[p0] (z11) at (0,12) {$3$};
    \node[p1] (q11) at (9,0) {$4$};
    \node[p1] (q12) at (9,12) {$4$};
    \node[p0, minimum size=7mm] (d11) at (-2,2) {};
    \node[p1,, minimum size=7mm] (e11) at (-2,3.5) {$3$};
    \node[p1,, minimum size=7mm] (f11) at (-2,5) {$1$};
    \node[p0, minimum size=7mm] (d12) at (-2,8) {};
    \node[p1,, minimum size=7mm] (e12) at (-2,9.5) {$3$};
    \node[p1,, minimum size=7mm] (f12) at (-2,11) {$1$};

    \node[p0, minimum size=7mm] (d13) at (2,10) {};
    \node[p1,, minimum size=7mm] (e13) at (2,8.5) {$3$};
    \node[p1,, minimum size=7mm] (f13) at (2,7) {$1$};
    \node[p0, minimum size=7mm] (d14) at (2,4) {};
    \node[p1,, minimum size=7mm] (e14) at (2,2.5) {$3$};
    \node[p1,, minimum size=7mm] (f14) at (2,1) {$1$};

    \node[p0, minimum size=7mm] (d15) at (3,-2) {};
    \node[p1,, minimum size=7mm] (e15) at (5,-2) {$3$};
    \node[p1,, minimum size=7mm] (f15) at (7,-2) {$1$};
    \node[p0, minimum size=7mm] (d16) at (3,14) {};
    \node[p1,, minimum size=7mm] (e16) at (5,14) {$3$};
    \node[p1,, minimum size=7mm] (f16) at (7,14) {$1$};

    \draw (a1) edge[->,p1Red] (z11);
    \draw (z11) edge[->,bend left=8, line width=3] node[pos=.5,right] {$13$} (z12);
    \draw (z12) edge[->,bend left=8, line width=3] node[pos=.5,right] {$14$} (z13);
    \draw (z13) edge[->, line width=3] node[pos=.5,above] {$15$}(q11);
    \draw (z11) edge[->] node[pos=.5,below] {$18$} (q12);
    \draw (z12) edge[->, bend left=8] node[pos=.5,left] {$17$} (z11);
    \draw (z13) edge[->, bend left=8] node[pos=.5,left] {$16$} (z12);
    \draw (q11) edge[->,bend right=30, p1Red] (z11);
    \draw (q12) edge[->,bend left=40, p1Red] (z13);

    \draw (d11) edge[->] node[pos=.2,right] {$25$} (z13);
    \draw (d11) edge[->, line width=3] node[pos=.5,right] {$26$} (e11);
    \draw (e11) edge[->,p1Red] (f11);
    \draw (f11) edge[->,p1Red] (z12);

    \draw (d12) edge[->] node[pos=.2,right] {$27$} (z12);
    \draw (d12) edge[->, line width=3] node[pos=.5,right] {$28$} (e12);
    \draw (e12) edge[->,p1Red] (f12);
    \draw (f12) edge[->,p1Red] (z11);

    \draw (d13) edge[->] node[pos=.2,left] {$31$} (z11);
    \draw (d13) edge[->, line width=3] node[pos=.5,left] {$32$} (e13);
    \draw (e13) edge[->,p1Red] (f13);
    \draw (f13) edge[->,p1Red] (z12);

    \draw (d14) edge[->] node[pos=.2,left] {$33$} (z12);
    \draw (d14) edge[->, line width=3] node[pos=.5,left] {$34$} (e14);
    \draw (e14) edge[->,p1Red] (f14);
    \draw (f14) edge[->,p1Red] (z13);

    \draw (d15) edge[->] node[pos=.5,below] {$35$} (z13);
    \draw (d15) edge[->, line width=3] node[pos=.5,below] {$36$} (e15);
    \draw (e15) edge[->,p1Red] (f15);
    \draw (f15) edge[->,p1Red] (q11);

    \draw (d16) edge[->] node[pos=.5,above] {$29$} (z11);
    \draw (d16) edge[->, line width=3] node[pos=.5,above] {$30$} (e16);
    \draw (e16) edge[->,p1Red] (f16);
    \draw (f16) edge[->,p1Red] (q12);

    \node[p1] (a2) at (14,12) {$\begin{array}{c}
        a_2\\
         7 
    \end{array}$};
    \node[p1] (b2) at (14,18) {$\begin{array}{c}
        b_2\\
         8 
    \end{array}$};

    \draw (b1) edge[->,p1Red] (b2);
    \draw (b1) edge[->,p1Red, bend left=10] (a2);
    \draw (q11) edge[->,p1Red] (a2);
    \draw (q12) edge[->,p1Red] (b2);

    \node[p0] (z23) at (18,0) {$3$};
    \node[p0] (z22) at (18,6) {$3$};
    \node[p0] (z21) at (18,12) {$3$};
    \node[p1] (q21) at (27,0) {$4$};
    \node[p1] (q22) at (27,12) {$4$};
    \node[p0, minimum size=7mm] (d21) at (16,2) {};
    \node[p1,, minimum size=7mm] (e21) at (16,3.5) {$3$};
    \node[p1,, minimum size=7mm] (f21) at (16,5) {$1$};
    \node[p0, minimum size=7mm] (d22) at (16,8) {};
    \node[p1,, minimum size=7mm] (e22) at (16,9.5) {$3$};
    \node[p1,, minimum size=7mm] (f22) at (16,11) {$1$};

    \node[p0, minimum size=7mm] (d23) at (20,10) {};
    \node[p1,, minimum size=7mm] (e23) at (20,8.5) {$3$};
    \node[p1,, minimum size=7mm] (f23) at (20,7) {$1$};
    \node[p0, minimum size=7mm] (d24) at (20,4) {};
    \node[p1,, minimum size=7mm] (e24) at (20,2.5) {$3$};
    \node[p1,, minimum size=7mm] (f24) at (20,1) {$1$};

    \node[p0, minimum size=7mm] (d25) at (21,-2) {};
    \node[p1,, minimum size=7mm] (e25) at (23,-2) {$3$};
    \node[p1,, minimum size=7mm] (f25) at (25,-2) {$1$};
    \node[p0, minimum size=7mm] (d26) at (21,14) {};
    \node[p1,, minimum size=7mm] (e26) at (23,14) {$3$};
    \node[p1,, minimum size=7mm] (f26) at (25,14) {$1$};

    \draw (a2) edge[->,p1Red] (z21);
    \draw (z21) edge[->,bend left=8, line width=3] node[pos=.5,right] {$19$} (z22);
    \draw (z22) edge[->,bend left=8, line width=3] node[pos=.5,right] {$20$} (z23);
    \draw (z23) edge[->, line width=3] node[pos=.5,above] {$21$}(q21);
    \draw (z21) edge[->] node[pos=.5,below] {$24$} (q22);
    \draw (z22) edge[->, bend left=8] node[pos=.5,left] {$23$} (z21);
    \draw (z23) edge[->, bend left=8] node[pos=.5,left] {$22$} (z22);
    \draw (q21) edge[->,bend right=30, p1Red] (z21);
    \draw (q22) edge[->,bend left=40, p1Red] (z23);

    \draw (d21) edge[->] node[pos=.2,right] {$37$} (z23);
    \draw (d21) edge[->, line width=3] node[pos=.5,right] {$38$} (e21);
    \draw (e21) edge[->,p1Red] (f21);
    \draw (f21) edge[->,p1Red] (z22);

    \draw (d22) edge[->] node[pos=.2,right] {$39$} (z22);
    \draw (d22) edge[->, line width=3] node[pos=.5,right] {$40$} (e22);
    \draw (e22) edge[->,p1Red] (f22);
    \draw (f22) edge[->,p1Red] (z21);

    \draw (d23) edge[->] node[pos=.2,left] {$43$} (z21);
    \draw (d23) edge[->, line width=3] node[pos=.5,left] {$44$} (e23);
    \draw (e23) edge[->,p1Red] (f23);
    \draw (f23) edge[->,p1Red] (z22);

    \draw (d24) edge[->] node[pos=.2,left] {$45$} (z22);
    \draw (d24) edge[->, line width=3] node[pos=.5,left] {$46$} (e24);
    \draw (e24) edge[->,p1Red] (f24);
    \draw (f24) edge[->,p1Red] (z23);

    \draw (d25) edge[->] node[pos=.5,above] {$47$} (z23);
    \draw (d25) edge[->, line width=3] node[pos=.5,above] {$48$} (e25);
    \draw (e25) edge[->,p1Red] (f25);
    \draw (f25) edge[->,p1Red] (q21);

    \draw (d26) edge[->] node[pos=.5,below] {$41$} (z21);
    \draw (d26) edge[->, line width=3] node[pos=.5,below] {$42$} (e26);
    \draw (e26) edge[->,p1Red] (f26);
    \draw (f26) edge[->,p1Red] (q22);

    \node[p0] (top) at (32,12) {$\top$};
    \node[p1] (b3) at (32,18) {$\begin{array}{c}
        b_3\\
         10 
    \end{array}$};

    \draw (b2) edge[->,p1Red] (b3);
    \draw (b3) edge[->,p1Red, bend left=10] (top);
    \draw (q21) edge[->,p1Red] (top);
    \draw (q22) edge[->,p1Red] (b3);
    
    \draw (top) edge[loop right, line width=3] (top);

    \node[p0] (z02) at (-18,9) {$3$};
    \node[p0] (z01) at (-18,18) {$3$};
    \node[p1] (q01) at (-9,9) {$4$};
    \node[p1] (q02) at (-9,18) {$4$};
    \node[p0, minimum size=7mm] (d01) at (-20,14) {};
    \node[p1,, minimum size=7mm] (e01) at (-20,15.5) {$3$};
    \node[p1,, minimum size=7mm] (f01) at (-20,17) {$1$};

    \node[p0, minimum size=7mm] (d02) at (-16,16) {};
    \node[p1,, minimum size=7mm] (e02) at (-16,14.5) {$3$};
    \node[p1,, minimum size=7mm] (f02) at (-16,13) {$1$};

    \node[p0, minimum size=7mm] (d03) at (-15,10) {};
    \node[p1,, minimum size=7mm] (e03) at (-13,10) {$3$};
    \node[p1,, minimum size=7mm] (f03) at (-11,10) {$1$};
    \node[p0, minimum size=7mm] (d04) at (-15,20) {};
    \node[p1,, minimum size=7mm] (e04) at (-13,20) {$3$};
    \node[p1,, minimum size=7mm] (f04) at (-11,20) {$1$};

    \draw (z01) edge[->,bend left=8] node[pos=.5,right] {$1$} (z02);
    \draw (z02) edge[->, line width=3] node[pos=.5,above] {$5$}(q01);
    \draw (z01) edge[->, line width=3] node[pos=.5,below] {$7$} (q02);
    \draw (z02) edge[->, bend left=8] node[pos=.5,left] {$2$} (z01);
    \draw (q01) edge[->,bend right=20, p1Red] (z01);
    \draw (q02) edge[->,bend left=20, p1Red] (z02);

    \draw (d01) edge[->] node[pos=.2,right] {$6$} (z02);
    \draw (d01) edge[->, line width=3] node[pos=.5,right] {$11$} (e01);
    \draw (e01) edge[->,p1Red] (f01);
    \draw (f01) edge[->,p1Red] (z01);

    \draw (d02) edge[->] node[pos=.2,right] {$4$} (z01);
    \draw (d02) edge[->, line width=3] node[pos=.5,right] {$10$} (e02);
    \draw (e02) edge[->,p1Red] (f02);
    \draw (f02) edge[->,p1Red] (z02);

    \draw (d03) edge[->] node[pos=.5,below] {$8$} (z02);
    \draw (d03) edge[->, line width=3] node[pos=.5,below] {$9$} (e03);
    \draw (e03) edge[->,p1Red] (f03);
    \draw (f03) edge[->,p1Red] (q01);

    \draw (d04) edge[->] node[pos=.5,above] {$3$} (z01);
    \draw (d04) edge[->, line width=3] node[pos=.5,above] {$12$} (e04);
    \draw (e04) edge[->,p1Red] (f04);
    \draw (f04) edge[->,p1Red] (q02);

    \draw (q01) edge[->,p1Red] (a1);
    \draw (q02) edge[->,p1Red] (b1);
\end{tikzpicture}